%% file: main.tex
\documentclass[11pt]{article}
\pdfoutput=1
\usepackage{etoolbox}

\usepackage{color}
\definecolor{DarkGreen}{rgb}{0.1,0.5,0.1}
\definecolor{DarkRed}{rgb}{0.5,0.1,0.1}
\definecolor{DarkBlue}{rgb}{0.1,0.1,0.5}
\definecolor{Gray}{rgb}{0.2,0.2,0.2}
\usepackage[pdftex]{hyperref}
\hypersetup{
    unicode=false,          
    pdftoolbar=true,        
    pdfmenubar=true,        
    pdffitwindow=false,      
    pdfnewwindow=true,      
    colorlinks=true,       
    linkcolor=violet,          
    citecolor=violet,        
    filecolor=DarkRed,      
    urlcolor=DarkBlue,          
    pdftitle={random-tensor-decomposition},
}
\usepackage{fullpage}

\usepackage[normalem]{ulem} 
\usepackage{braket}
\usepackage[english]{babel}
\usepackage[utf8]{inputenc}
\usepackage{amsmath,amsthm,amsopn,amsfonts,amssymb,mathtools,nicefrac}
\usepackage{mleftright}\mleftright
\usepackage{bm}
\usepackage{graphicx,epstopdf}
\usepackage[caption=false]{subfig}
\usepackage{graphicx}
\usepackage{mathrsfs}
\usepackage{enumitem}
\usepackage{algorithm}
\usepackage{comment}
\usepackage{algorithmic}
\usepackage[capitalize]{cleveref}
\usepackage{braket}
\usepackage{qcircuit}
\usepackage{blkarray}
\usepackage{thm-restate}
\usepackage{thmtools}
\usepackage{tikz}
\usepackage{tikz-cd}
\usetikzlibrary{matrix}
\usepackage{filecontents}
\usepackage{adjustbox}
\usepackage{lipsum}
\usepackage{ifthen}

\numberwithin{equation}{section}

\DeclareMathOperator*{\E}{{\mathbb{E}}}

\DeclareMathOperator{\nnz}{nnz}

\DeclareMathOperator{\rank}{rank}

\declaretheorem[numberwithin=section]{theorem}
\declaretheorem[sibling=theorem]{lemma}

\declaretheorem[style=definition]{definition}

\declaretheorem[numbered=no,style=definition,name=Fact]{fact*}

\newtheorem*{theorem*}{Theorem}
\newtheorem*{corollary*}{Corollary}
\newtheorem*{proposition*}{Proposition}
\newtheorem*{lemma*}{Lemma}
\newtheorem*{claim*}{Claim}
\newtheorem*{problem*}{Problem}

\newcommand{\BB}[1]{\mathbb{#1}}

\newcommand{\bigO}[1]{\mathcal{O}\left( #1 \right)}
\newcommand{\bigOt}[1]{\widetilde{\mathcal{O}}\left( #1 \right)}

\newcommand{\R}{\BB{R}}

\newcommand{\underflow}[2]{\underset{\kern-60mm \overbrace{#1} \kern-60mm}{#2}}

\newcommand{\tsr}[1]{\pmb{\mathcal{#1}}}

\newcommand{\vcr}[1]{\mathbf{#1}}
\newcommand{\mat}[1]{\mathbf{#1}}

\newcommand{\fnrm}[1]{{\| #1 \|}_F}

\newcommand{\inti}[2]{\{{#1},\ldots, {#2}\}}

\newcommand{\qtext}[1]{\quad\text{#1}\quad}

\usepackage{cite}
\usepackage{sectsty}
\usepackage[section]{placeins}
\usepackage[bottom=0.85in,top=0.85in,left=0.8in,right=0.8in,footskip=.3in]{geometry}
\sectionfont{\fontsize{14}{10}\selectfont}
\subsectionfont{\fontsize{12}{10}\selectfont}

\makeatletter
\patchcmd{\@maketitle}{\LARGE \@title}{\fontsize{16}{20}\selectfont\@title}{}{}
\makeatother

\title{
Fast and accurate randomized algorithms for low-rank tensor decompositions\thanks{
This work is supported by the US NSF OAC via award No.\ 1942995.}
}

\author{Linjian Ma\\
Department of Computer Science\\
University of Illinois at Urbana-Champaign\\
lma16@illinois.edu

\and Edgar Solomonik\\
Department of Computer Science\\
University of Illinois at Urbana-Champaign\\
solomon2@illinois.edu}

\date{}

\begin{document}
\maketitle

\input{contents/abstract}

\section{Introduction}
\input{material_arxiv/intro_arxiv}

\section{Background}
\label{sec:background}
\input{material_arxiv/bg_arxiv}

\section{Background on Sketching}
\label{sec:sketch}
\input{contents/sketching}

\section{Sketched Rank-constrained Linear Least Squares for Tucker-ALS}
\label{sec:sketch-rcls}
\input{contents/ls}

\section{Initialization of Factor Matrices via Randomized Range Finder}
\label{sec:init}

\input{contents/init}

\section{Cost Analysis}
\label{sec:cost}
\input{material_arxiv/cost_arxiv}

\section{Experiments}
\label{sec:exp}
\input{material_arxiv/exp_arxiv}

\section{Conclusions}
\label{sec:conclu}
\input{contents/conclusion}


\small
\bibliographystyle{abbrv}
\bibliography{main}


\normalsize
\appendix
\section{Detailed Proofs for \cref{sec:sketch-rcls}}\label{appendix:detailed_proofs}

\input{contents/rc_least_squares}

\input{contents/tensorsketch}

\input{contents/leverage}

\section{TensorSketch for General Constrained Least Squares}
\input{contents/general_constrain_tensorsketch}\label{appendix:general_ts}


\end{document}

%% file: contents/abstract.tex
\begin{abstract}
Low-rank Tucker and CP tensor decompositions are powerful tools in data analytics. The widely used alternating least squares (ALS) method, which solves a sequence of over-determined least squares subproblems, is costly for large and sparse tensors. We propose a fast and accurate
sketched ALS algorithm for Tucker decomposition, which 
solves a sequence of sketched rank-constrained linear least squares subproblems. 
Theoretical sketch size upper bounds are provided to achieve $\bigO{\epsilon}$ relative error for each subproblem with two sketching techniques, TensorSketch and leverage score sampling.
Experimental results show that this new ALS algorithm, combined with a new initialization scheme based on randomized range finder, yields up to $22.0\%$ relative decomposition residual improvement compared to the state-of-the-art sketched
randomized algorithm for Tucker decomposition of various synthetic and real datasets. 
This Tucker-ALS algorithm is further used to accelerate CP decomposition, 
by using randomized Tucker compression followed by CP decomposition of the Tucker core tensor.
Experimental results show that this algorithm not only converges faster, but also yields more accurate CP decompositions. 
\end{abstract}

%% file: material_arxiv/intro_arxiv.tex
Tensor decompositions~\cite{kolda2009tensor} are general
tools for compression and approximation of high dimensional data, and are widely used in both scientific computing~\cite{pazner2018approximate,hohenstein2012tensor,hummel2017low} and machine learning~\cite{anandkumar2014tensor,sidiropoulos2017tensor}. 
In this paper, we focus on Tucker decomposition~\cite{tucker1966some} and CANDECOMP/PARAFAC (CP) decomposition~\cite{hitchcock1927expression,harshman1970foundations}. Both decompositions are higher order generalizations of the matrix singular value decomposition (SVD). Tucker decomposition decomposes the input tensor
into a core tensor along with factor matrices composed of orthonormal columns. CP decomposition decomposes the input tensor into a sum of outer products of vectors, which may also be assembled into factor matrices. 

The alternating least squares
(ALS) method is most widely used to compute both CP and Tucker decompositions. 
The ALS algorithm consists of \textit{sweeps}, and each sweep updates every factor matrix once in a fixed order. 
The ALS method for Tucker decomposition, called the \textit{higher-order orthogonal iteration} (HOOI)~\cite{andersson1998improving,de2000best,kolda2009tensor}, updates one of the factor matrices along with the core tensor at a time. 
Similarly, each update procedure in the ALS algorithm for CP decomposition (CP-ALS) updates one of the factor matrices.
For both decompositions, each optimization subproblem guarantees decrease of the decomposition residual.

\input{contents/intro/intro}

Recent works have applied
different randomized techniques to accelerate both CP and Tucker decompositions. 
Based on sketching techniques introduced for low-rank matrix approximation~\cite{woodruff2014sketching},
several sketching
algorithms have been developed for CP decomposition~\cite{larsen2020practical,aggour2020adaptive,cheng2016spals,zhou2014decomposition}.
These algorithms accelerate CP-ALS by approximating
the original highly over-determined linear least squares problems by smaller sketched linear least squares problems.
These techniques have been shown to greatly improve the efficiency compared to the original ALS. 
For Tucker decomposition, several randomized algorithms based on random projection~\cite{che2021randomized,ahmadi2020randomized,che2019randomized,minster2020randomized,zhou2014decomposition,sun2018tensor} have been developed to accelerate \textit{higher-order singular value decomposition} (HOSVD)~\cite{de2000multilinear,tucker1966some}. However, these methods calculate the core tensor via TTMc among the input tensor and all the factor matrices, which 
incurs a cost of $\Omega(\nnz(\tsr{T})R)$ for sparse tensors.

Becker and Malik~\cite{malik2018low} introduced a sketched
ALS algorithm for Tucker decomposition, which avoids the expensive cost of
TTMc. 
Unlike the traditional HOOI, 
each sweep of this ALS scheme contains $N+1$ subproblems, where only one of the factor matrices or the core tensor is updated in each subproblem. 
This scheme is easier to analyze theoretically, since each subproblem is an \textit{unconstrained} linear least squares problem, which can be efficiently solved via sketching. However, the scheme produces decompositions that are generally less effective 
than HOOI.

\input{contents/intro/contribution}

\subsection{Organization}
\cref{sec:background} introduces notation used throughout the paper, ALS algorithms for Tucker and CP decompositions, and previous work. \cref{sec:sketch} introduces TensorSketch and leverage score sampling techniques. In \cref{sec:sketch-rcls}, we provide sketch size upper bounds for the  rank-constrained linear least squares subproblems in sketched Tucker-ALS. Detailed proofs are presented in \cref{appendix:detailed_proofs}. In \cref{sec:init}, we introduce the initialization scheme for factor matrices via randomized range finder. In \cref{sec:cost}, we present detailed cost analysis for our new sketched ALS algorithms. We present experimental results in \cref{sec:exp} and conclude in \cref{sec:conclu}.

%% file: contents/intro/intro.tex
In this work, we consider decomposition of
order $N$ tensors ($\tsr{T}$) that are large in dimension size ($s$) and can be potentially sparse.
We focus on the problem of computing low-rank (with target rank $R \ll s$ and $R \ll \nnz(\tsr{T})$) decompositions for such tensors
via ALS, which is often used for extracting principal component information from
large-scale datasets. For Tucker decomposition, ALS is bottlenecked by the operation called
\textit{the tensor times matrix-chain} (TTMc).
For CP decomposition, ALS is bottlenecked by the operation called \textit{the matricized tensor-times Khatri-Rao product} (MTTKRP). 
Both TTMc and MTTKRP have a per-sweep cost of $\Omega(\nnz(\tsr{T})R)$ \cite{smith2015splatt}.
Consequently, the per-sweep costs of both HOOI and CP-ALS are proportional to the number of nonzeros in the tensor, which are expensive for large tensors with billions of nonzeros.

%% file: contents/intro/contribution.tex
\subsection{Our Contributions}
In this work, we propose a new sketched
ALS algorithm for Tucker decomposition. Different from Becker and Malik~\cite{malik2018low}, our ALS scheme is the same as HOOI, where one of the factor matrices along with the core tensor are updated in each subproblem. This guarantees the algorithm can reach the same accuracy as HOOI with sufficient sketch size. Experimental results show that it provides more accurate results compared to those in \cite{malik2018low}.

In this scheme, each subproblem is a sketched \textit{rank-constrained} linear least squares problem, with the left-hand-side matrix with size $s^{N-1}\times R^{N-1}$  composed of orthonormal columns.
To the best of our knowledge, the relative error analysis of sketching techniques for this problem have not been discussed in the literature. Existing works either only provide sketch size upper bounds for the relaxed problem~\cite{pilanci2016iterative}, where rank constraint is relaxed with a nuclear norm constraint, or provide upper bounds for general constrained problems, which are too loose~\cite{woodruff2014sketching}.
We provide tighter sketch size upper bounds to achieve $\bigO{\epsilon}$ relative error with two state-of-the-art sketching techniques, TensorSketch~\cite{pagh2013compressed} and leverage score sampling~\cite{drineas2012fast}. 

With leverage score sampling, our analysis shows that 
 with probability at least $1-\delta$,
the sketch size of $\bigO{R^{N-1}/(\epsilon^2\delta) }$ is sufficient for results with $\bigO{\epsilon}$-relative error. With TensorSketch, the sketch size upper bound is $\bigO{(R^{N-1} \cdot 3^{N-1})/\delta \cdot (R^{N-1}  + 1/\epsilon^2)}$, at least $\bigO{3^{N-1}}$ times that for leverage score sampling. For both techniques, our bounds are at most $\bigO{1/\epsilon}$ times the sketch size upper bounds for the unconstrained linear least squares problem.

The upper bounds suggest that under the same accuracy criteria, leverage score sampling potentially needs smaller sketch size for each linear least squares problem and thus can be more efficient than TensorSketch. Therefore, with the same sketch size, the accuracy with leverage score sampling can be better. However, with the standard random initializations for factor matrices, leverage score sampling can perform poorly on tensors with high coherence~\cite{candes2009exact} (the orthogonal basis for the row space of each matricization of the input tensor has large row norm variability), making it less robust than TensorSketch.
To improve the robustness of leverage score sampling, we introduce an algorithm that uses the randomized range finder (RRF)~\cite{halko2011finding} to initialize the factor matrices. 
The initialization scheme uses the composition of CountSketch and Gaussian random matrix as the RRF embedding matrix, which
only requires one pass over the non-zero elements of the input tensor. Our experimental results show that the leverage score sampling based randomized algorithm combined with this RRF scheme performs well on tensors with high coherence. 

For $R\ll s$, our new sketching based algorithm for Tucker decomposition can also be used to accelerate CP decomposition. Tucker compression is performed first, and then CP decomposition is applied to the core tensor~\cite{zhou2014decomposition,bro1998improving,erichson2020randomized}.
Since the per-sweep costs for both sketched Tucker-ALS and sketched CP-ALS are comparable, and Tucker-ALS often needs much fewer sweeps than CP-ALS (Tucker-ALS typically converges in less than 5 sweeps based on our experiments), this Tucker + CP method can be more efficient than directly applying randomized CP decomposition~\cite{larsen2020practical,cheng2016spals} on the input tensor.

In summary, this paper makes the following contributions.
\begin{itemize}[itemsep=0pt,leftmargin=*]
    \item We introduce a new sketched ALS algorithm for Tucker decomposition, which contains a sequence of sketched rank-constrained linear least squares subproblems. We provide theoretical upper bounds for the sketch size, and experimental results show that the algorithm provides up to $22.0\%$ relative decomposition residual improvement compared to the previous work.
    \item We provide detailed comparison of TensorSketch and leverage score sampling in terms of efficiency and accuracy. Our theoretical analysis shows that leverage score sampling is better in terms of both metrics.
    \item We propose an initialization scheme based on RRF that improves the accuracy of leverage score sampling based sketching algorithm on tensors with high coherence.
    \item We show that CP decomposition can be more efficiently and accurately calculated based on the sketched Tucker + CP method, compared to directly performing sketched CP-ALS on the input tensor.
\end{itemize}

%% file: material_arxiv/bg_arxiv.tex
\input{contents/bg/bg_notation}

\subsection{Tucker Decomposition with ALS}
\label{sec:bg-tucker}
\input{contents/bg/bg_tucker}

\input{contents/bg/bg_tucker_alg}
\subsection{CP Decomposition with ALS}
\label{sec:bg-cp}
\input{contents/bg/bg_cp}

\input{contents/bg/bg_cp_alg}
\input{contents/bg/bg_previous}

%% file: contents/bg/bg_notation.tex
\subsection{Notation}
\label{sec:notations}

Our analysis makes use of tensor algebra in both element-wise equations and specialized notation for tensor operations~\cite{kolda2009tensor}.
Vectors are denoted with bold lowercase Roman letters (e.g., $\vcr{v}$), matrices are denoted with bold uppercase Roman letters (e.g., $\mat{M}$), and tensors are denoted with bold calligraphic font (e.g., $\tsr{T}$). 
An order $N$ tensor corresponds to an $N$-dimensional array. 
Elements of vectors, matrices, and tensors are denoted in parentheses, e.g., $\vcr{v}(i)$ for a vector $\vcr{v}$, $\mat{M}(i,j)$ for a matrix $\mat{M}$, and $\tsr{T}(i,j,k,l)$ for an order 4 tensor $\tsr{T}$.
The $i$th column of $\mat{M}$ is denoted by $\mat{M}(:,i)$, and the $i$th row is denoted by $\mat{M}(i,:)$. Parenthesized superscripts are used to label different vectors, matrices and tensors (e.g. $\tsr{T}^{(1)}$ and $\tsr{T}^{(2)}$ are unrelated tensors). Number of nonzeros of the tensor $\tsr{T}$ is denoted by $\nnz(\tsr{T})$.

The pseudo-inverse of matrix $\mat{A}$ is denoted with $\mat{A}^\dagger$. The Hadamard product of two matrices is denoted with $*$. The outer product of two or more vectors is denoted with $\circ$. The Kronecker product of two vectors/matrices is denoted with $\otimes$.
For matrices $\mat{A}\in \mathbb{R}^{m\times k}$ and $\mat{B}\in \mathbb{R}^{n\times k}$, their Khatri-Rao product results in a matrix of size $(mn)\times k$ defined by
$
   \mat{A}\odot \mat{B} = [\mat{A}(:,1)\otimes \mat{B}(:,1),\ldots, \mat{A}(:,k)\otimes \mat{B}(:,k)] .
$
The mode-$n$ tensor times matrix of an order $N$ tensor $\tsr{T} \in \mathbb{R}^{s_1\times \cdots \times s_N}$ with a matrix $ \textbf{A}\in \mathbb{R}^{J\times s_n}$ is denoted by $\tsr{T}\times_n \mat{A}$, whose output size is $s_1\times\cdots\times s_{n-1}\times J\times s_{n+1}\times\cdots\times s_N$.
Matricization is the process of unfolding a tensor into a matrix. The mode-$n$ matricized version of $\tsr{T}$ is denoted by $\textbf{T}_{(n)}\in \mathbb{R}^{s_n\times K}$ where $K=\prod_{m=1,m\neq n}^N s_m$. 
We use $[N]$ to denote $\{1,\ldots, N\}$. $\widetilde{\mathcal{O}}$ denotes the asymptotic cost with logarithmic factors ignored.

%% file: contents/bg/bg_tucker.tex
  We review the ALS method for computing Tucker decomposition of an input tensor~\cite{tucker1966some}.
Throughout the analysis we assume the input tensor has order $N$ and size $s \times \cdots \times s$, and the Tucker ranks are $R\times \cdots \times R$.
Tucker decomposition approximates a tensor by a core tensor contracted along each mode with matrices that have orthonormal columns. The goal of Tucker decomposition is to minimize the objective function,
  \begin{equation}
  f(\tsr{C}, \mat{A}^{(1)}, \ldots, \mat{A}^{(N)}) := \frac{1}{2} 
  \left\|
  \tsr{T} - \tsr{C}\times_1\mat{A}^{(1)}\times_2\mat{A}^{(2)}\cdots\times_N\mat{A}^{(N)}
  \right\|_F^2.      
  \end{equation}
The core tensor $\tsr{C}$ is of order $N$ with dimensions $R\times \cdots \times R$.
Each matrix $\textbf{A}^{(n)} \in \mathbb{R}^{s\times R}$ for $n\in[N]$ has orthonormal columns.
The ALS method for Tucker decomposition~\cite{andersson1998improving,de2000best,kolda2009tensor}, called the \textit{higher-order orthogonal iteration} (HOOI), proceeds by updating one of the factor matrices along with the core tensor at a time. The $n$th subproblem can be written as
\begin{equation}
    \min_{\tsr{C}, \mat{A}^{(n)}} \frac{1}{2} \left\|
    \mat{P}^{(n)} \mat{C}_{(n)}^T\mat{A}^{(n)T} - \mat{T}_{(n)}^T
    \right\|_F^2,
\end{equation}
where $\mat{P}^{(n)} = \mat{A}^{(1)} \otimes  \cdots \otimes \mat{A}^{(n-1)} \otimes \mat{A}^{(n+1)}\otimes  \cdots \otimes \mat{A}^{(N)}$.
This problem can be formulated as a rank-constrained linear least squares problem,
\begin{equation}\label{eq:tucker_ls}
    \min_{\mat{B}^{(n)}} \frac{1}{2} \left\|
    \mat{P}^{(n)} \mat{B}^{(n)} - \mat{T}_{(n)}^T
    \right\|_F^2, \quad \text{such that} \quad \text{rank}(\mat{B}^{(n)})\leq R.
\end{equation}
$\mat{A}^{(n)}$ corresponds to the right singular vectors of the optimal $\mat{B}^{(n)}$, while $\mat{C}_{(n)}^T = \mat{B}^{(n)}\mat{A}^{(n)}$. Since $\mat{P}^{(n)}$ contains orthonormal columns, the optimal $\mat{B}^{(n)}$ can be obtained by calculating the \textit{Tensor Times Matrix-chain} (TTMc),
\begin{equation}
\label{eq:Yn}
    \tsr{Y}^{(n)}=\tsr{T}\times_1\mat{A}^{(1)T}\cdots\times_{n-1}\mat{A}^{(n-1)T}\times_{n+1}\mat{A}^{(n+1)T}\cdots\times_{N}\mat{A}^{(N)T},
\end{equation}
and taking $\mat{B}^{(n)}$ to be the transpose of the mode-$n$ matricized $\tsr{Y}^{(n)}$, $\mat{Y}_{(n)}^{(n)T}$. Calculating $\tsr{Y}^{(n)}$ costs $\bigO{s^NR}$ for dense tensors and $\bigO{\nnz(\tsr{T})R^{N-1}}$ for sparse tensors. HOOI typically converges in less than 5 sweeps.
Before the HOOI procedure, the factor matrices are often initialized with
the \textit{higher-order singular value decomposition} (HOSVD)~\cite{de2000multilinear,tucker1966some}. HOSVD computes the  truncated SVD of each $\mat{T}_{(n)}\approx\mat{U}^{(n)}\mat{\Sigma}^{(n)}\mat{V}^{(n)T}$, and sets $\mat{A}^{(n)} = \mat{U}^{(n)}$ for $n\in[N]$. 
If performing SVD via randomized SVD~\cite{halko2011finding}, updating $\mat{A}^{(n)}$ for each mode costs $\bigO{s^NR}$ for dense tensors, and costs $\bigO{s^{N-1}R^2 + \nnz(\tsr{T})R}$ for sparse tensors.
The Tucker-ALS algorithm is presented in Algorithm~\ref{tuckerals}.

%% file: contents/bg/bg_tucker_alg.tex
\begin{algorithm}
\caption{\textbf{Tucker-ALS}: ALS procedure for Tucker decomposition}\label{tuckerals}
\begin{algorithmic}[1]
\label{alg:tucker-als}
\STATE{\textbf{Input: }Tensor $\tsr{T}\in\R^{s_1 \times \cdots\times s_N}$, 
decomposition ranks $\{R_1,\ldots,R_N\}$}
\STATE{Initialize $\{\mat{A}^{(1)}, \ldots , \mat{A}^{(N)} \}$ using HOSVD}
\WHILE{not converged} {
	\FOR{{$n\in [N] $}} {
  \STATE{Update $ \tsr{Y}^{(n)}$ based on \eqref{eq:Yn}}
  \STATE{$   \textbf{A}^{(n)} \leftarrow R_n$  leading left singular vectors of $\textbf{Y}^{(n)}_{(n)} $} 
	}\ENDFOR
  \STATE{$\tsr{C}\leftarrow
  \tsr{Y}^{(N)}\times_{N}\mat{A}^{(N)T}$}
}
\ENDWHILE
\RETURN $\{\tsr{C}, \mat{A}^{(1)}, \ldots , \mat{A}^{(N)} \}$ 
\end{algorithmic}
\end{algorithm}

%% file: contents/bg/bg_cp.tex
CP tensor decomposition~\cite{hitchcock1927expression,harshman1970foundations} decomposes the input tensor into a sum of outer products of vectors. 
Throughout analysis we assume the input tensor has order $N$ and size $s \times \cdots \times s$, and the CP rank is $R$.
The goal of CP decomposition is to minimize the objective function,
  \begin{equation}
  f(\mat{A}^{(1)}, \ldots, \mat{A}^{(N)}) := \frac{1}{2} 
  \left\|
  \tsr{T} - \sum_{r=1}^{R} \mat{A}^{(1)}(:,r)\circ \cdots \circ \mat{A}^{(N)}(:,r)
  \right\|_F^2,      
  \end{equation} where $\mat{A}^{(i)}\in\R^{s\times R}$ for $i\in[N]$ are called factor matrices.
CP-ALS is the mostly widely used algorithm to get the factor matrices. 
In each ALS sweep, we solve $N$ subproblems, and the objective for the update of $\mat{A}^{(n)}$, with $n\in[N]$, is expressed as,
\begin{equation}
\label{eq:subproblem}
    \mat{A}^{(n)} = \arg\min_{\mat{A}}\frac{1}{2}\left\|\mat{P}^{(n)}\mat{A}{}^T - \mat{X}_{(n)}^T\right\|^2_F,
\end{equation}
 where 
$\mat{P}^{(n)} = \mat{A}^{(1)} \odot \cdots \odot  \mat{A}^{(n-1)}  \odot  \mat{A}^{(n+1)} \odot \cdots \odot \mat{A}^{(N)}.$
 Solving the linear least squares problem above has a cost of $\bigO{s^NR}$. For instance, when solving via normal equations the term $\mat{P}^{(n)T}\mat{X}_{(n)}^T$, which is called MTTKRP, needs to be calculated, and it costs $\bigO{s^NR}$ for dense tensors and $\bigO{\nnz(\tsr{T})R}$ for sparse tensors.
We show the CP-ALS algorithm in~\cref{alg:cp_als}.

A major disadvantage of
CP-ALS is its slow convergence. There are many cases where CP-ALS 
takes a large number of sweeps to converge when high resolution is necessary~\cite{mitchell1994slowly}.
When $R < s$, the procedure can be accelerated by performing Tucker-ALS first, which typically converges in fewer sweeps, and then computing a CP decomposition of the core tensor~\cite{carroll1980candelinc,zhou2014decomposition,bro1998improving}, which only has $\bigO{R^N}$ elements.

%% file: contents/bg/bg_cp_alg.tex
\begin{algorithm}
    \caption{\textbf{CP-ALS}: ALS procedure for CP decomposition}
\label{alg:cp_als}
\begin{algorithmic}[1]
\STATE{\textbf{Input: }Tensor $\tsr{T}\in\mathbb{R}^{s_1\times\cdots \times s_N}$, rank $R$
}
\STATE{Initialize $\{ \textbf{A}^{(1)}, \ldots , \textbf{A}^{(N)} \}$ as uniformly distributed random matrices within $[0,1]$}
\WHILE{not converged}
\FOR{\texttt{$n\in[N] $}}
\STATE{Update $ \mat{A}^{(n)}$ via solving~\cref{eq:subproblem}}
\ENDFOR
\ENDWHILE
\RETURN $\{ \textbf{A}^{(1)}, \ldots , \textbf{A}^{(N)} \}$
\end{algorithmic}
\end{algorithm}

%% file: contents/bg/bg_previous.tex
\subsection{Previous Work}

Randomized algorithms have been applied to both Tucker and CP decompositions in several previous works. For Tucker decomposition, Ahmadi-Asl et al. \cite{ahmadi2020randomized} review a variety of random projection, sampling and sketching based randomized algorithms. 
Methods introduced in \cite{che2021randomized,che2021randomized,che2019randomized,zhou2014decomposition,sun2018tensor} accelerate the traditional HOSVD/HOOI via random projection, where factor matrices are updated based on performing SVD on the matricization of the randomly projected input tensor.
For these methods, random projections are all performed based on Gaussian embedding matrices, 
and the core tensor is calculated via TTMc among the input tensor and all the factor matrices, which costs $\Omega(\nnz(\tsr{T})R)$ and is computationally inefficient for large sparse tensors.
Sun et al. \cite{sun2020low} introduce randomized algorithms for Tucker decompositions for streaming data.

The most similar work to ours is Becker and Malik~\cite{malik2018low}. 
This work computes Tucker decomposition via a sketched ALS scheme where in each optimization subproblem, one of the factor matrices or the core tensor is updated. They also solve each sketched linear least squares subproblem via TensorSketch.
Our new scheme provides more accurate results compared to this method. 
Another work that is closely relevant to us is \cite{minster2020randomized}. This work introduces structure-preserving decomposition, which is similar to Tucker decomposition but the factor matrices are not necessary orthogonal, and the entries of the core tensor are explicitly taken from the original tensor. The authors design an algorithm based on rank-revealing QR~\cite{gu1996efficient}, which is efficient for sparse tensors, to calculate the decomposition. However, their experimental results show that the relative error of the algorithm for sparse tensors is much worse than that of the traditional HOSVD~\cite{minster2020randomized}.

Several works discuss algorithms for sparse Tucker decomposition. Oh et al.~\cite{oh2018scalable} propose PTucker, which provides algorithms for scalable sparse Tucker decomposition. Kaya and Ucar~\cite{kaya2016high} provide parallel algorithms for sparse Tucker decompositions. Li et al.~\cite{li2020sgd_tucker} introduce SGD-Tucker, which uses stochastic gradient descent to perform Tucker decomposition of sparse tensors.

For CP decomposition, Battaglino et al.~\cite{battaglino2018practical} and Jin et al.~\cite{jin2019faster} introduce a randomized algorithm based on Kronecker fast Johnson-Lindenstrauss Transform (KFJLT) to accelerate CP-ALS. However, KFJLT is effective only for the decomposition of dense tensors.
Aggour et al.~\cite{aggour2020adaptive} introduce adaptive sketching for CP decomposition.
Song et al.~\cite{song2019relative} discuss the theoretical relative error of various tensor decompositions based on sketching. 
The work by Cheng et al.~\cite{cheng2016spals} and Larsen and Kolda~\cite{larsen2020practical} accelerate CP-ALS based on leverage score sampling. 
Cheng et al.~\cite{cheng2016spals} use leverage score sampling to accelerate MTTKRP calculations.
Larsen and Kolda~\cite{larsen2020practical} propose an approximate leverage score sampling scheme for the Khatri-Rao product, and they show with $\bigO{R^{(N-1)}\log(1/\delta)/\epsilon^2 }$ number of samples, each unconstrained linear least squares subproblem in CP-ALS can be solved with $\bigO{\epsilon}$-relative error with probability at least $1-\delta$. 
Zhou et al.~\cite{zhou2014decomposition} and Erichson et al.~\cite{erichson2020randomized} accelerate CP decomposition via performing randomized Tucker decomposition of the input tensor first, and then performing CP decomposition of the smaller core tensor.

Several other works discuss techniques to parallelize and accelerate CP-ALS. Ma and Solomonik \cite{ma2018accelerating,ma2020efficient} approximate MTTKRP within CP-ALS based on information from previous sweeps.
For sparse tensors, parallelization strategies for MTTKRP have been developed both on shared memory systems~\cite{nisa2019load,smith2015splatt} and distributed memory systems~\cite{li2017model,smith2016medium,kaya2018parallel}.

%% file: contents/sketching.tex
Throughout the paper we consider the linear least squares problem,
\begin{equation}\label{eq:ls}
    \min_{\mat{X}\in\mathcal{C}} \frac{1}{2} \left\|
    \mat{P} \mat{X} - \mat{B}
    \right\|_F^2,
\end{equation}
where $\mat{P} = \mat{A}^{(1)} \otimes  \cdots \otimes \mat{A}^{(N)} \in \R^{s^{N} \times R^{N}}$ is a chain of Kronecker products, $N\geq 2$, $\mat{P}$ is dense and $\mat{B}$ is sparse. In each subproblem of Tucker HOOI, 
the feasible region $\mathcal{C}$ contains matrices with the rank constraint, as is shown in \eqref{eq:tucker_ls}. The associated sketched problem is
\begin{equation}
    \min_{\mat{X}\in\mathcal{C}} \frac{1}{2} \left\|
    \mat{S}\mat{P} \mat{X} - \mat{S}\mat{B}
    \right\|_F^2,
\end{equation}
where $\mat{S}\in \R^{m \times s^{N}}$ is the sketching matrix with $m \ll s^{N}$. We refer to $m$ as the sketch size throughout the paper.

The Kronecker product structure of $\mat{P}$ prevents efficient application of widely-used sketching matrices, including Gaussian matrices and CountSketch matrices. For these sketching matrices, the computation of $\mat{S}\mat{P}$ requires forming $\mat{P}$ explicitly, which has a cost of $\bigO{s^{N}R^{N}}$. We consider two sketching techniques, TensorSketch and leverage score sampling, that are efficient for the problem.
With these two sketching techniques, $\mat{S}\mat{P}$ can be calculated without explicitly forming $\mat{P}$, and $\mat{S}\mat{B}$ can be calculated efficiently as well (with a cost of $\bigO{\nnz(\mat{B})}$).

\subsection{TensorSketch}

TensorSketch is a special type of CountSketch, where the hash map is restricted to a specific format to allow fast multiplication of the sketching matrix with the chain of Kronecker products. We introduce the definition of CountSketch and TensorSketch below.

\begin{definition}[CountSketch]\label{def:countsketch}
The CountSketch matrix is defined as $\mat{S} = \mat{\Omega D}\in \R^{m\times n}$, where
\begin{itemize}[itemsep=0pt,leftmargin=*]
    \item $h : [n] \rightarrow [m]$ is a hash map such that $\forall i \in [n]$ and $\forall j \in [m]$, $\Pr[h(i) = j] = 1/m$,
    \item $\mat{\Omega} \in \R^{m\times n}$ is a matrix with $\mat{\Omega}(j,i) = 1$ if $j=h(i)$ $\forall i\in[n]$ and $\mat{\Omega}(j,i) = 0$ otherwise,
    \item $\mat{D} \in \R^{n\times n}$ is a diagonal matrix whose diagonal is a Rademacher vector (each entry is $+1$ or $-1$ with equal probability).
\end{itemize}
\end{definition}

\begin{definition}[TensorSketch~\cite{pagh2013compressed}]
The order $N$ TensorSketch matrix $\mat{S}= \mat{\Omega D}\in \R^{m \times \prod_{i=1}^N s_i}$ is defined based on two hash maps $H$ and $S$ defined below,
\begin{equation}\label{eq:h_ts}
    H : [s_1] \times [s_2] \times \cdots \times [s_N] \rightarrow [m] : (i_1, \ldots , i_N ) \mapsto
\left(
\sum_{n=1}^N
(H_n(i_n) - 1) \mod m
\right) + 1,
\end{equation}
\vspace{-3mm}
\begin{equation}\label{eq:s_ts}
    S : [s_1] \times [s_2] \times \cdots \times [s_N] \rightarrow \{-1, 1\} : (i_1, \ldots , i_N ) \mapsto
    \prod_{n=1}^N S_n(i_n),
\end{equation}
where each $H_n$ for $n\in[N]$ is a 3-wise independent hash map that maps $[s_n]\rightarrow[m]$, and each $S_n$ is a 4-wise independent hash map that maps $[s_n]\rightarrow\{-1,1\}$.
A hash map is $k$-wise independent if any designated $k$ keys are independent random variables.
Two matrices $\mat{\Omega}$ and $\mat{D}$ are defined based on $H$ and $S$, respectively,
\begin{itemize}[itemsep=0pt,leftmargin=*]
    \item $\mat{\Omega} \in \R^{m\times \prod_{i=1}^N s_i}$ is a matrix with $\mat{\Omega}(j,i) = 1$ if $j=H(i)$ $\forall i\in\left[\prod_{i=1}^N s_i\right]$, and $\mat{\Omega}(j,i) = 0$ otherwise,
    \item $\mat{D} \in \R^{n\times n}$ is a diagonal matrix with $\mat{D}(i,i) = S(i)$.
\end{itemize}
Above we use the notation $S(i) = S(i_1,\dots,i_N)$ where $i = i_1 + \sum_{k=2}^N \left(\prod_{\ell=1}^{k-1} s_l\right) \left(i_k-1\right)$, and similar for $H$.
\end{definition}

The restricted hash maps \eqref{eq:h_ts},\eqref{eq:s_ts} used in $\mat{S}$ make it efficient to multiply with a chain of Kronecker products. 
Define $\mat{S}^{(n)}:= \mat{\Omega}^{(n)}\mat{D}^{(n)} \in \R^{m \times s_n}$, where $\mat{\Omega}^{(n)} \in \R^{m \times s_n}$ is defined based on $H_n$ and $\mat{D}^{(n)} \in \R^{s_n \times s_n}$ defined based on $S_n$, and
let $\mat{P} = \mat{A}^{(1)}\otimes \mat{A}^{(2)}\otimes \cdots \otimes \mat{A}^{(N)}$ with $\mat{A}^{(n)}\in\R^{s_n \times R_n}$ for $n\in[N]$, 
\begin{equation}
    \mat{SP}  = \text{FFT}^{-1}\left(
    \left(
    \bigodot_{n=1}^N \left(
    \text{FFT}\left(\mat{S}^{(n)}\mat{A}^{(n)}\right)
    \right)^T
    \right)^T
    \right).
\end{equation}
Calculating each $\text{FFT}\left(\mat{S}^{(n)}\mat{A}^{(n)}\right)$ costs $\bigO{s_nR_n + m\log m R_n}$, and performing the Kronecker product as well as the outer FFT costs $\bigO{m\log m\prod_{n=1}^NR_n}$. When each $s_n=s$ and $R_n=R$, the overall cost is $\bigO{NsR + m\log mR^N}$.

\subsection{Leverage Score Sampling}

Leverage score sampling is a useful tool to pick important rows to form the sampled/sketched linear least squares problem. 
Intuitively, let $\mat{Q}_P$ be an orthogonal basis for the column space of $\mat{P}$. Then large-norm rows of $\mat{Q}_P$ suggest large contribution to $\mat{Q}_P^T\mat{B}$, which is part of the linear least squares right-hand-side we can solve for.

\begin{definition}[Leverage Scores \cite{drineas2012fast,mahoney2011randomized}]\label{def:leverages_scores}
  Let $\mat{P} \in \R^{s \times R}$ with $s > R$, and
  let $\mat{Q} \in \R^{s \times R}$ be any orthogonal basis for the column space of $\mat{P}$.
  The \emph{leverage scores} of the rows of $\mat{P}$ are given by
  \begin{displaymath}
    \ell_i(\mat{P}) := (\mat{QQ}^T)(i,i) = \|\mat{Q}(i,:)\|_2^2 \qtext{for all} i \in [s].
  \end{displaymath}
\end{definition}

\begin{definition}[Importance Sampling based on Leverage Scores]\label{def:leverage_sampling} Let $\mat{P}\in\R^{s\times R}$ be a full-rank matrix and $s>R$.
The leverage score sampling matrix of $\mat{P}$ is defined as $\mat{S} = \mat{D\Omega}$, where $\mat{\Omega}\in\R^{m\times s}$, $m<s$ is the sampling matrix, and $\mat{D}\in\R^{m \times m}$ is the rescaling matrix.
For each row $j\in[m]$ of $\mat{\Omega}$,  one column index $i \in [s]$ is picked independently with replacement with probability $p_i = \ell_i(\mat{P}) / R$, and we set $\mat{\Omega}(j,i) = 1, \mat{D}(j,j) = \frac{1}{\sqrt{m p_i}}$. Other elements of $\mat{\Omega}, \mat{D}$ are 0.
\end{definition}

To calculate the leverage scores of $\mat{P}$, one can get the matrix $\mat{Q}$ via QR decomposition, and the scores can be retrieved via calculating the norm of each row of $\mat{Q}$.
However, performing QR decomposition of $\mat{P}$ is almost as costly as solving the linear least squares problem. The lemma below shows that leverage scores of $\mat{P}$ can be efficiently calculated
from smaller QR decompositions of the Kronecker product factors composing $\mat{P}$.

\begin{lemma}[Leverage Scores for Kronecker product~\cite{cheng2016spals}] \label{lem:lev_bound}
  Let $\mat{P} = \mat{A}^{(1)} \otimes \cdots \otimes \mat{A}^{(N)} \in \R^{s^N \times R^N}$, where $\mat{A}^{(i)}\in \R^{s \times R}$ and $s>R$.
    Leverage scores of $\mat{P}$ satisfy
    \begin{equation}\label{score_bound}
        \ell_i(\mat{P}) = \prod_{k=1}^N \ell_{i_k}(\mat{A}^{(k)}), \quad
        \text{where } i = 1+\sum_{k=1}^N (i_k-1) s^{k-1}.
    \end{equation}
\end{lemma}
\begin{proof}
To show \eqref{score_bound}, we only need to show the case when $N=2$, since it can then be easily generalized to arbitrary $N$.
Consider the reduced QR decomposition of $\mat{A}^{(1)} \otimes \mat{A}^{(2)}$,
\begin{equation*}
\mat{A}^{(1)} \otimes \mat{A}^{(2)} = \mat{Q}^{(1)}\mat{R}^{(1)} \otimes \mat{Q}^{(2)}\mat{R}^{(2)} = (\mat{Q}^{(1)} \otimes \mat{Q}^{(2)})(\mat{R}^{(1)} \otimes \mat{R}^{(2)}) =  \mat{Q}\mat{R}.
\end{equation*}
The reduced $\mat{Q}$ term for $\mat{A}^{(1)} \otimes \mat{A}^{(2)}$ is $ \mat{Q}^{(1)} \otimes \mat{Q}^{(2)}$. Therefore, the leverage score of the $i$th row in $\mat{Q}$, $\ell_i$, can be expressed as,
\begin{align*}
  \ell_i(\mat{P}) &= \left\|\mat{Q}(i,:)\right\|_2^2
= \left\|\mat{Q}^{(1)}(i_1,:) \otimes \mat{Q}^{(2)}(i_2,:)\right\|_2^2 \\
&= \left\|\mat{Q}^{(1)}(i_1,:)\right\|_2^2 \left\|\mat{Q}^{(2)}(i_2,:)\right\|_2^2 = \ell_{i_1}(\mat{A}^{(1)})\ell_{i_2}(\mat{A}^{(2)}).   
\end{align*} 
\end{proof}
Let $p_i = \ell_i(\mat{P})/R^N$ denote the leverage score sampling probability for $i$th index, and $p^{(k)}_{i_k} = \ell_{i_k}(\mat{A}^{(k)})/R$ for $k\in[N]$ denote the leverage score sampling probability for $i_k$th index of $\mat{A}^{(k)}$. Based on \cref{lem:lev_bound}, we have
\[
p_i = p^{(1)}_{i_1} \cdots p^{(N)}_{i_N}.
\]
Therefore, leverage score sampling can be efficiently performed
by sampling the  row of $\mat P$ associated with multi-index $(i_1,\dots,i_N)$, where $i_k$ is selected with probability $p^{(k)}_{i_k}$.
To calculate the leverage scores of each $\mat{A}^{(k)}$,
$N$ QR decompositions are needed, which in total cost
 $\bigO{NsR^2}$. 
In addition, the cost of this sampling process would be $\bigO{Nm}$ if $m$ samples are needed, making the overall cost 
$\bigO{NsR^2 + Nm}$.
To calculate $\mat{SP}$, for each sampled multi-index $(i_1,\dots,i_N)$, we need to perform the Kronecker product,
\[
\mat{A}^{(1)}(i_1,:) \otimes \cdots \otimes \mat{A}^{(N)}(i_N,:),
\]
which costs $\bigO{R^N}$.
Therefore, including the cost of QR decompositions, overall cost is
$\bigO{NsR^2 + mR^N}$.

Rather than performing importance random sampling based on leverage scores, another way introduced in \cite{jolliffe1972discarding} to construct the sketching matrix is to deterministically sample rows having the largest leverage scores. This idea is also used in \cite{larsen2020practical} for randomized CP decomposition. Papailiopoulos et al.~\cite{papailiopoulos2014provable} show that if the leverage scores follow a moderately steep power-law decay, then deterministic sampling can be provably as efficient and even better than random sampling. We compare both leverage score sampling techniques in \cref{sec:exp}. For the sampling complexity analysis in \cref{sec:sketch-rcls} and \cref{sec:cost}, we only consider the random sampling technique.

%% file: contents/ls.tex
Each subproblem of Tucker HOOI solves a linear least squares problem with the following properties,
\begin{enumerate}[itemsep=0pt,leftmargin=*]
    \item the left-hand-side matrix is a chain of Kronecker products of factor matrices,
    \item the left-hand-side matrix has orthonormal columns, since each factor matrix has orthonormal columns,
    \item the rank of the output solution is constrained to be less than $R$, as is shown in \eqref{eq:tucker_ls}.
\end{enumerate}
To the best of our knowledge, the relative error analysis of sketching techniques for this problem have not been discussed in the literature. 
In the following two theorems, we will show the sketch sizes of TensorSketch and leverage score sampling that are sufficient for 
the relative residual norm error of the problems to be bounded by $\bigO{\epsilon}$ with at least $1-\delta$ probability. The detailed proofs are presented in \cref{appendix:detailed_proofs}.

\begin{theorem}[TensorSketch for Rank-constrained Linear Least Squares]
\label{thm:tensorsketch-rcls}
Consider matrices $\mat{P} = \mat{A}^{(1)}\otimes \mat{A}^{(2)}\otimes \cdots \otimes \mat{A}^{(N-1)}$, where each $\mat{A}^{(i)}\in\R^{s\times R}$ has orthonormal columns, $s>R$, and $\mat{B}\in \mathbb{R}^{s^{N-1}\times n}$.
Let $\mat{S}\in\R^{m \times s^{N-1}}$ be an order $N-1$ TensorSketch matrix.
Let $\mat{\widetilde{X}}_{\text{r}}$ be the best rank-$R$ approximation of the solution of the problem
$\min_{\mat{X}} \fnrm{\mat{SPX}-\mat{SB}}$, and let $\mat{X}_{\text{r}} = \arg\min_{\mat{X},\text{rank}(\mat{X})=R} \left\|\mat{PX}-\mat{B}\right\|_F$. 
With 
\begin{equation}\label{eq:samplesize_ts}
  m = \bigO{(R^{(N-1)} \cdot 3^{N-1})/\delta \cdot (R^{(N-1)}  + 1/\epsilon^2)},  
\end{equation}
the approximation,
	\begin{equation}
\left\|\mat{A}\mat{\widetilde{X}}_{\text{r}}- \mat{B}\right\|_F^2	 \leq \left(1+\bigO{\epsilon}\right) \Big\|\mat{A}\mat{X}_{\text{r}}- \mat{B}\Big\|_F^2,
	\end{equation}
holds with probability at least $1-\delta$.
\end{theorem}

\begin{theorem}[Leverage Score Sampling for Rank-constrained Linear Least Squares]
\label{thm:leverage-rcls}
Given matrices $\mat{P} = \mat{A}^{(1)}\otimes \mat{A}^{(2)}\otimes \cdots \otimes \mat{A}^{(N-1)}$, where each $\mat{A}^{(i)}\in\R^{s\times R}$ has orthonormal columns, $s>R$, and $\mat{B}\in \mathbb{R}^{s^{N-1}\times n}$.
Let $\mat{S}\in\R^{m \times s^{N-1}}$ be a leverage score sampling matrix for $\mat{P}$.
Let $\mat{\widetilde{X}}_{\text{r}}$ be the best rank-$R$ approximation of the solution of the problem
$\min_{\mat{X}} \fnrm{\mat{SPX}-\mat{SB}}$, and let $\mat{X}_{\text{r}} = \arg\min_{\mat{X},\text{rank}(\mat{X})=R} \left\|\mat{PX}-\mat{B}\right\|_F$. 
With $m = \bigO{R^{(N-1)}/(\epsilon^2\delta) }$,
the approximation,
	\begin{equation}
\left\|\mat{A}\mat{\widetilde{X}}_{\text{r}}- \mat{B}\right\|_F^2	 \leq \left(1+\bigO{\epsilon}\right) \Big\|\mat{A}\mat{X}_{\text{r}}- \mat{B}\Big\|_F^2,
	\end{equation}
holds with probability at least $1-\delta$.
\end{theorem}

Therefore, for leverage score sampling, $\bigO{R^{(N-1)}/(\epsilon^2\delta) }$ number of samples are sufficient to get $(1+\bigO{\epsilon})$-accurate residual with probability at least $1-\delta$. The sketch size upper bound for TensorSketch is higher than that for leverage score sampling, suggesting that leverage score sampling is better. As can be seen in \eqref{eq:samplesize_ts}, when $R^{N-1}\leq 1/\epsilon^2$, the sketch size bound for TensorSketch is $\bigO{3^{N-1}}$ times that for leverage score sampling. When $R^{N-1} > 1/\epsilon^2$, the ratio is even higher. The accuracy comparison of the two methods is discussed further in \cref{sec:exp}. 

While TensorSketch has a worse upper bound compared to leverage score sampling,
it is more flexible since the sketching matrix is independent of the left-hand-side matrix. 
One can derive a sketch size bound that is sufficient to get $(1+\bigO{\epsilon})$-accurate residual norm for
linear least squares with general (not necessarily rank-based) constraints (detailed in \cref{appendix:general_ts}).
This bound is looser than \eqref{eq:samplesize_ts}, while applicable for general constraints. For leverage score sampling, 
we do not provide a sample size bound for general constrained linear least squares.

\begin{table}[!ht]
  \begin{center}
    \renewcommand{\arraystretch}{1.3}
    {\small
    \begin{tabular}{l|p{0.34\textwidth}|p{0.34\textwidth}}
      Sketching method  & Rank-constrained least squares & Unconstrained least squares \\ \hline
      Leverage score sampling  & $\bigO{R^{(N-1)}/(\epsilon^2\delta) }$  (\cref{thm:leverage-rcls})
      & 
      $\bigO{R^{(N-1)}/(\epsilon\delta) } \text{  or}$ 
      $\bigO{R^{(N-1)}\log(1/\delta)/\epsilon^2 }$ \cite{larsen2020practical}  \\ \hline
      TensorSketch & $\bigO{(3R)^{(N-1)}/\delta \cdot (R^{(N-1)}  + 1/\epsilon^2)}$ (\cref{thm:tensorsketch-rcls})
      & $\bigO{(3R)^{(N-1)}/\delta \cdot (R^{(N-1)}  + 1/\epsilon)}$ 
      \\ \hline
    \end{tabular}
    }
    \renewcommand{\arraystretch}{1}
  \end{center}
\caption{Comparison of sketch size upper bounds for rank-constrained linear least squares and unconstrained linear least squares. The upper bounds are sufficient for 
the relative residual norm error to be bounded by $\bigO{\epsilon}$ with at least $1-\delta$ probability.
}
\label{tab:sketch-compare}
\end{table}

We also compare the sketch size upper bounds for rank-constrained linear least squares and unconstrained linear least squares in \cref{tab:sketch-compare}. For both leverage score sampling and TensorSketch, the upper bounds for rank-constrained problems are at most $\bigO{1/\epsilon}$ times the upper bounds for unconstrained linear least squares problem. 
The error of sketched rank-constrained solution consists of two parts, the error of the sketched unconstrained linear least squares solution, and the error from low-rank approximation of the unconstrained solution. To make sure the second error term has a relative error bound of $\bigO{\epsilon}$, we restrict the first error term to be relatively bounded by $\bigO{\epsilon^2}$, incurring a larger sketch size upper bound.

%% file: contents/init.tex
\begin{algorithm}
    \caption{\textbf{Init-RRF}: Initialization based on randomized range finder}
\label{alg:rrf}
\begin{algorithmic}[1]
\STATE{\textbf{Input: }Matrix $\mat{M}\in\R^{n\times m}$, rank $R$, tolerance $\epsilon$
}
\STATE{Initialize $\mat{S} \in \R^{m \times k}$, with $k = \bigO{R/\epsilon}$, as a composite sketching matrix
(see \cref{def:composite})}
\STATE{$\mat{B} \leftarrow\mat{MS}$}
\STATE{$\mat{U},\mat{\Sigma},\mat{V} \leftarrow \texttt{SVD}(\mat{B})$}
\RETURN $\mat{U}(:,:R)$
\end{algorithmic}
\end{algorithm}

The effectiveness of sketching with leverage score sampling for Tucker-ALS is dependent on finding a good initialization of the factor matrices. 
This sensitivity arises because in each subproblem \eqref{eq:tucker_ls}, only part of the input tensor being sampled is taken into consideration, and some non-zero input tensor elements are unsampled
in all ALS linear least squares subproblems
if the initialization of the factor matrices are far from the accurate solutions. 
Initialization is not 
a big problem for CountSketch/TensorSketch, since all the non-zero elements in the input tensor appear in the sketched right-hand-side.

An unsatisfactory initialization can severely affect the accuracy of leverage score sampling if 
elements of the tensor have large variability in magnitudes, 
a property known as high coherence. The coherence~\cite{candes2009exact} of a matrix $\mat{U}\in\R^{n\times r}$ with $n>r$ is defined as $\mu(\mat{U})= \frac{n}{r}\max_{i<n}\|\mat{Q}_{U}^T\vcr{e}_i\|$, where $\mat{Q}_{U}$ is an orthogonal basis for the column space of $\mat{U}$ and $\vcr{e}_i$ for $i\in[n]$ is a standard basis. Large coherence means that the orthogonal basis $\mat{Q}_{U}$ has large row norm variability. A tensor $\tsr{T}$ has high coherence if all of its matricizations $\mat{T}_{(i)}^T$ for $i\in[N]$ have high coherence.

We use an example to illustrate the problem of bad initializations
for leverage score sampling on tensors with high coherence. Suppose we seek a
rank $R$ Tucker decomposition of $\tsr{T}\in \R^{s\times s\times s}$ expressed as
\[
\tsr{T} = \tsr{C} \times_1 \mat{A} \times_2 \mat{A} \times_3 \mat{A} + \tsr{D},
\]
where $\tsr{C}\in \R^{R\times R\times R}$ is a tensor with elements drawn from a normal distribution, $\tsr{D}\in \R^{s\times s\times s}$ is a very sparse tensor (has high coherence), and $\mat{A}\in\R^{s\times R}$ is an orthogonal basis for the column space of a matrix with elements drawn from a normal distribution. 
Let all the factor matrices be initialized by $\mat{A}$. Consider $R \ll s$ and let the leverage score  sample size $m=R$.
Since $\tsr{D}$ is very sparse, there is a
high probability that most of the non-zero elements in $\tsr{D}$ are not sampled in all the sketched subproblems, resulting in a decomposition error proportional to $\|\tsr{D}\|_F$.

This problem can be fixed by initializing factor matrices using the randomized range finder (RRF) algorithm.
For each matricization $\mat{T}^{(i)}\in \R^{s\times s^{N-1}}$, where $i\in[N]$, we first find a good low-rank subspace $\mat{U}\in \R^{s \times m}$, where $m=\bigO{R/\epsilon}$, such that it is $\epsilon$-close to the rank-$R$ subspace defined by its leading left singular vectors,
\begin{equation}\label{eq:err_rrf}
      \left\|\mat{T}^{(i)} - \mat{U}\mat{U}^T\mat{T}^{(i)}\right\|_F^2  \leq (1 + \epsilon)\min_{\text{rank}(\mat{X})\leq R}\left\|\mat{T}^{(i)} - \mat{X}\right\|^2_F,
\end{equation}
and then initialize $\mat{A}^{(i)}$ based on the first $R$ columns of $\mat{U}$.
To calculate $\mat{U}$, we use a composite sketching matrix $\mat{S}$ defined in \cref{def:composite}, such that $\mat{U}$ is calculated via performing SVD on the sketched matrix $\mat{\mat{T}}^{(i)}\mat{S}$. Based on \cref{thm:good_lowrank}, \eqref{eq:err_rrf} holds with high probability.

\begin{definition}[Composite sketching matrix~\cite{Boutsidis2015CommunicationoptimalDP,wang2015practical}]\label{def:composite}
Let $k_1 = \bigO{R/\epsilon}$ and $k_2 = \bigO{R^2 + R/\epsilon}$. 
The composite sketching matrix $\mat{S}\in\R^{s \times k_1}$ is defined as $\mat{S}=\mat{T}\mat{G}$, where $\mat{T}\in \R^{s\times k_2}$ is a CountSketch matrix (defined in \cref{def:countsketch}), and $\mat{G}\in\R^{k_2\times k_1}$ contains elements selected randomly from a normal distribution with variance $1/k_1$.
\end{definition}

\begin{theorem}[Good low-rank subspace~\cite{Boutsidis2015CommunicationoptimalDP}]\label{thm:good_lowrank}
Let $\mat{T}$ be an $m\times n$ matrix, $R < \rank(\mat{T})$ be a rank parameter, and $\epsilon > 0$ be an accuracy parameter. Let $\mat{S} \in \R^{n\times k}$ be a 
composite sketching matrix defined as in \cref{def:composite}. 
Let $\mat{B} = \mat{TS}$
and let $\mat{Q}\in\R^{m\times k}$ be any orthogonal basis for the column space of $\mat{B}$. Then, with probability at least 0.99,
\begin{equation}
    \left\|\mat{T} - \mat{QQ}^T\mat{T}\right\|_F^2  \leq (1 + \epsilon)\left\|\mat{T} - \widetilde{\mat{T}}\right\|^2_F,
\end{equation}
where $\widetilde{\mat{T}}$ is the best rank-$R$ approximation of $\mat{T}$.
\end{theorem}

The algorithm is shown in \cref{alg:rrf}. The multiplication $\mat{\mat{T}}^{(i)}\mat{S}$ costs $\bigO{\nnz(\tsr{T}) + sR^3/\epsilon}$, and the SVD step costs $\bigO{sR^2/\epsilon}$, making the cost of the initialization step $\bigO{\nnz(\tsr{T}) + sR^3/\epsilon}$. Since we need at least go over all the non-zero elements of the input tensor for a good initialization guess,
the cost is $\Omega(\nnz(\tsr{T}) + sR)$. Consequently,
\cref{alg:rrf} is computationally efficient for small $R$.

Note that since $\mat{A}^{(i)}$ is only part of $\mat{U}$,  the error $\left\|\mat{T}^{(i)} - \mat{A}^{(i)}\mat{A}^{(i)}{}^T\mat{T}^{(i)}\right\|_F^2$ is generally higher than that shown in \eqref{eq:err_rrf}, so further ALS sweeps are necessary to further decrease the residual. 
Based on the experimental results shown in \cref{sec:exp}, this initialization greatly enhances the performance of leverage score sampling for tensors with high coherence.

%% file: material_arxiv/cost_arxiv.tex
\subsection{Cost Analysis for Sketched Tucker-ALS}
\label{subsec:cost}
\input{contents/cost/main_alg}
\input{contents/cost/cost}

\subsection{Cost Analysis for CP Decomposition}
\label{sec:cpd}
\input{contents/cost/cpd}

%% file: contents/cost/main_alg.tex
\begin{algorithm}[!ht]\small
\caption{\textbf{Sketch-Tucker-ALS}: Sketched ALS procedure for Tucker decomposition}\label{tucker-rals}
\begin{algorithmic}[1]
\label{alg:rand-tucker-als}
\STATE{\textbf{Input: }Input tensor $\tsr{T}\in\mathbb{R}^{s_1 \times \cdots\times s_N}$, 
Tucker ranks $\{R_1,\ldots,R_N\}$, 
maximum number of sweeps $I_{\text{max}}$, sketching tolerance $\epsilon$}
\STATE $\tsr{C}\leftarrow \tsr{O}$
\FOR{{$n\in \inti{2}{N} $}} \label{line:start-rrf}{ 
\STATE $\mat{A}^{(n)} \leftarrow \texttt{Init-RRF}(\mat{T}_{(n)}, R_n, \epsilon)$
}\ENDFOR \label{line:end-rrf}
\FOR{{$i\in \inti{1}{I_{\text{max}}} $}} {
	\FOR{{$n\in \inti{1}{N} $}} \label{line:innerfor}{
  \STATE{Build the sketching matrix $\mat{S}^{(n)}$} \label{line:build-omega}
  \STATE{$\mat{Y} \leftarrow \mat{S}^{(n)} \mat{T}_{(n)}$}\label{line:sketch-rhs}
  \STATE{$\mat{Z} \leftarrow \mat{S}^{(n)} (\mat{A}^{(1)} \otimes \cdots \otimes \mat{A}^{(n-1)} \otimes \mat{A}^{(n+1)}\otimes \cdots\otimes\mat{A}^{(N)})$}\label{line:sketch-lhs}
  \STATE $\mat{C}_{(n)}^T, \mat{A}^{(n)} \leftarrow $ \texttt{RSVD-LRLS}($\mat{Z}, \mat{Y}, R$) \label{line:rcls}
	}\ENDFOR \label{line:innerforend}
}\ENDFOR
\RETURN $\left\{\tsr{C}, \mat{A}^{(1)}, \ldots , \mat{A}^{(N)} \right\}$ 
\end{algorithmic}
\end{algorithm}

\begin{table}[!ht]
  \begin{center}
    \renewcommand{\arraystretch}{1.3}
    {\small
    \begin{tabular}{p{0.215\textwidth}|p{0.22\textwidth}|p{0.31\textwidth}|l}
      Algorithm for Tucker  & LS subproblem cost & Sketch size ($m$) & Prep cost\\ \hline
      ALS  & $\Omega(\nnz(\tsr{T})R)$ & / & / \\ \hline
      ALS+TensorSketch~\cite{malik2018low}  & $\bigOt{msR + mR^{N}}$ & 
      $\bigO{(3R)^{N-1}/\delta \cdot (R^{N-1}  + 1/\epsilon)}$ 
      & $\bigO{N\nnz(\tsr{T})}$  \\ \hline
      ALS+TTMTS \cite{malik2018low}&  $\bigOt{msR^{N-1}}$ & 
      Not shown 
      & $\bigO{N\nnz(\tsr{T})}$  \\ \hline
      \underline{ALS + TensorSketch} & $\bigO{msR + mR^{2(N-1)}}$ & $\bigO{(3R)^{N-1}/\delta \cdot (R^{N-1}  + 1/\epsilon^2)}$ (\cref{thm:tensorsketch-rcls})
      & $\bigO{N\nnz(\tsr{T})}$  \\ \hline
      \underline{ALS+leverage scores}
      & $\bigO{msR + mR^{2(N-1)}}$
      & $\bigO{R^{N-1}/(\epsilon^2\delta)}$ (\cref{thm:leverage-rcls}) & /
    \end{tabular}
    }
    \renewcommand{\arraystretch}{1}
  \end{center}
\caption{Comparison of algorithm complexity between Tucker-ALS (HOOI), ALS with the TensorSketch/leverage score sampling, and the sketched Tucker-ALS algorithms introduced in \cite{malik2018low}. The 
third column
shows the sketch size sufficient for the sketched linear least squares to be $(1+\bigO{\epsilon})$ accurate with probability at least $1-\delta$.
Underlined algorithms are our new contributions.
}
\label{tab:alscompare}
\end{table}

%% file: contents/cost/cost.tex
\begin{algorithm}\small
\caption{\textbf{RSVD-LRLS}: Low-rank approximation of least squares solution via randomized SVD}
\label{alg:RCLS}
\begin{algorithmic}[1]
\STATE{\textbf{Input: }Left-hand-side matrix $\mat{Z}\in \mathbb{R}^{m\times r}$, right-hand-side matrix $\mat{Y}\in \mathbb{R}^{m\times s}$, rank $R$}
\STATE Initialize $\mat{S}\in \R^{s\times \bigO{R}}$ as a random Gaussian sketching matrix
\STATE $\mat{B} \leftarrow (\mat{Z}^T\mat{Z})^{-1}$
\STATE $\mat{C} \leftarrow \mat{B} \mat{Z}^T \mat{Y}\mat{S}$
\STATE{$\mat{Q}, \mat{R}\leftarrow$ \texttt{qr}($\mat{C}$)}
\STATE $\mat{D} \leftarrow \mat{Q}^T \mat{B} \mat{Z}^T \mat{Y}$
\STATE $\mat{U}, \mat{\Sigma}, \mat{V} \leftarrow \texttt{svd}(\mat{D})$
\RETURN $\mat{Q}\mat{U}(:,:R)\mat{\Sigma}(:R,:R), \mat{V}(:,:R)$
\end{algorithmic}
\end{algorithm}

In this section, we provide detailed cost analysis for the sketched Tucker-ALS algorithm. The algorithm is shown in \cref{alg:rand-tucker-als}.
Note that for leverage score sampling, lines \ref{line:build-omega} and \ref{line:sketch-rhs} need to be recalculated for every sweep, since $\mat{S}^{(n)}$ is dependent on the factor matrices. 
On the 
other hand, the TensorSketch embedding is oblivious to the state of the factor matrices, so
 we can choose to use the same $\mat{S}^{(n)}$ for all the sweeps for each mode $n$ to save cost. This strategy is also used in \cite{malik2018low}.
Detailed cost analysis for each part of \cref{alg:rand-tucker-als} is listed below, where we assume $s_1 = \cdots = s_N = s$ and $R_1 = \cdots = R_N = R$.
\begin{itemize}[leftmargin=*,itemsep=0pt]
    \item Line \ref{line:start-rrf}-\ref{line:end-rrf}: the cost is $\bigO{N\nnz(\tsr{T}) + NsR^3/\epsilon}$ by the analysis
    in \cref{sec:init}.
    \item Line \ref{line:build-omega}: if using leverage score sampling, the cost is $\bigO{sR}$ per subproblem (for computing the leverage scores of the previously updated $\mat{A}^{(i)}$). If using TensorSketch, the cost is $\bigO{Ns}$, which is only incurred for the first sweep.
    \item Line \ref{line:sketch-rhs}: if using leverage score sampling, the cost is $\bigO{ms}$ per subproblem; if using TensorSketch, the cost is $\bigO{N\nnz(\tsr{T})}$, 
    and is only incurred for the first sweep.
    \item Line \ref{line:sketch-lhs}: if using leverage score sampling, the cost is $\bigO{mR^{N-1}}$ per subproblem; if using TensorSketch, the cost is $\bigO{NsR + m\log (m)R^{N-1}}$ per subproblem, as analyzed in \cref{sec:sketch}.
    \item Line \ref{line:rcls}: the cost is $\bigO{msR + mR^{2(N-1)}}$ per subproblem, under the condition that $m \geq R^{N-1}$ and using randomized SVD as detailed in \cref{alg:RCLS}. 
\end{itemize}
Therefore, the cost for each subproblem (lines~\ref{line:build-omega}-\ref{line:rcls}) is $\bigO{msR + mR^{2(N-1)}}$, for both leverage score sampling and TensorSketch. For TensorSketch, another cost of $\bigO{N\nnz(\tsr{T})}$ is incurred at the first sweep to sketch the right-hand-side matrix, which we refer to as preparation cost. Using the initialization scheme based on RRF to initialize the factor matrices would increase the cost of both sketching techniques by
$\bigO{N\nnz(\tsr{T}) + NsR^3/\epsilon}$.

We compare the cost of each linear least squares subproblem between our sketched ALS algorithms with both HOOI and the sketched ALS algorithms introduced in \cite{malik2018low} in \cref{tab:alscompare}. For the ALS + TensorSketch algorithm in \cite{malik2018low}, $N + 1$ subproblems are solved in each sweep, and in each subproblem either one factor matrix or the core tensor is updated based on the sketched \textit{unconstrained} linear least squares solutions.  
For the ALS + TTMTS algorithm, TensorSketch is simply used to accelerate the TTMc operations. 

For the solutions of sketched linear least squares problems to be unique, we need $m\geq R^{N-1}$ and hence $m = \Omega(R^{N-1})$. With this condition,
the cost of each linear least squares subproblem of our sketched ALS algorithms is less than that for ALS + TTMTS, but a bit more than the ALS + TensorSketch in \cite{malik2018low}. However, as shown in \cref{sec:exp}, our algorithms provide better accuracy as a result of updating more variables at a time.
We also show the sketch size upper bound sufficient to get a $(1+\bigO{\epsilon})$-accurate approximation in residual norm. As can be seen in the table, our sketching algorithm with leverage score sampling has the smallest sketch size, making it the best algorithm considering both the cost of each subproblem and the sketch size. In~\cite{malik2018low}, the authors give an error bound for the approximate matrix multiplication in ALS + TTMTS, but the relative error of the overall linear least squares problem is not given. 
For the ALS + TensorSketch algorithm in \cite{malik2018low}, the sketch size upper bound in \cref{tab:alscompare}  comes from the upper bound for the unconstrained linear least squares problem.

%% file: contents/cost/cpd.tex
\begin{algorithm}
    \caption{\textbf{CP-Sketch-Tucker}: CP decomposition with sketched Tucker-ALS}
\label{alg:cp-rand-tucker}
\begin{algorithmic}[1]
\STATE{\textbf{Input: }Tensor $\tsr{T}\in\mathbb{R}^{s_1\times\cdots s_N}$, rank $R$, maximum number of Tucker-ALS sweeps $I_{max}$, Tucker sketching tolerance $\epsilon$
}
\STATE{$\left\{\tsr{C}, \mat{B}^{(1)}, \ldots , \mat{B}^{(N)} \right\} \leftarrow$  \texttt{Rand-Tucker-ALS}($\tsr{T}, \{R, \ldots, R\}, I_{max}, \epsilon$)}
\STATE $\left\{\mat{A}^{(1)}, \ldots , \mat{A}^{(N)} \right\} \leftarrow \texttt{CP-ALS}(\tsr{C}, R)$
\RETURN $\{  \mat{B}^{(1)}\mat{A}^{(1)}, \ldots , \mat{B}^{(N)}\mat{A}^{(N)} \}$
\end{algorithmic}
\end{algorithm}

When $R \ll s$, sketched Tucker-ALS can also be used to accelerate CP decomposition.
When an exact CP decomposition of the desired rank exists, it is attainable from a Tucker decomposition of the same or greater rank.
In particular, given a CP decomposition of the desired rank for the core tensor from Tucker decomposition, it suffices to multiply respective factor matrices of the CP and Tucker decompositions to obtain a CP decomposition of the original tensor.
For the exact case, Tucker decomposition can be computed exactly via the sequentially truncated HOSVD, and for approximation, the Tucker model is generally easier to fit than CP.
Consequently, Tucker decomposition has been employed as a pre-processing step prior to running CP decomposition algorithms such as CP-ALS~\cite{carroll1980candelinc,zhou2014decomposition,bro1998improving,erichson2020randomized}.

We leverage the ability of Tucker decomposition to preserve low-rank CP structure to apply our fast randomized Tucker algorithms to low-rank CP decomposition.
We show the algorithm in \cref{alg:cp-rand-tucker}. 
In practice, the randomized Tucker-ALS algorithm takes a small number of sweeps (less than 5) to converge, and then CP-ALS can be applied on the core tensor, which is computationally efficient. 

The state-of-the-art approach for randomized CP-ALS \cite{larsen2020practical}
is to use leverage score sampling
to solve each subproblem \eqref{eq:subproblem}.
The cost sufficient to get $(1+\bigO{\epsilon})$-accurate residual norm for each subproblem is $\bigO{sR^N\log(1/\delta)/\epsilon^2}$. With the same criteria, the cost for sketched Tucker-ALS with leverage score sampling is $\bigO{sR^{N}/(\epsilon^2\delta) + R^{3(N-1)}/(\epsilon^2\delta)}$. As we can see, when $R\ll s$, the cost of each Tucker decomposition subproblem is only slightly higher than that of CP decomposition, and the fast convergence of Tucker-ALS 
makes this Tucker + CP method more efficient than directly applying CP decomposition on the input tensor.

%% file: material_arxiv/exp_arxiv.tex
In this section, we compare our randomized algorithms with reference algorithms for both Tucker and CP decompositions on several synthetic tensors. 
We evaluate accuracy based on the final fitness $f$ for each algorithm, defined as 
    \begin{equation}
    f = 1 - \frac{\|\tsr{T} - \widetilde{\tsr{T}}\|_F}{\|\tsr{T}\|_F},
    \end{equation}
    where $\tsr{T}$ is the input tensor and $\widetilde{\tsr{T}}$ is the reconstructed low-rank tensor. 
For Tucker decomposition, we focus on the comparison of accuracy and robustness 
of attained fitness across various synthetic datasets for different algorithms.
For CP decomposition, we focus on the comparison of accuracy and 
sweep count.
All of our experiments were carried out on an Intel Core i7 2.9 GHz Quad-Core machine using NumPy routines in Python.

\subsection{Experiments for Tucker Decomposition}
\label{subsec:exp_tucker}

We compare five different algorithms for Tucker decomposition.
Two baselines from previous work are considered, standard HOOI and the original TensorSketch-based randomized Tucker-ALS algorithm, which optimizes only one factor in Tucker decomposition at a time~\cite{malik2018low}.
We compare these to
our new randomized algorithm (\cref{tucker-rals}) based on TensorSketch, random leverage score sampling, and deterministic leverage score sampling. For each randomized algorithm, we test both random initialization for factor matrices as well as the initialization scheme based on RRF detailed in \cref{sec:init}. For the baseline HOOI algorithm, we report the performance with both random and HOSVD initializations. 
We use four synthetic tensors to evaluate these algorithms.
\begin{enumerate}[itemsep=0pt,leftmargin=*]
    \item \label{tsr:dense} \textbf{Dense tensors with specific Tucker rank}. We create order 3 tensors based on randomly-generated factor matrices $\mat{B}^{(n)}\in \R^{s\times R_{\text{true}}}$ and a core tensor $\tsr{C}$,
    \begin{equation}\label{eq:tsr_dense}
     \tsr{T} = \tsr{C}\times_1\mat{B}^{(1)}\times_2\mat{B}^{(2)}\times_3\mat{B}^{(3)}. 
    \end{equation}
    Each element in the core tensor and the factor matrices are i.i.d. normally distributed random variables $\mathcal{N}(0,1)$. 
    The ratio $R_{\text{true}}/R$, where $R$ is the decomposition rank, is denoted as $\alpha$. 
    \item \label{tsr:dense-bias} \textbf{Dense tensors with strong low-rank signal}. We also test on dense tensors with strong low-rank signal,
    \begin{equation}
        \tsr{T}^{(b)} = \tsr{T} + \sum_{i=1}^n \lambda_i \vcr{a}^{(1)}_i \circ \vcr{a}^{(2)}_i \circ \vcr{a}^{(3)}_i. 
    \end{equation}
    $\tsr{T}$ is generated based on \eqref{eq:tsr_dense}, and each vector $\vcr{a}_i^{(j)}$ has unit 2-norm. The magnitudes $\lambda_i$ for $i\in[n]$ are constructed based on the power-law distribution,
    $\lambda_i = C \frac{\|\tsr{T}\|_F}{i^{1+\eta}}$. In our experiments, we set $n=5,C=3$ and $\eta=0.5$. This tensor is used to model data whose leading low-rank components obey the power-law distribution, which is common in real datasets.
    \item \label{tsr:sparse} \textbf{Sparse tensors with specific Tucker rank}. We generate tensors based on \eqref{eq:tsr_dense} with each element in the core tensor and factor matrices being
    an i.i.d normally distributed random variable $\mathcal{N}(0,1)$ with probability $p$ and zero otherwise. 
    Since each element,
    \begin{equation}\label{eq:tsr_sparse}
    \tsr{T}(i,j,k) = \sum_{x,y,z} \mat{B}^{(1)}(i, x)\cdot \mat{B}^{(2)}(j, y)\cdot \mat{B}^{(3)}(k, z)\cdot \tsr{C}(x,y,z), 
    \end{equation}
    and 
    \[
    P\left[
    \mat{B}^{(1)}(i, x)\cdot\mat{B}^{(2)}(j, y)\cdot\mat{B}^{(3)}(k, z)\cdot\tsr{C}(x,y,z) \neq 0\right]
    = p^4,
    \]
    the expected sparsity of $\tsr{T}$, which is equivalent to the probability that each element $\tsr{T}(i,j,k)=0$, is lower-bounded by $1 - R_{\text{true}}^3p^4$. 
    Through varying $p$, we generate tensors with different expected sparsity. 

    \item \label{tsr:sparse-bias} \textbf{Tensors with large coherence}. 
    We also test on tensors with large coherence,
    \begin{equation}
        \tsr{T}^{(b)} = \tsr{T} + \tsr{N}. 
    \end{equation}
    $\tsr{T}$ is generated based on \eqref{eq:tsr_dense} or \eqref{eq:tsr_sparse}, and $\tsr{N}$ contains $n \ll s$ elements with random positions and same large magnitude. In our experiments, we set $n=10$, and each nonzero element in $\tsr{N}$ has the i.i.d. normal distribution $\mathcal{N}(\|\tsr{T}\|_F/\sqrt{n}, 1)$, which means the expected norm ratio $\E[\|\tsr{N}\|_F/\|\tsr{T}\|_F] = 1$. This tensor has large coherence and is used to test the robustness problem detailed in \cref{sec:init}.
\end{enumerate}

For all the experiments, we run 5 ALS sweeps  unless otherwise specified, and calculate the fitness based on the output factor matrices as well as the core tensor. We observe that 5 sweeps are sufficient for both HOOI and randomized algorithms to converge. For each randomized algorithm, we set the sketch size to be $KR^2$. 
For the RRF-based initialization, we set the sketch size ($k$ in \cref{alg:rrf}) as $\sqrt{K}R$.

\begin{figure}[!ht]   
\centering
\subfloat[Tensor \ref{tsr:dense} with $s=200$]{\includegraphics[width=0.33\textwidth, keepaspectratio]{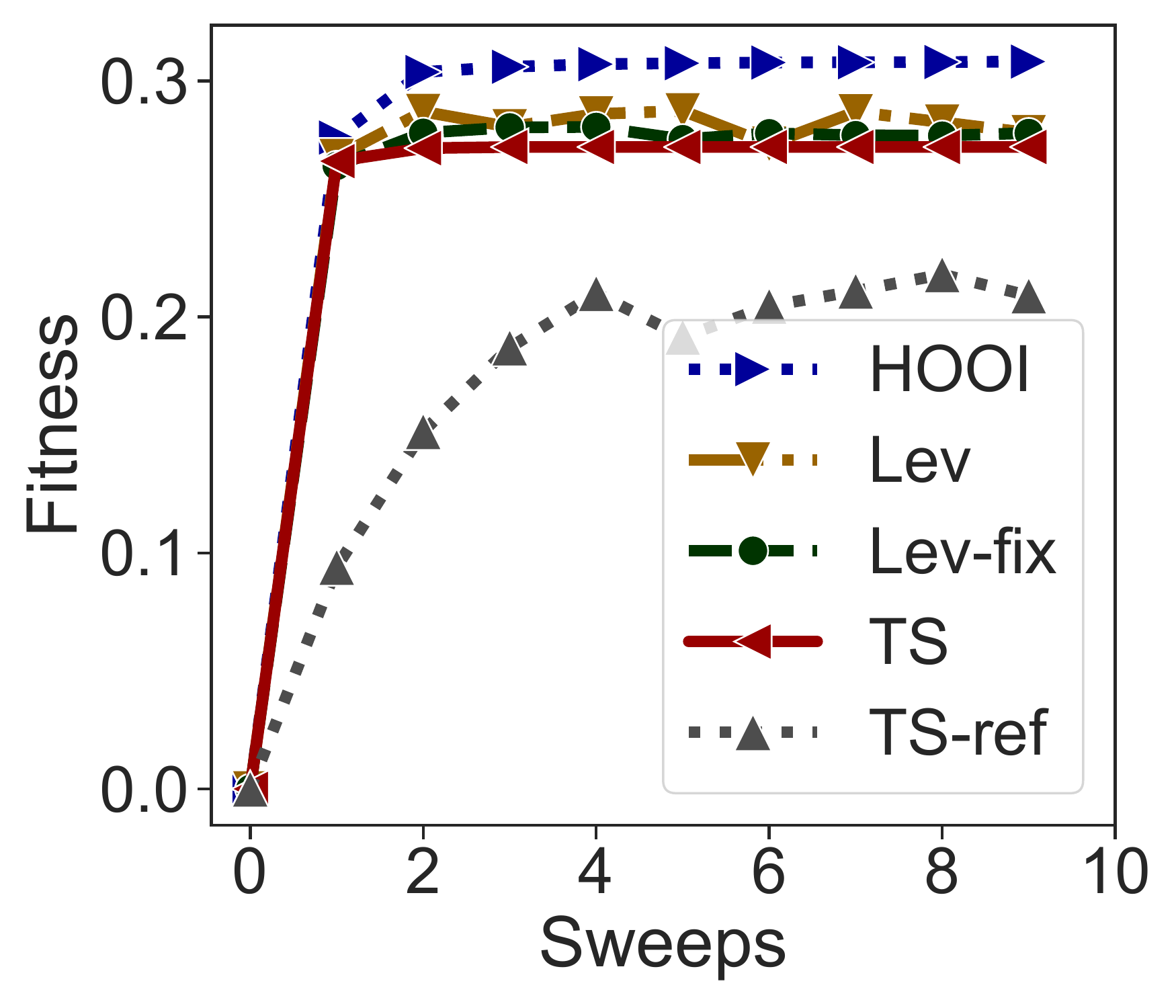}\label{figs-tuckerdetail1}}
\subfloat[Tensor \ref{tsr:dense-bias} with $s=200$]{\includegraphics[width=0.33\textwidth, keepaspectratio]{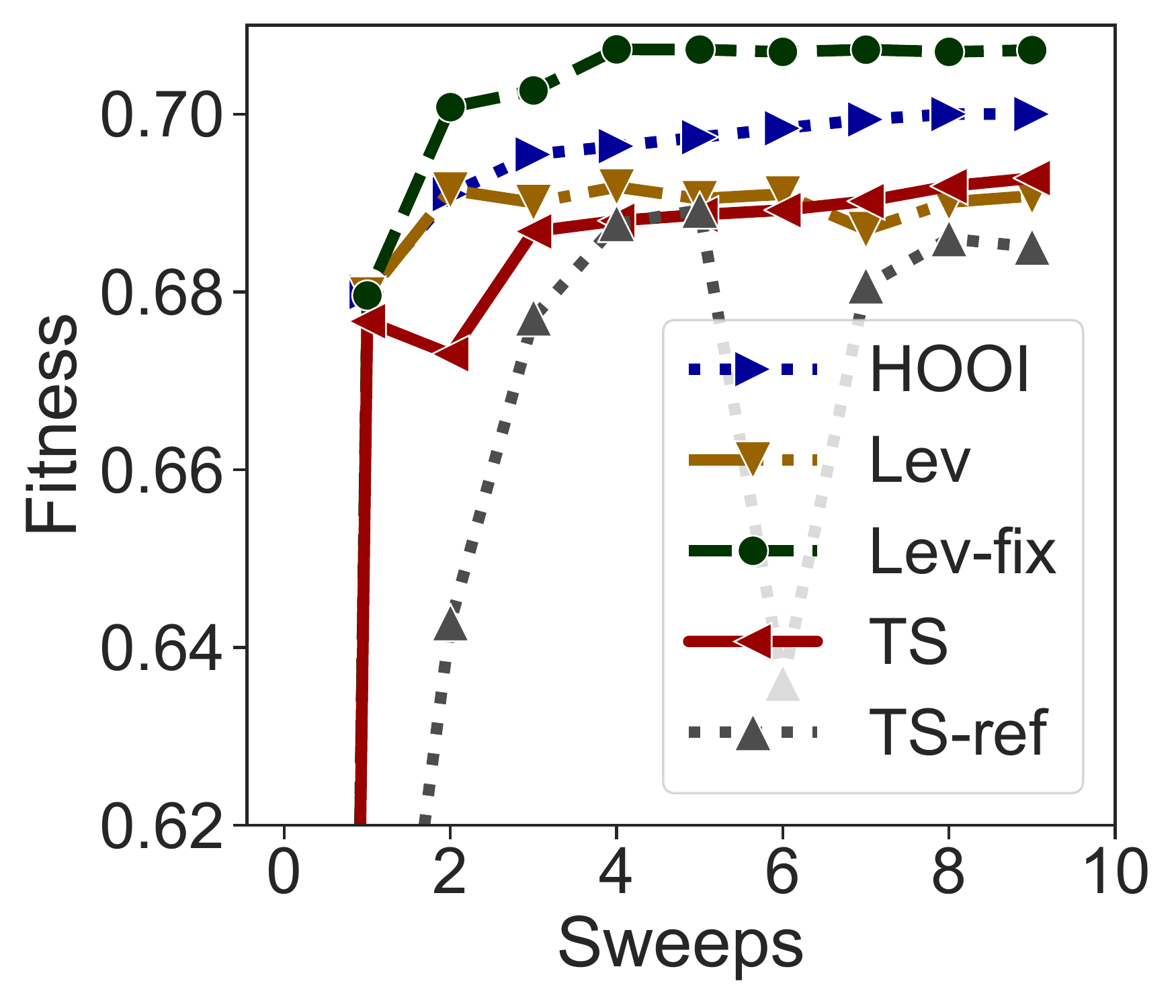}\label{figs-tuckerdetail2}}
\subfloat[Tensor \ref{tsr:sparse-bias} with $s=1000$]{\includegraphics[width=0.33\textwidth, keepaspectratio]{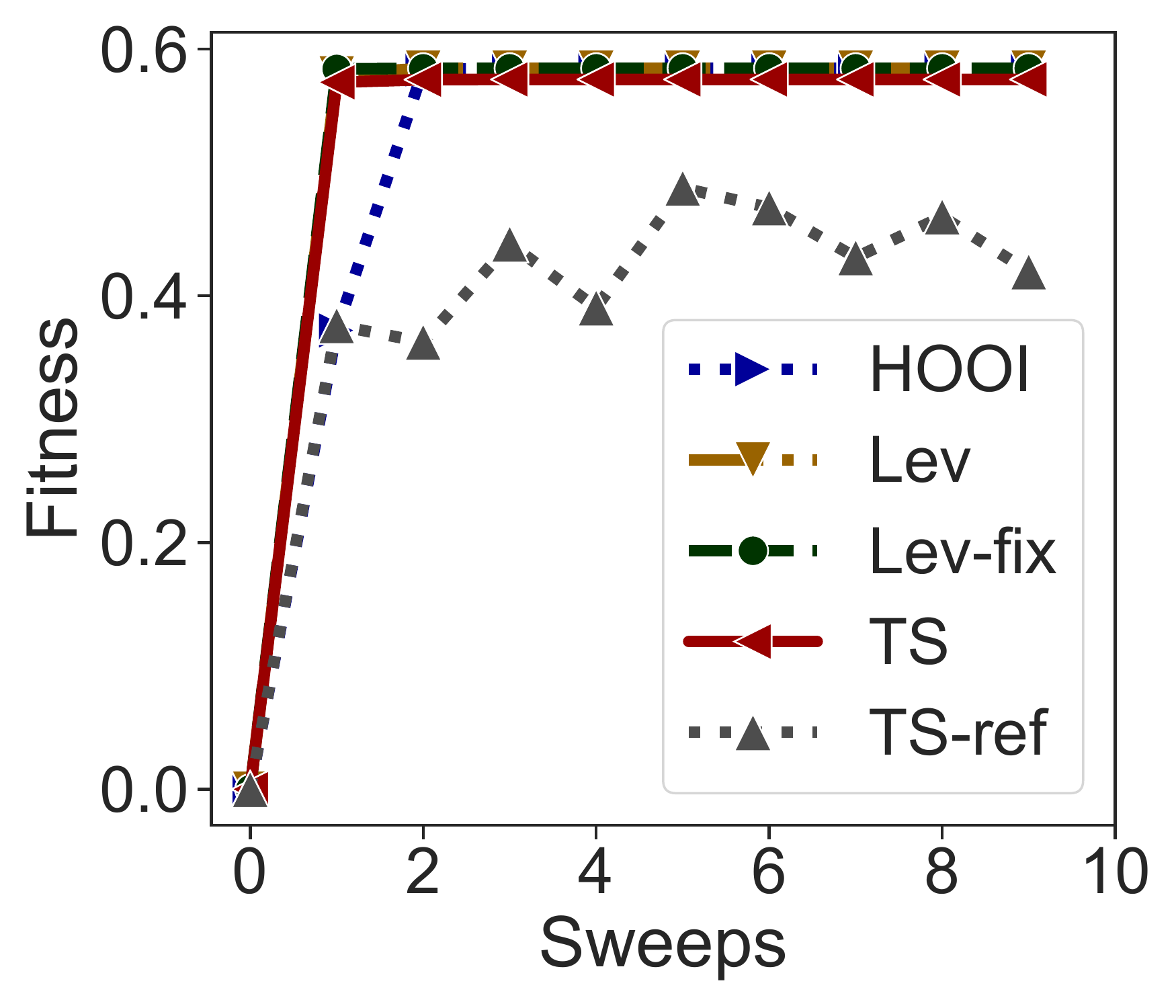}\label{figs-tuckerdetail3}}

\caption{Detailed fitness-sweeps relation for Tucker decomposition of three dense tensors with different parameters. 
For Tensor \ref{tsr:sparse-bias}, $\tsr{T}$ is generated based on \eqref{eq:tsr_dense}.
For all the experiments, we set $R=5, \alpha=1.6$, and $K=16$. 
In the plots, Lev, Lev-fix, and TS denote our new sketched Tucker-ALS scheme with leverage score random sampling, leverage score deterministic sampling, and TensorSketch, respectively. TS-ref denotes the reference sketched Tucker-ALS algorithm with TensorSketch.
HOOI is initialized with HOSVD, and all other methods are initialized with RRF (\cref{alg:rrf}). Markers represent the results per sweep.
}
\label{fig:tucker-fitness-sweep}
\end{figure}

\begin{figure}[!ht]   
\centering
\subfloat[Tensor \ref{tsr:dense} with $\alpha=1.2, K=16$]{\includegraphics[width=0.50\textwidth, keepaspectratio]{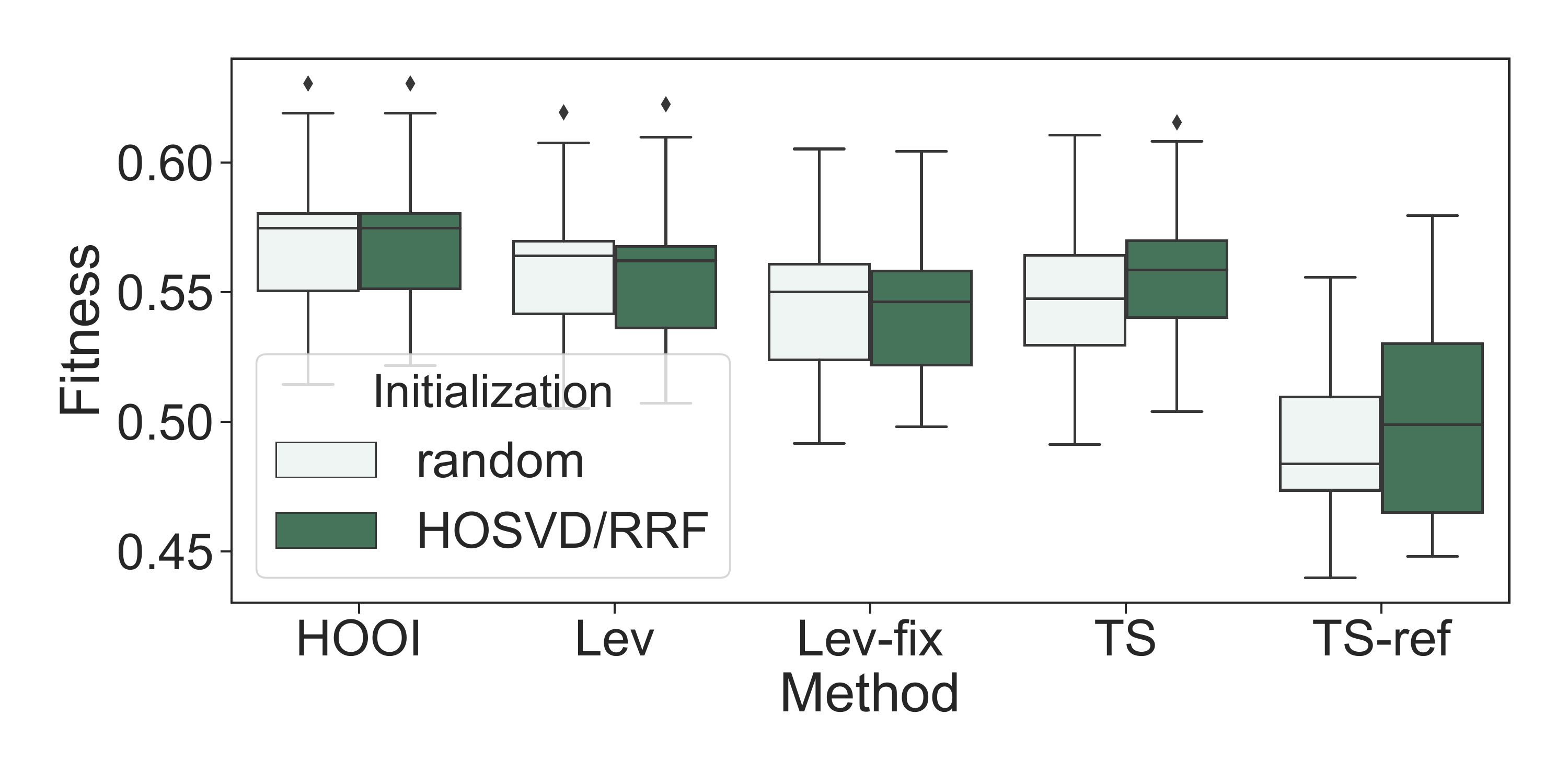}\label{figs-dense1}}
\subfloat[Tensor \ref{tsr:dense} with $\alpha=1.2$ and HOSVD/RRF init]{\includegraphics[width=0.50\textwidth, keepaspectratio]{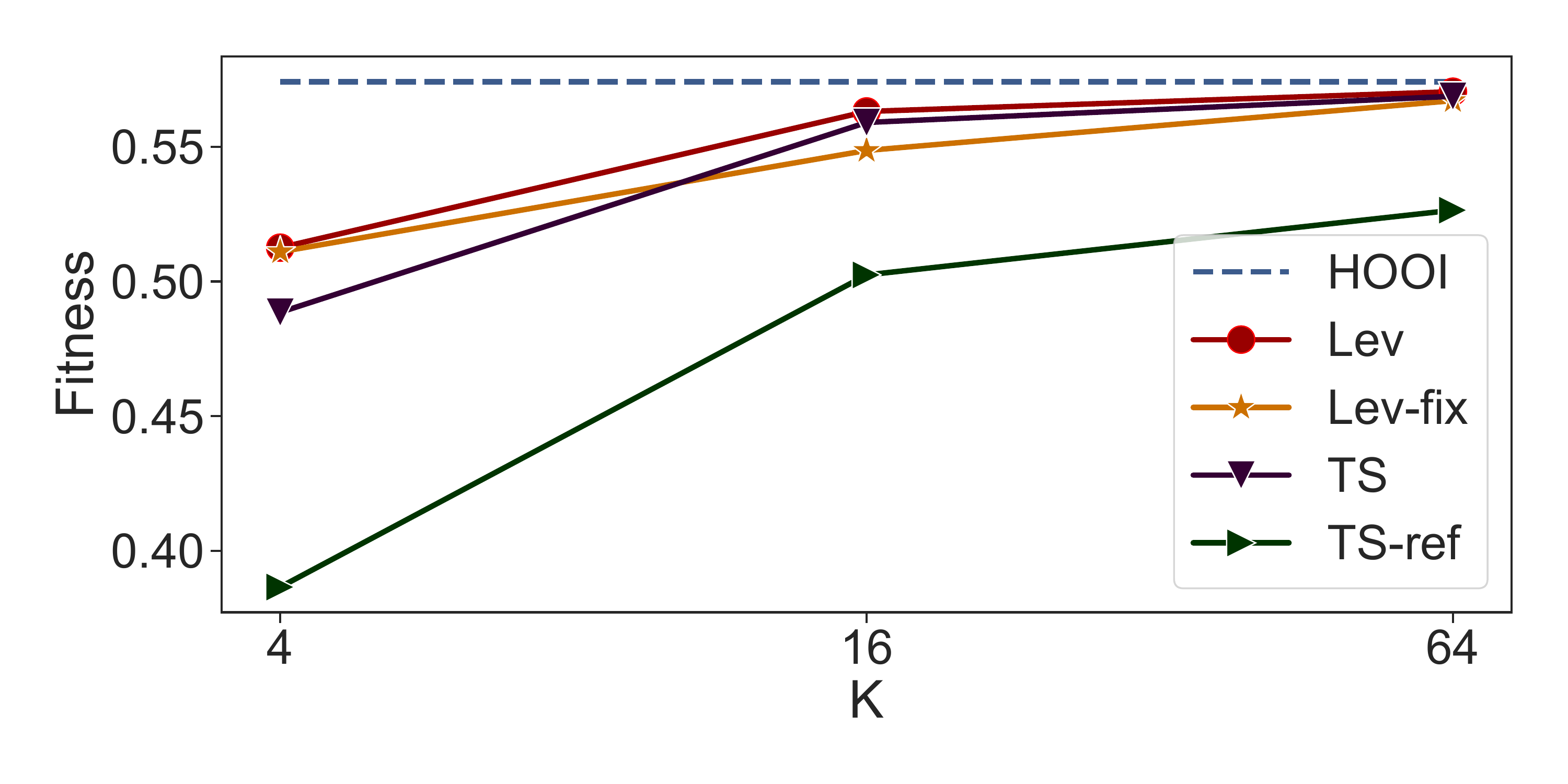}\label{figs-dense2}}

\subfloat[Tensor \ref{tsr:dense} with $\alpha=1.6, K=16$]{\includegraphics[width=0.50\textwidth, keepaspectratio]{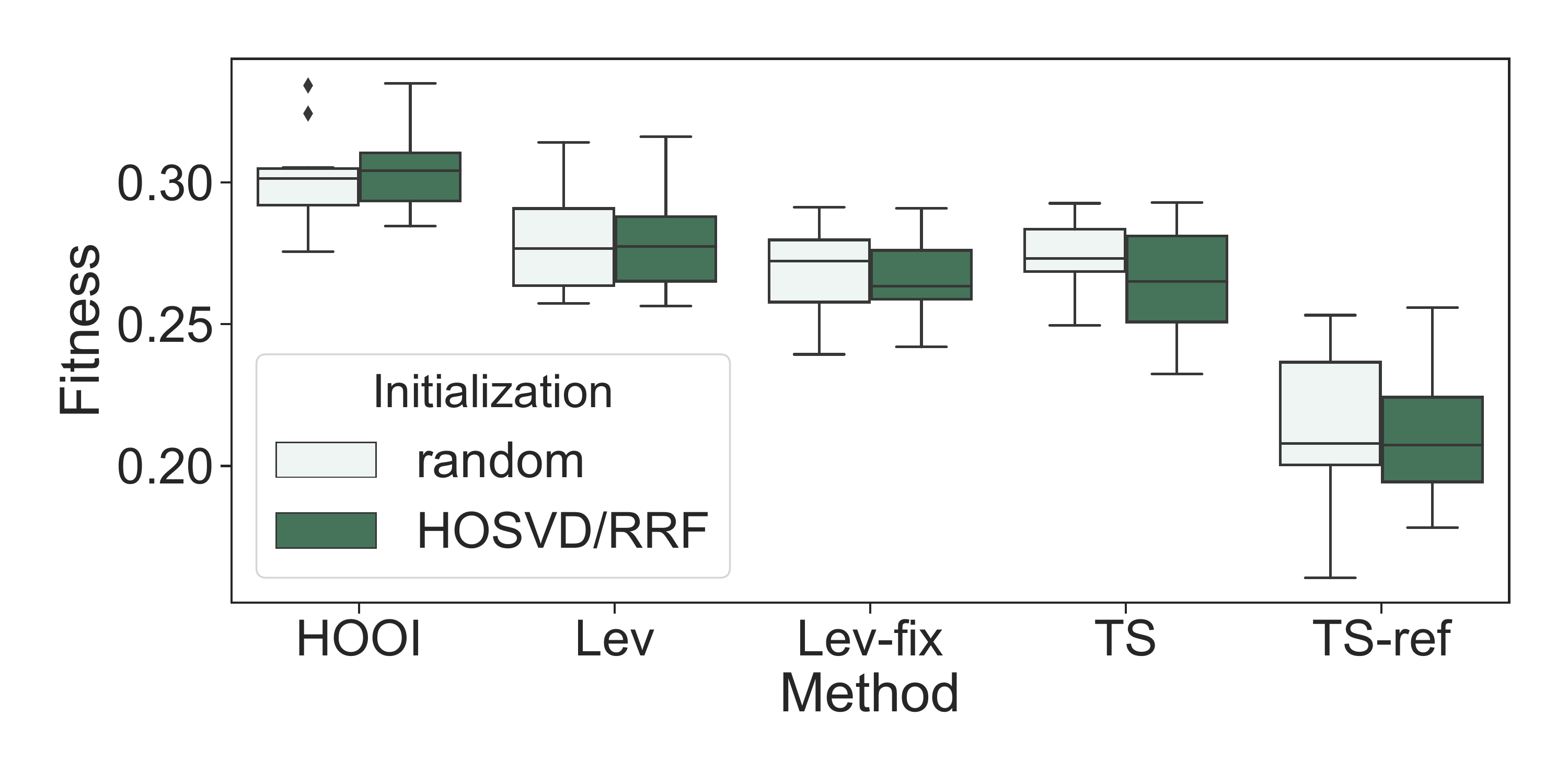}\label{figs-dense3}}
\subfloat[Tensor \ref{tsr:dense} with $\alpha=1.6$ and HOSVD/RRF init]{\includegraphics[width=0.50\textwidth, keepaspectratio]{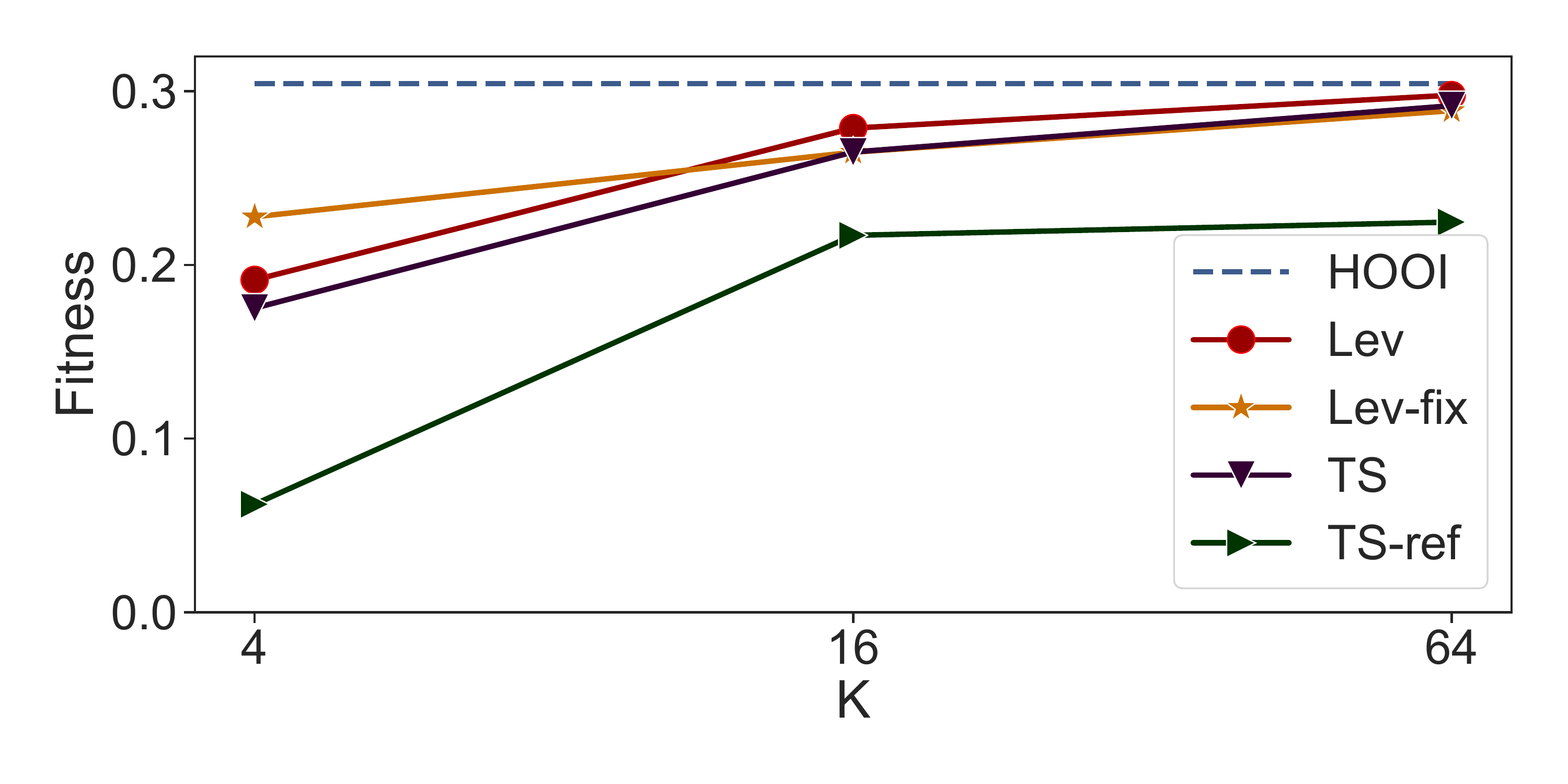}\label{figs-dense4}}

\caption[]{Experimental results for Tucker decomposition of Tensor \ref{tsr:dense}. For all the experiments, we set $s=200$ and $R=5$. 
HOSVD/RRF means HOOI is initialized with HOSVD, and all other methods are initialized with RRF (\cref{alg:rrf}).
\textbf{(a)(c)} Box plots of the final fitness for each algorithm with different input tensors. Each box is based on 10 experiments with different random seeds. 
Each box shows the 25th-75th quartiles, the median is indicated by a horizontal line inside the box, and outliers are displayed as dots.
\textbf{(b)(d)} Relation between the final fitness and sketch size parameter $K$ for each algorithm with different input tensors. Each data point is the mean of 10 experimental results with different random seeds. 
}
\label{fig:dense}
\end{figure}

\begin{figure}[!ht]   
\centering
\subfloat[Tensor \ref{tsr:dense-bias} with $\alpha=1.6, K=16,s=200$]{\includegraphics[width=0.50\textwidth, keepaspectratio]{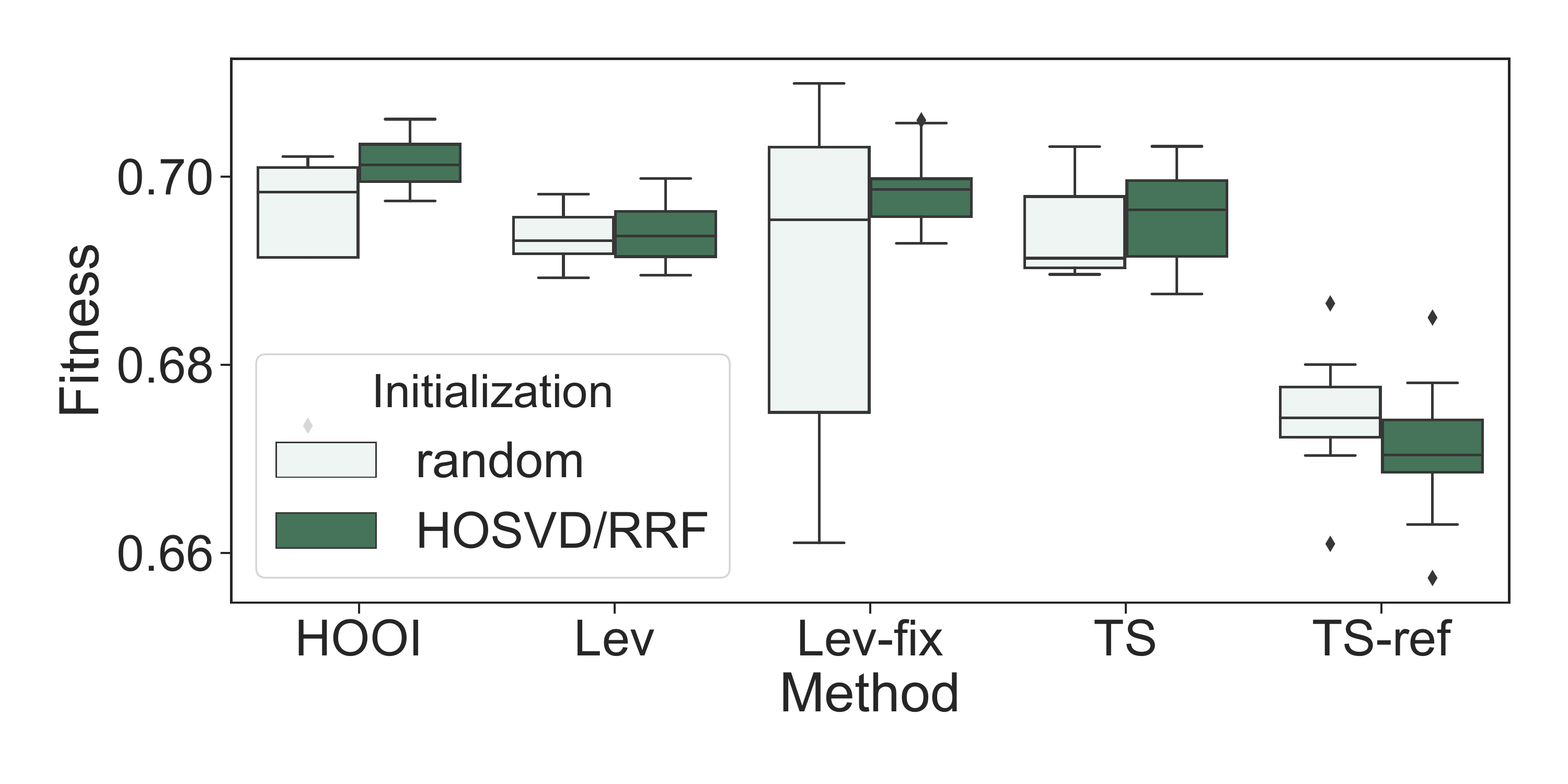}\label{figs-dense5}}
\subfloat[Tensor \ref{tsr:dense-bias} with $\alpha=1.6, s=200$ and HOSVD/RRF init]{\includegraphics[width=0.50\textwidth, keepaspectratio]{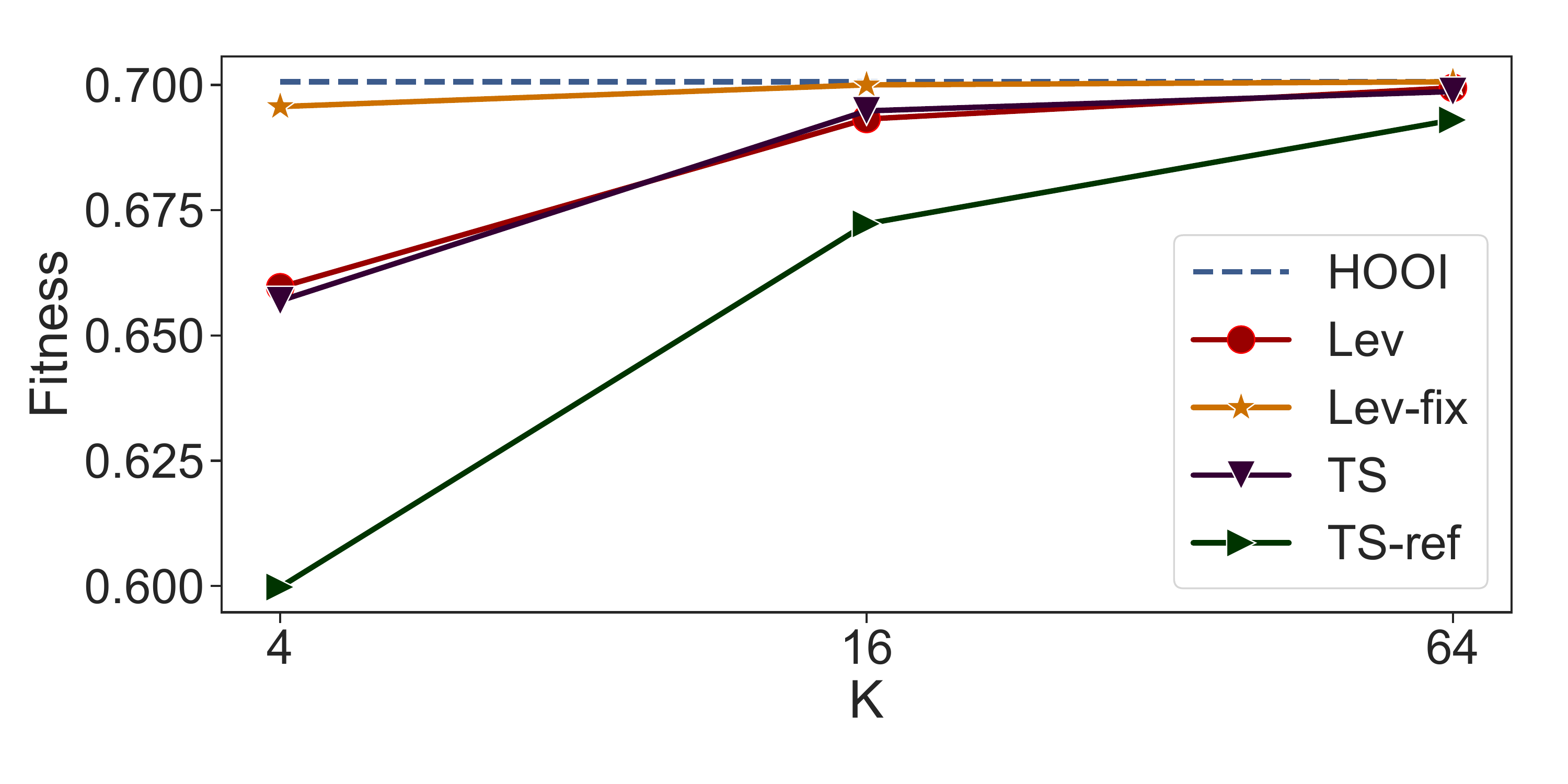}\label{figs-dense6}}

\subfloat[Tensor \ref{tsr:sparse-bias} with $\alpha=1.6, K=16$, $s=1000$]{\includegraphics[width=0.50\textwidth, keepaspectratio]{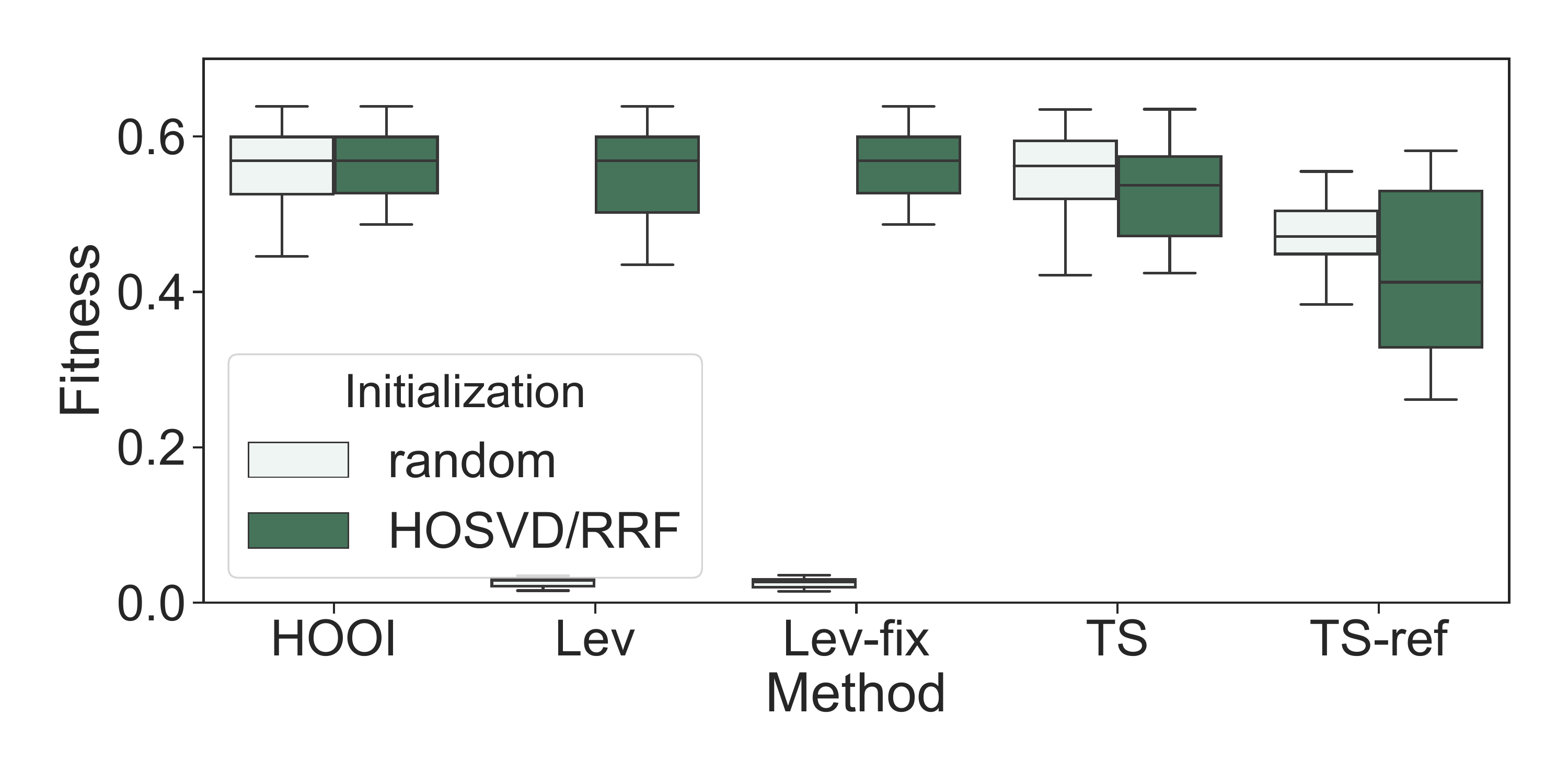}\label{figs-dense7}}
\subfloat[Tensor \ref{tsr:sparse-bias} with $\alpha=1.6, s=1000$ and HOSVD/RRF init]{\includegraphics[width=0.50\textwidth, keepaspectratio]{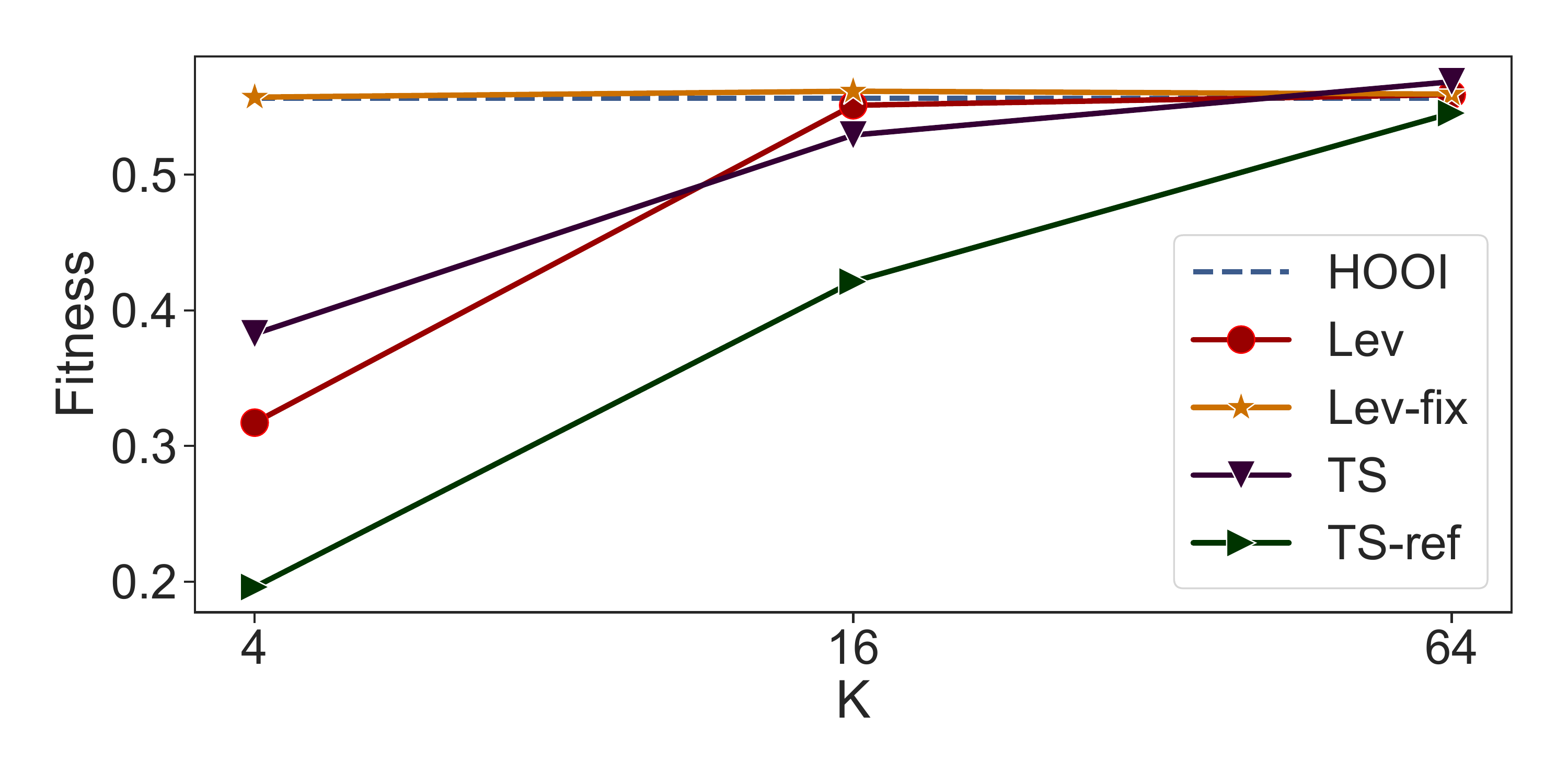}\label{figs-dense8}}

\caption[]{Experimental results for Tucker decomposition of dense tensors. 
For Tensor \ref{tsr:sparse-bias}, $\tsr{T}$ is generated based on \eqref{eq:tsr_dense}.
For all the experiments, we set $R=5$. 
\textbf{(a)(c)} Box plots of the final fitness for each algorithm with different input tensors. Each box is based on 10 experiments with different random seeds. 
\textbf{(b)(d)} Relation between the final fitness and sketch size parameter $K$ for each algorithm with different input tensors. 
}
\label{fig:dense-bias}
\end{figure}

\subsubsection{Experimental Results for Dense Tensors}

We show the detailed fitness-sweeps relation for different dense tensors in \cref{fig:tucker-fitness-sweep}. The reference randomized algorithm suffers from unstable convergence as well as low fitness, while our new randomized ALS scheme, with either leverage score sampling or TensorSketch, converges faster than the reference randomized algorithm and reaches higher accuracy. 

We show the final fitness for each method on different dense tensors in  \cref{fig:dense} and \cref{fig:dense-bias}. For Tensor \ref{tsr:dense}, we test on two different values of the true rank versus decomposition rank ratio, $\alpha$. 
We do not report the results for $\alpha=1$, which corresponds to the exact Tucker decomposition, since all the randomized algorithms can find the exact decomposition easily.
We also test on Tensor \ref{tsr:dense-bias} and Tensor \ref{tsr:sparse-bias}. As can be seen in the figure, our new randomized ALS scheme, with either leverage score sampling or TensorSketch, outperforms the reference randomized algorithm for all the tensors. The relative fitness improvement ranges from $4.5\%$ (\cref{figs-dense5},\ref{figs-dense6}) to $22.0\%$ (\cref{figs-dense3},\ref{figs-dense4}) when $K=16$.

\cref{figs-dense2},\ref{figs-dense4},\ref{figs-dense6},\ref{figs-dense8} show the relation between the final fitness and $K$. As is expected, increasing $K$ can increase the accuracy of the randomized linear least squares solve, thus improving the final fitness.  With our new randomized scheme, the relative final fitness difference between HOOI and the randomized algorithms is less than  $8.5\%$ when $K=16$, indicating the efficacy of our new scheme. 

\cref{figs-dense1},\ref{figs-dense3},\ref{figs-dense5},\ref{figs-dense7} include a comparison between random initialization and the initialization scheme based on RRF detailed in \cref{alg:rrf}. For Tensor \ref{tsr:dense}, both initialization schemes have similar performance. For the deterministic leverage score  sampling on Tensor \ref{tsr:dense-bias} (\cref{figs-dense5}), using RRF-based initialization substantially decreases variability of attained accuracy. For leverage score sampling on Tensor \ref{tsr:sparse-bias} (\cref{figs-dense7}), we observe that the random initialization is not effective, resulting in approximately zero final fitness. 
This is because the random initializations are far from the accurate solutions, and some elements with large amplitudes are not sampled in all the ALS sweeps, consistent with our discussion in \cref{sec:init}. With the RRF-based initialization, the output fitness of the algorithms based on leverage score sampling is close to HOOI. Therefore, our proposed initialization scheme is important for improving the robustness of leverage score sampling.

Although leverage score sampling has a better sketch size upper bound, based on analysis in \cref{sec:sketch-rcls}, the random leverage score sampling scheme performs similar to TensorSketch for the tested dense tensors. For leverage score sampling, \cref{figs-dense6},\ref{figs-dense8} shows that when the sketch size is small ($K=4$), the deterministic leverage score sampling scheme outperforms the random sampling scheme for Tensor \ref{tsr:dense-bias} and Tensor \ref{tsr:sparse-bias}. This means that when the tensor has a strong low-rank signal, the deterministic sampling scheme can be better, consistent with the results in \cite{papailiopoulos2014provable}.

\subsubsection{Experimental Results for Sparse Tensors}
\input{material_arxiv/exp_sparse_tucker}

\subsection{Experiments for CP Decomposition}

\begin{figure}[!ht]   
\centering
\subfloat[$\alpha=1.2$]{\includegraphics[width=0.5\textwidth, keepaspectratio]{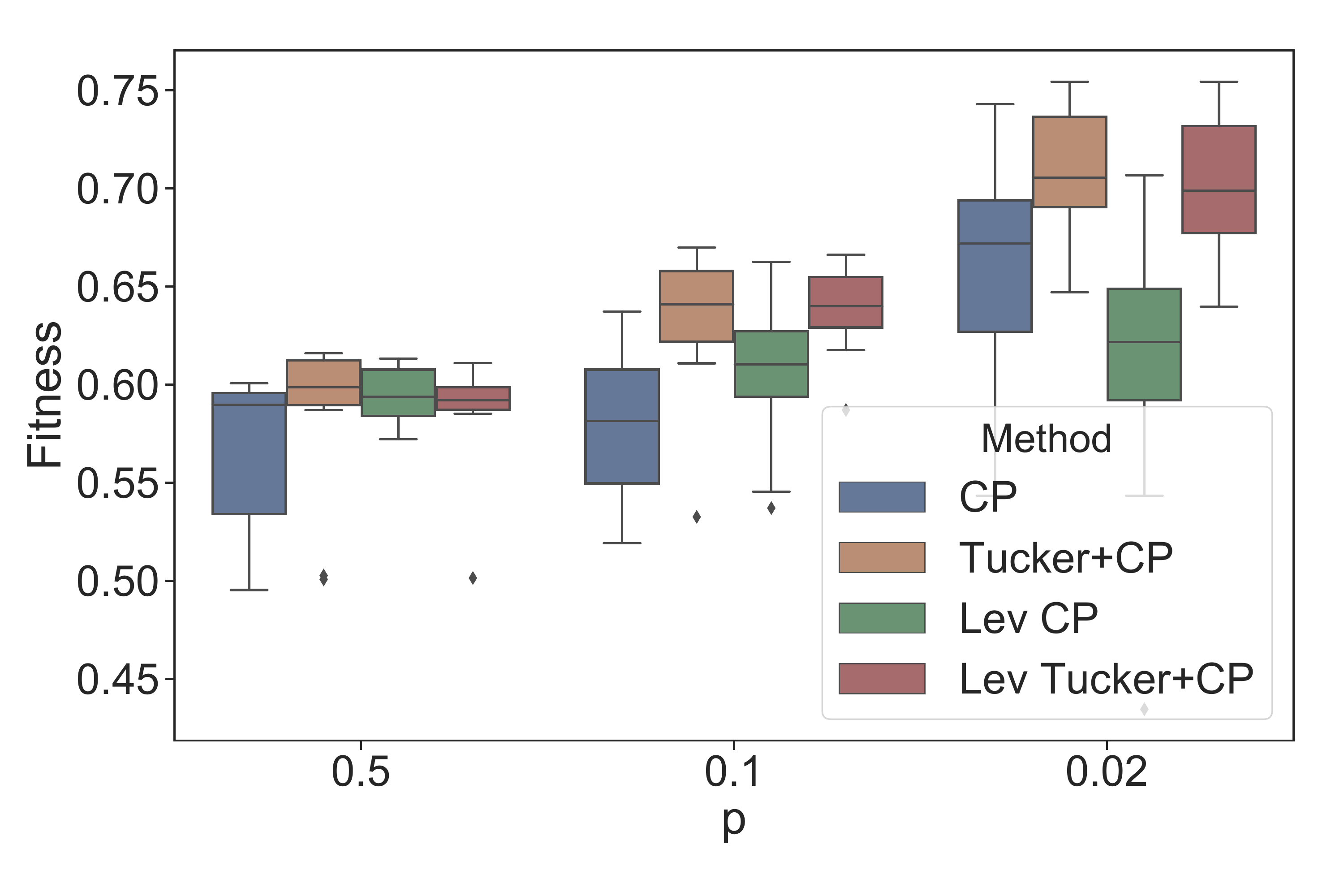}\label{figs-cp1}}
\subfloat[$\alpha=1.6$]{\includegraphics[width=0.5\textwidth, keepaspectratio]{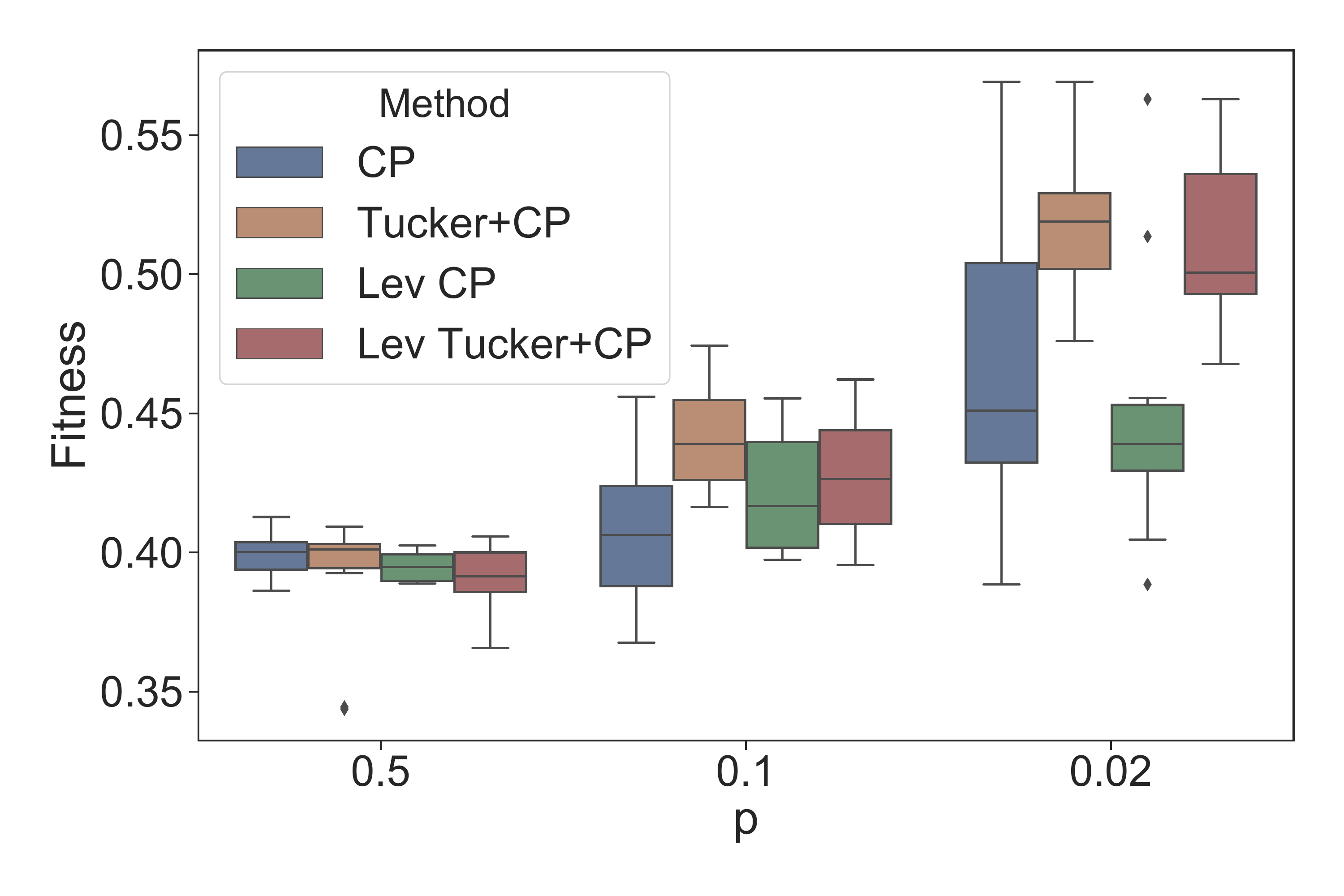}\label{figs-cp2}}

\caption{Relation between final fitness and sparsity parameter $p$ for CP decomposition. For all the experiments, we set $s=2000, R=10$ and $K=16$. In the plots, CP denotes running CP-ALS, Tucker+CP denotes running the Tucker HOOI + CP-ALS algorithm, Lev CP denotes running leverage score sampling based randomized CP-ALS, and Lev Tucker+CP denotes running the leverage score sampling based Tucker-ALS + CP-ALS algorithm.
Each box is based on 10 experiments with different random seeds. 
}
\label{fig:cp}
\end{figure}

\begin{figure}[!ht]   
\centering
\subfloat[$p=0.5$]{\includegraphics[width=0.33\textwidth, keepaspectratio]{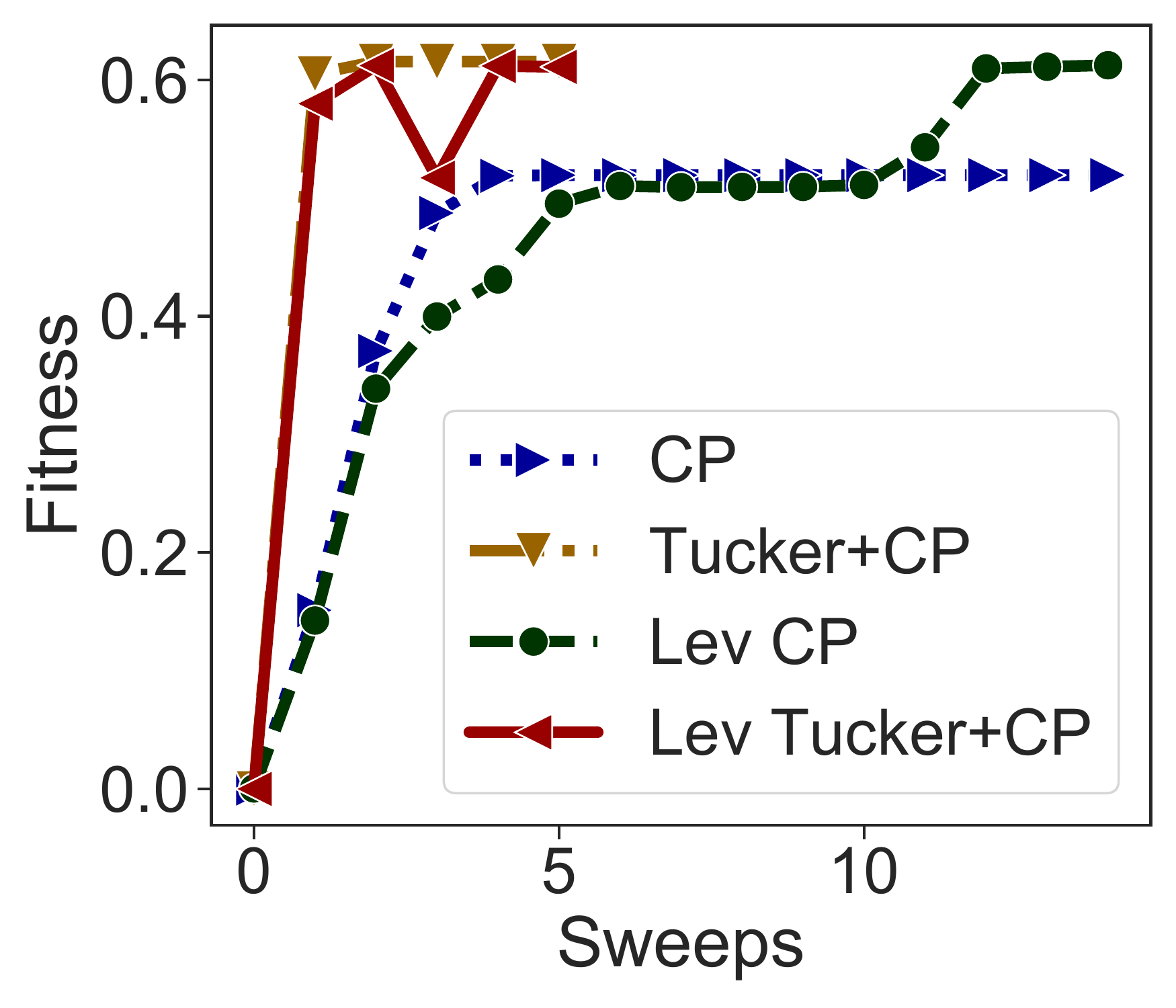}\label{figs-cp3}}
\subfloat[$p=0.1$]{\includegraphics[width=0.33\textwidth, keepaspectratio]{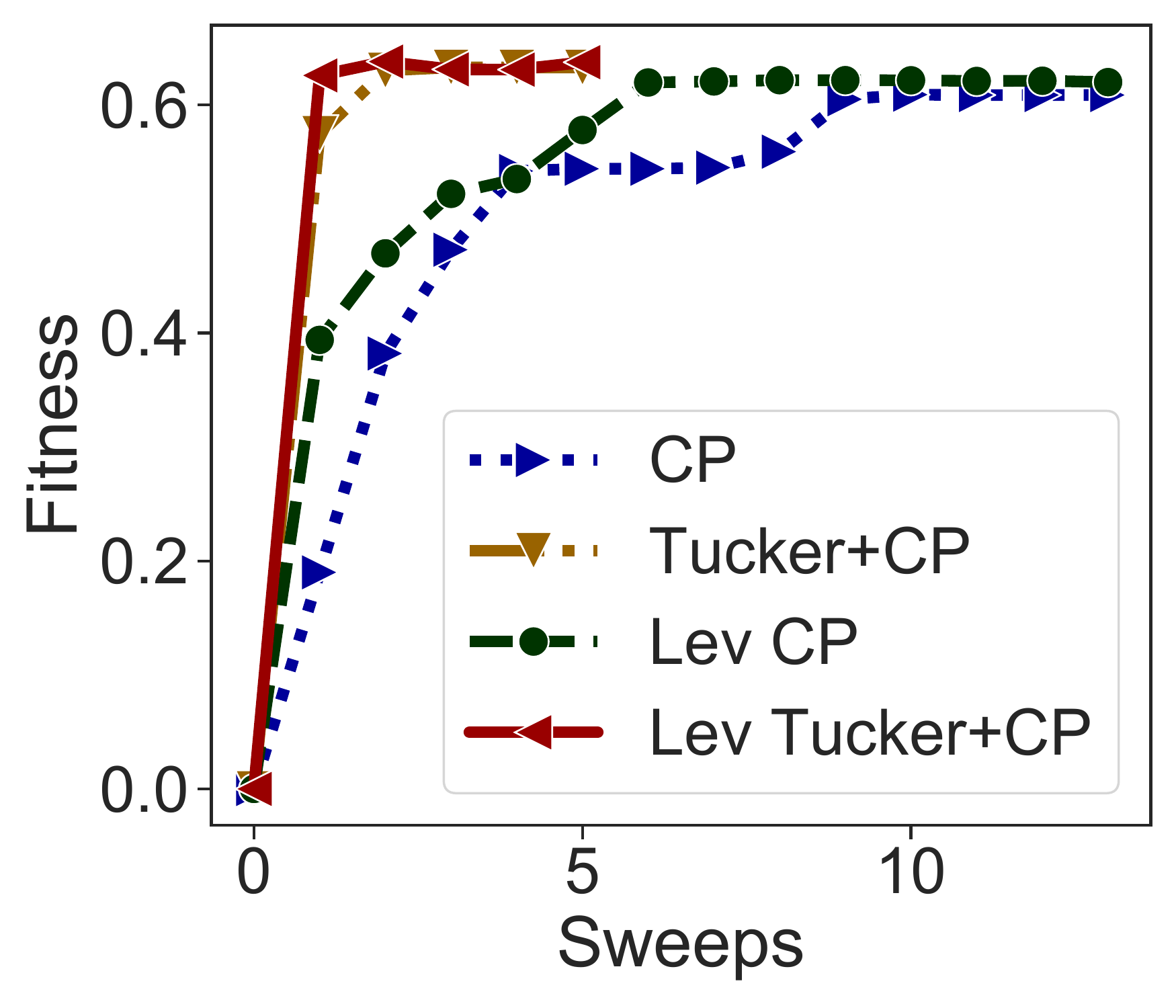}\label{figs-cp4}}
\subfloat[$p=0.02$]{\includegraphics[width=0.33\textwidth, keepaspectratio]{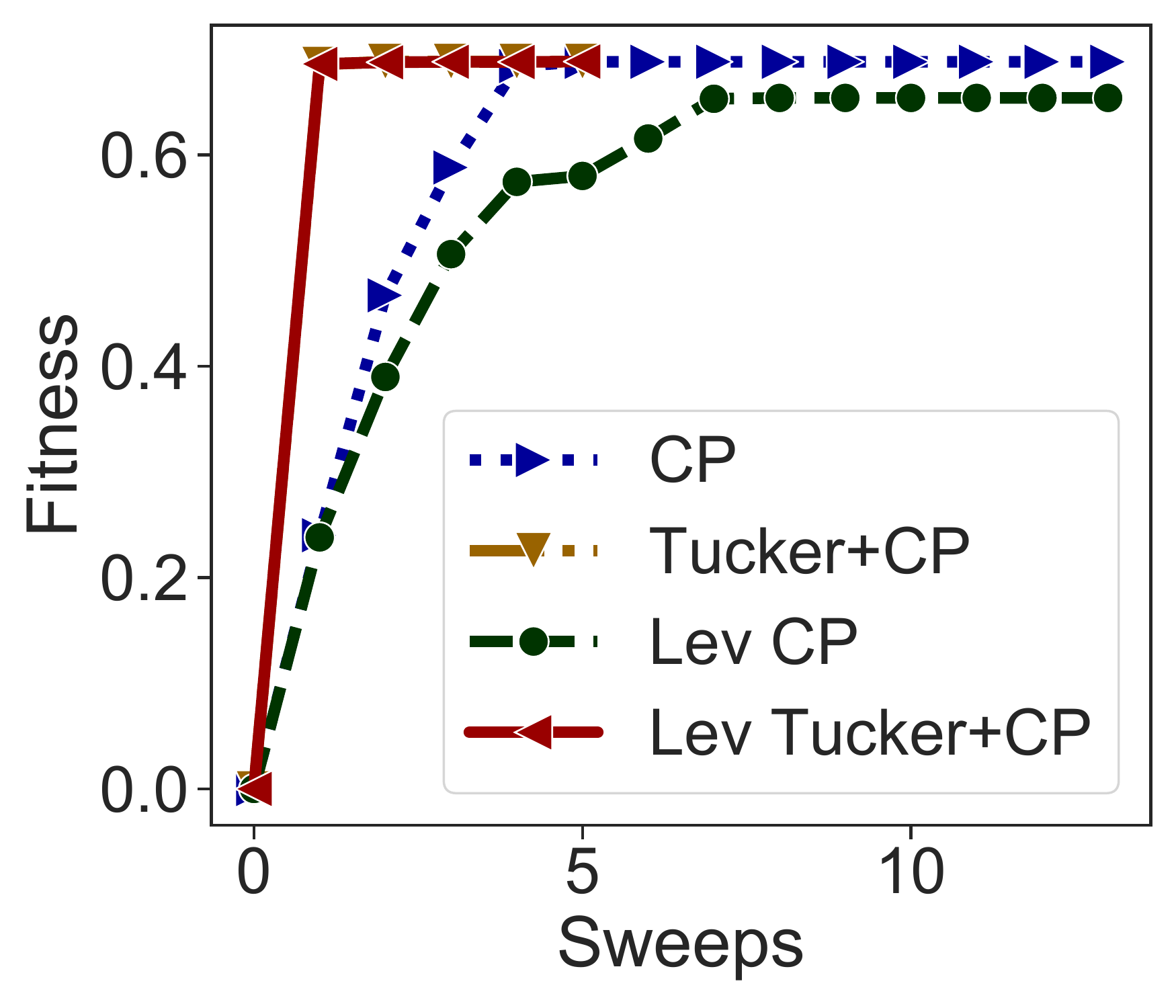}\label{figs-cp5}}

\caption{Detailed fitness-sweeps relation for CP decomposition of three tensors with different parameters. For all the experiments, we set $s=2000, R=10, \alpha=1.2$ and $K=16$. Markers represent the results per sweep. For Tucker + CP algorithms, the fitness shown for $i$th sweep is the output fitness after running $i$ Tucker sweeps along with 25 CP-ALS sweeps on core tensors afterwards.
}
\label{fig:cp-fitness-sweep}
\end{figure}

We show the efficacy of accelerating CP decomposition via performing Tucker compression first. We compare four different algorithms, the standard CP-ALS algorithm, the Tucker HOOI + CP-ALS algorithm, sketched CP-ALS, where the sketching matrix is applied to each linear least squares subproblem \eqref{eq:subproblem}, as well as the sketched Tucker-ALS + CP-ALS algorithm shown in \cref{alg:cp-rand-tucker}. 
Random leverage score sampling is used for sketching, since it has been shown to be efficient for both Tucker (\cref{subsec:exp_tucker}) and CP (reference \cite{larsen2020practical}) decompositions. We use the following synthetic tensor to evaluate these four algorithms,
\begin{equation}
    \tsr{T} = \sum_{i=1}^{R_{\text{true}}} \vcr{a}^{(1)}_i \circ \vcr{a}^{(2)}_i \circ \vcr{a}^{(3)}_i, 
\end{equation}
where each element in $\vcr{a}^{(j)}_i$ is an i.i.d normally distributed random variable $\mathcal{N}(0,1)$ with probability $p$ and is zero otherwise. 
This guarantees that the expected sparsity of $\tsr{T}$ is lower-bounded by $1 - R_{\text{true}}p^3$. 
The ratio $R_{\text{true}}/R$, where $R$ is the decomposition rank, is denoted as $\alpha$. 

For (sketched) Tucker + CP algorithms, we run 5 (sketched) Tucker-ALS sweeps first, and then run the CP-ALS algorithm on the core tensor for 25 sweeps. RRF-based initialization is used for Tucker-ALS, and HOSVD on the core tensor is used to initialize the factor matrices of the small CP decomposition problem. For (sketched) CP-ALS algorithms, we also use the RRF-based initialization and run 30 sweeps afterwards, which is sufficient for CP-ALS to converge based on our experiments. This initialization makes sure that leverage score sampling is effective for sparse tensors. We set the sketch size as $KR^2$ for  both algorithms. The constant factor $K$ reveals the accuracy of each subproblem. 
For the RRF-based initialization, we set the sketch size ($k$ in \cref{alg:rrf}) as $\sqrt{K}R$.

We show the relation between final fitness and the tensor sparsity parameter, $p$, in \cref{fig:cp}. As can be seen, for all the tested tensors, the Tucker + CP algorithms perform similarly, and usually better than directly performing CP decomposition. When the input tensor is sparse ($p=0.1$ and $0.02$), the advantage of the Tucker + CP algorithms is greater.
The sketched Tucker-ALS + CP-ALS scheme has a comparable performance compared to Tucker HOOI + CP-ALS, while requiring less computation.

We show the detailed fitness-sweeps relation in \cref{fig:cp-fitness-sweep}. We observe that for (sketched) CP-ALS, more than 10 sweeps are necessary for the algorithms to converge. On the contrary, less than 5 Tucker-ALS sweeps are needed for the sketched Tucker + CP scheme, making it more efficient. 

In summary, we observe CP decomposition can be more efficiently and accurately calculated based on the sketched Tucker + CP method.

%% file: material_arxiv/exp_sparse_tucker.tex
\begin{figure}[!ht]   
\centering
\subfloat[Tensor \ref{tsr:sparse} with $p=0.5$]{\includegraphics[width=0.50\textwidth, keepaspectratio]{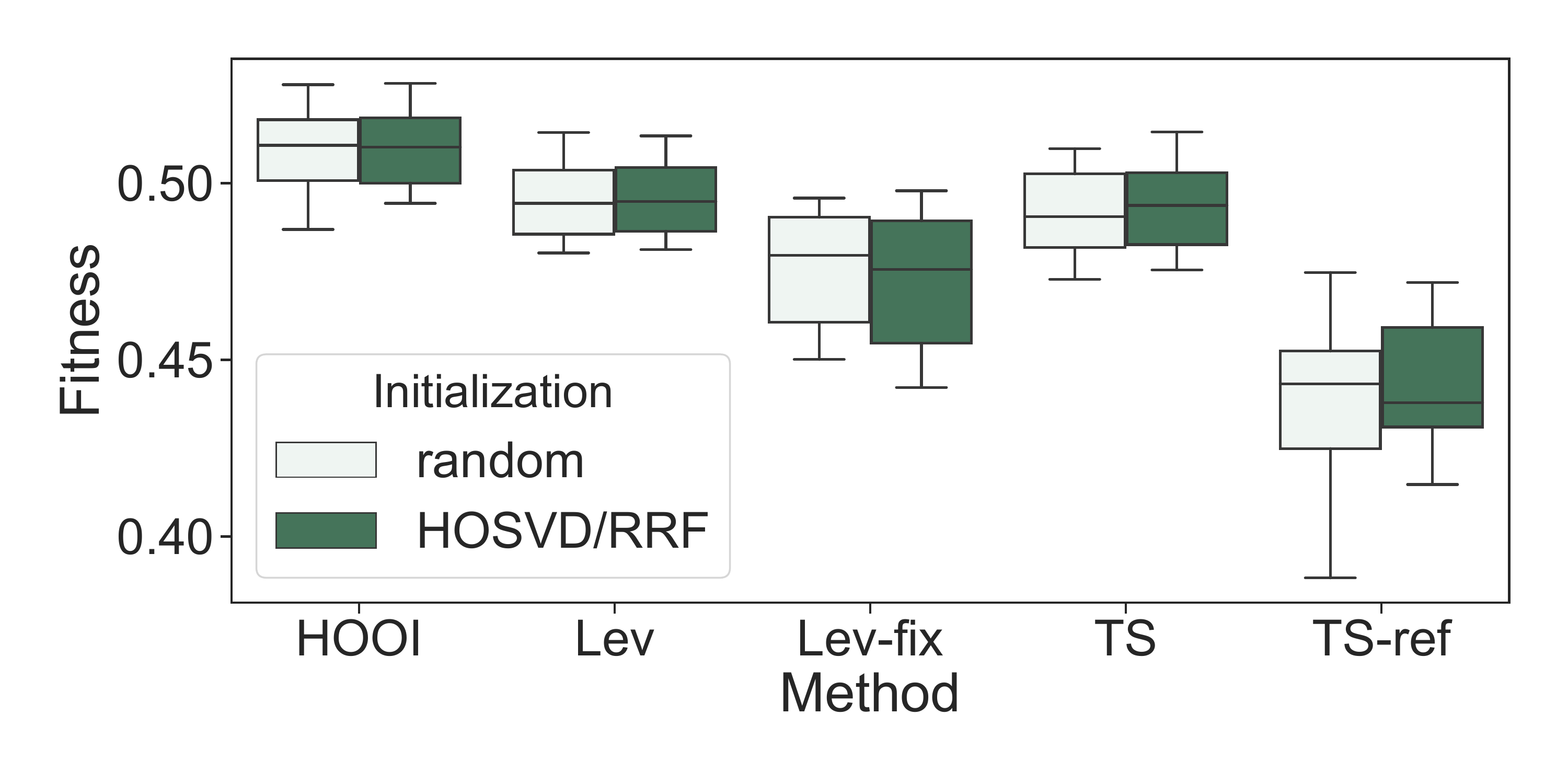}\label{figs-sparse1}}
\subfloat[Tensor \ref{tsr:sparse-bias} with $p=0.5$]{\includegraphics[width=0.50\textwidth, keepaspectratio]{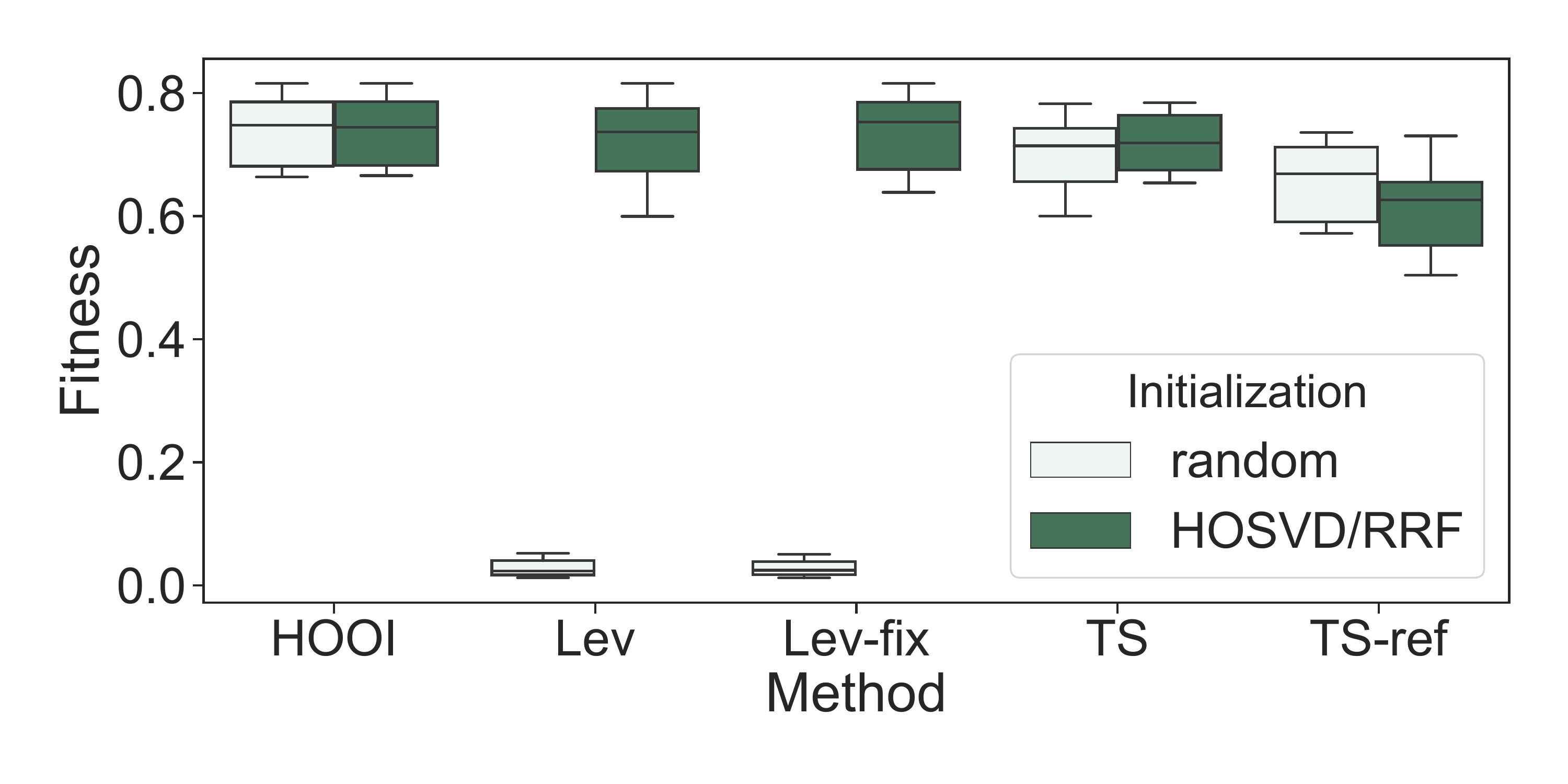}\label{figs-sparse2}}

\subfloat[Tensor \ref{tsr:sparse} with $p=0.1$]{\includegraphics[width=0.50\textwidth, keepaspectratio]{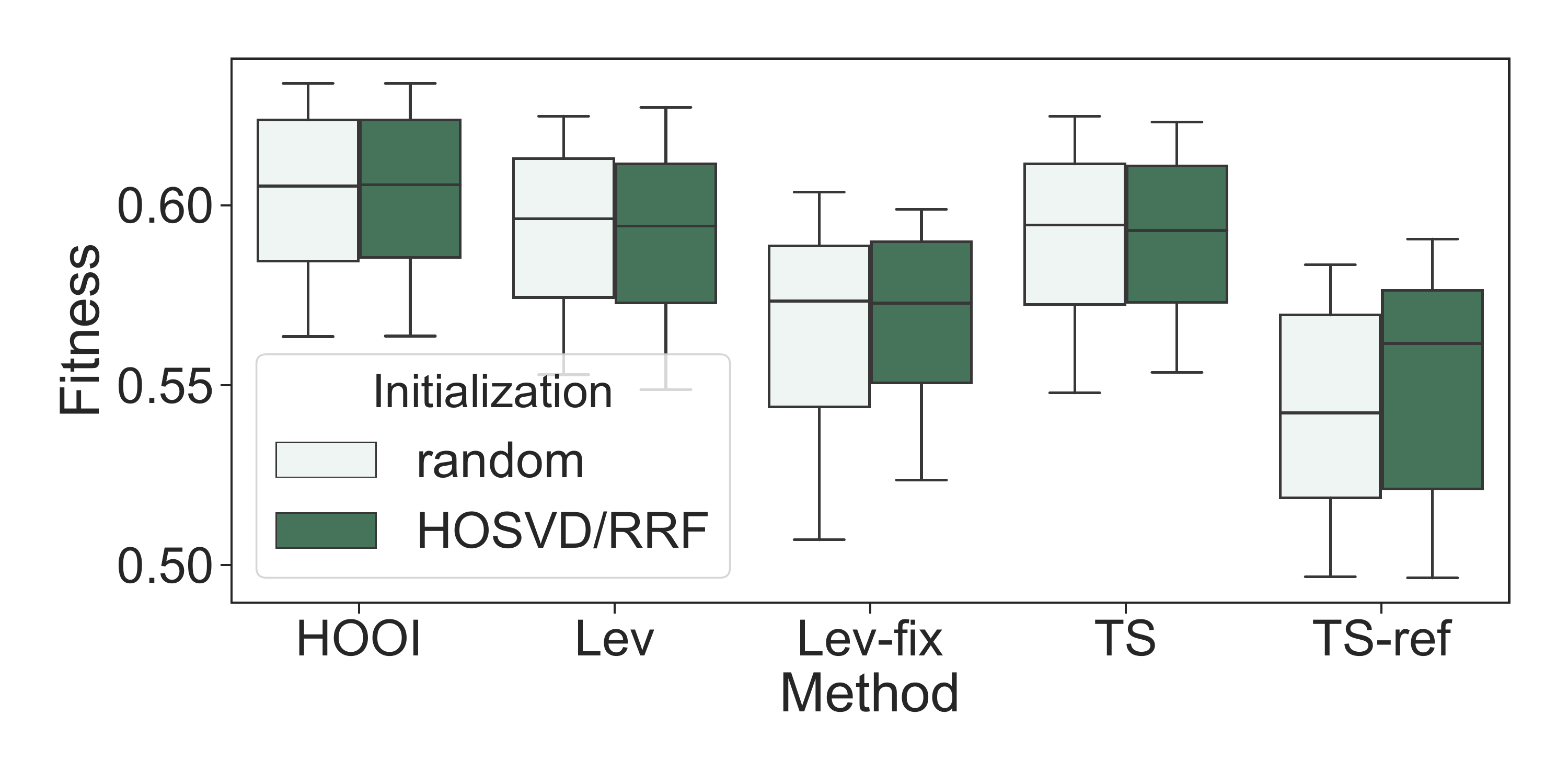}\label{figs-sparse3}}
\subfloat[Tensor \ref{tsr:sparse-bias} with $p=0.1$]{\includegraphics[width=0.50\textwidth, keepaspectratio]{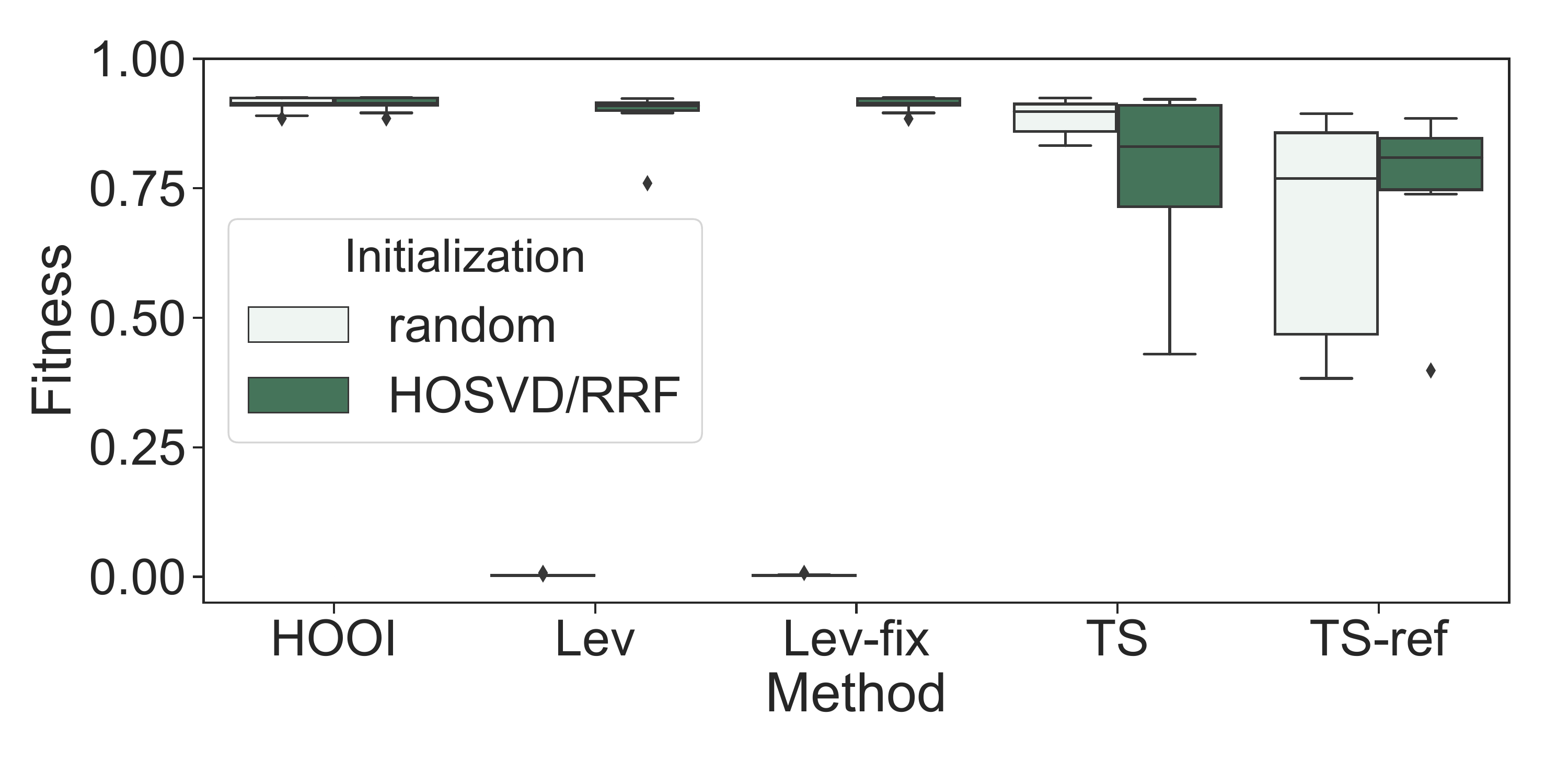}\label{figs-sparse4}}

\subfloat[Tensor \ref{tsr:sparse} with $p=0.02$]{\includegraphics[width=0.50\textwidth, keepaspectratio]{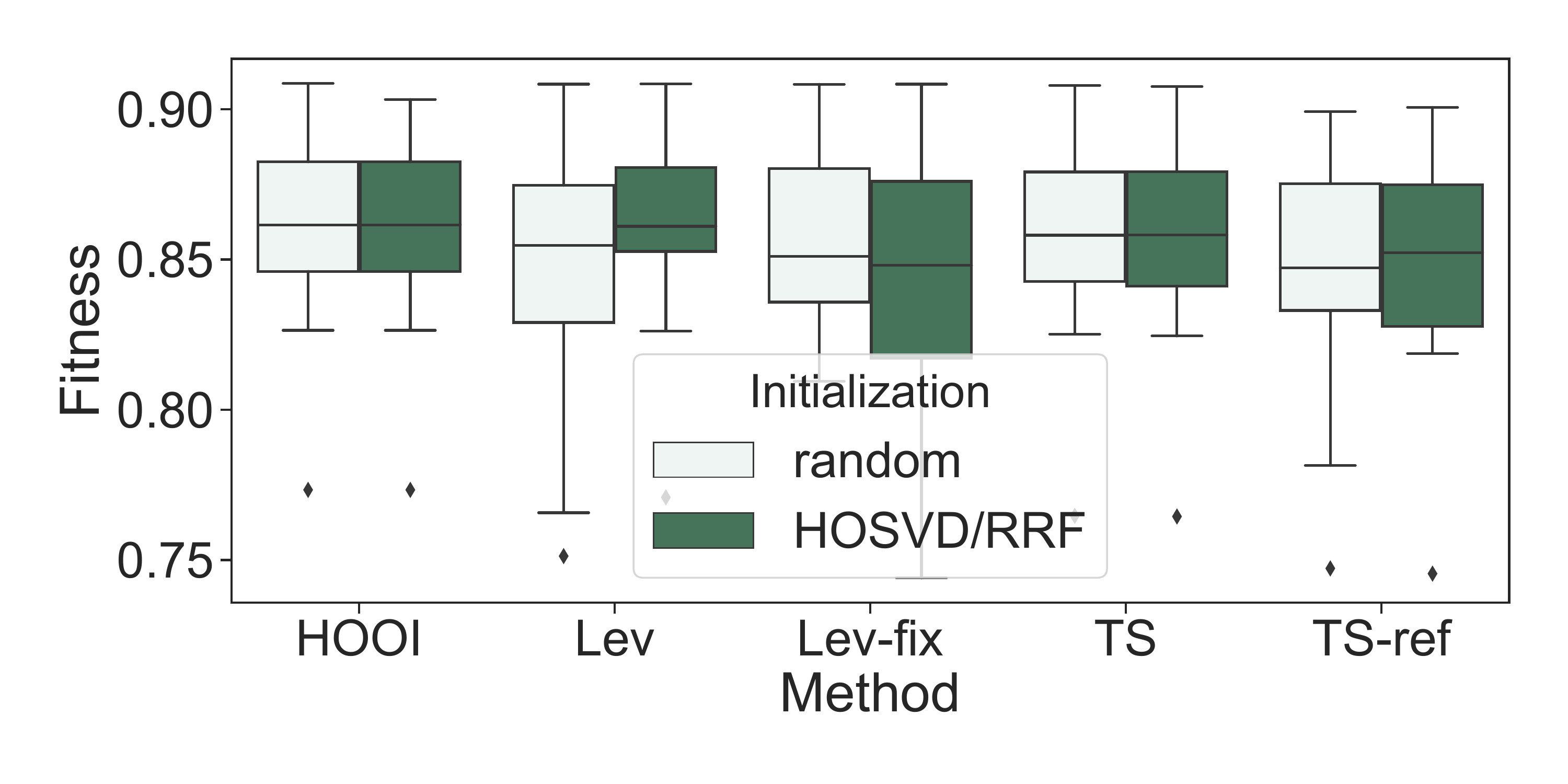}\label{figs-sparse5}}
\subfloat[Tensor \ref{tsr:sparse-bias} with $p=0.02$]{\includegraphics[width=0.50\textwidth, keepaspectratio]{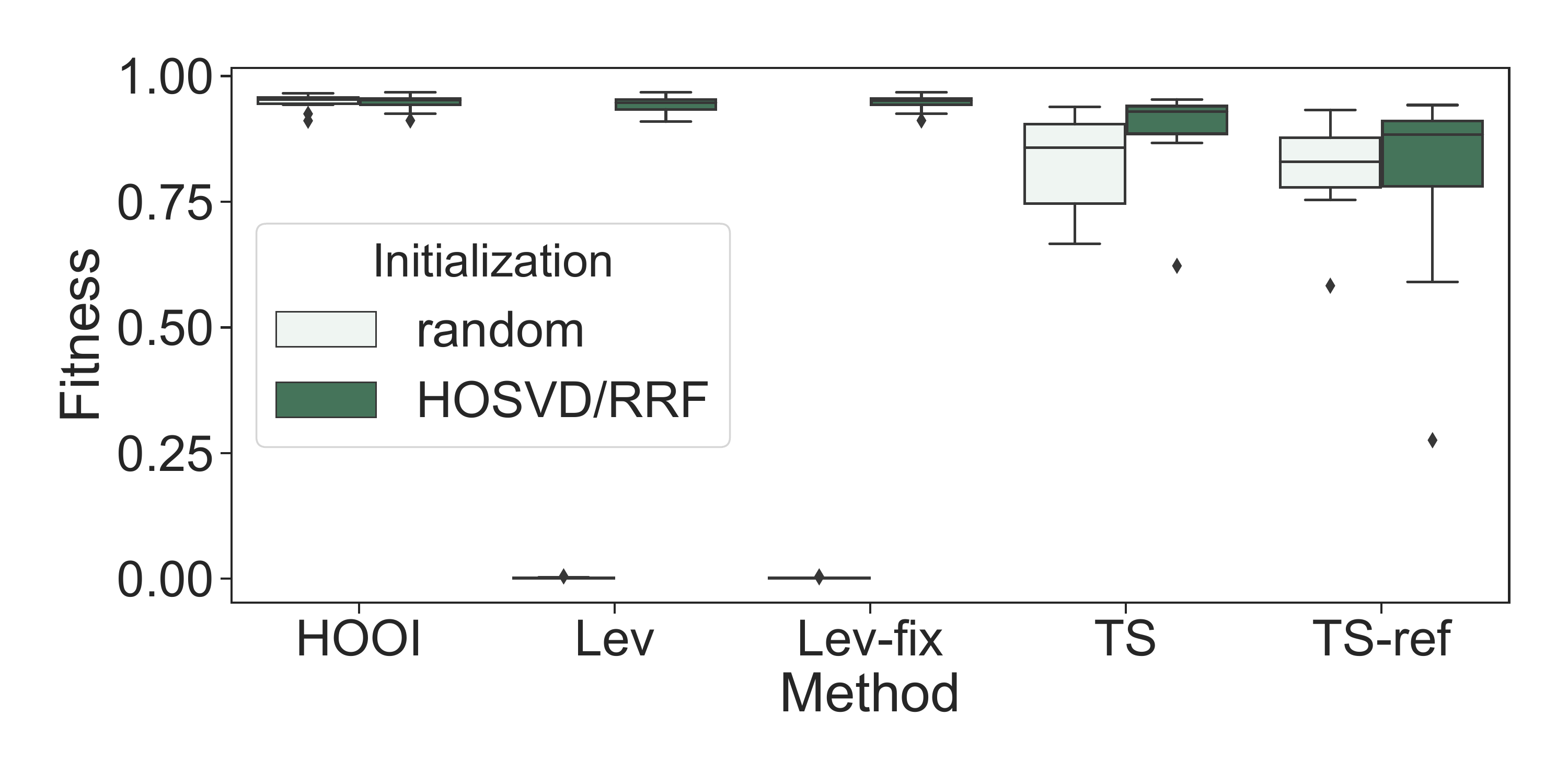}\label{figs-sparse6}}
\caption{Experimental results for Tucker decomposition of sparse tensors. For all the experiments, we set $s=2000, R=10, \alpha=1.2$ and $K=16$. 
\textbf{(a)(c)(e)} Box plots of the final fitness for each algorithm on Tensor~\ref{tsr:sparse} with different sparsity parameter $p$. 
\textbf{(b)(d)(f)} Box plots of the final fitness for each algorithm on Tensor~\ref{tsr:sparse-bias} with different sparsity parameter $p$. Each box is based on 10 experiments with different random seeds. 
}
\label{fig:sparse}
\end{figure}

We show our experimental results for sparse tensors in \cref{fig:sparse}. For both Tensor~\ref{tsr:sparse} and Tensor~\ref{tsr:sparse-bias}, we test on tensors with different sparsity via varying the parameter $p$.  When $p=0.1$ (\cref{figs-sparse3}, \ref{figs-sparse4}), the expected sparsity of the tensor is greater than 0.9. When $p=0.02$ (\cref{figs-sparse5}, \ref{figs-sparse6}), the expected sparsity of the tensor is greater than $0.9998$. 

The results for Tensor \ref{tsr:sparse} are shown in \cref{figs-sparse1},\ref{figs-sparse3},\ref{figs-sparse5}. 
Our new randomized ALS scheme, with either leverage score sampling or TensorSketch, outperforms the reference randomized algorithm with $p=0.1$ and $p=0.5$.
The relative fitness improvement ranges from $3.6\%$ (\cref{figs-sparse3}) to $12.7\%$ (\cref{figs-sparse1}).
The performance of our new scheme is comparable to the reference with $p=0.02$. 
The reason for the reduced improvements
is that these tensors have high decomposition fitness ($0.8\sim 0.9$) and each non-zero element has the same distribution, 
so sophisticated sampling is not needed to achieve high accuracy.
Similar to the case of dense tensors shown in \cref{figs-dense1},\ref{figs-dense3}, we observe similar behavior
for Tensor~\ref{tsr:sparse} with random initialization and RRF-based initialization. 

The results for Tensor \ref{tsr:sparse-bias} are shown in \cref{figs-sparse2},\ref{figs-sparse4},\ref{figs-sparse6}. Our new randomized ALS scheme outperforms the reference randomized algorithm for all the cases. Similar to the case of dense tensors (\cref{figs-dense7}),
for leverage score sampling, the random initialization results in approximately zero final fitness, and the RRF-based initialization can greatly improve the output fitness. Therefore, the RRF-based initialization scheme is important for improving the robustness of leverage score sampling.

On the contrary, TensorSketch based algorithms are not sensitive to the choice of
initialization scheme. Although they perform much better compared to the leverage score sampling with random initialization, the output fitness is still a bit worse than HOOI and can have relatively larger variance (\cref{figs-sparse4},\ref{figs-sparse6}). This means TensorSketch is less effective than leverage score sampling with RRF initialization for this tensor.

In summary, we observe the algorithm combining leverage score sampling, the RRF-based initialization and our new ALS scheme 
achieves the highest accuracy and the most robust performance across test problems among randomized schemes.

%% file: contents/conclusion.tex
In this work, we propose a fast and accurate sketching based ALS algorithm for Tucker decomposition, which consists of a sequence of sketched rank-constrained linear least squares subproblems. 
Theoretical sketch size upper bounds are provided to achieve $\bigO{\epsilon}$-relative residual norm error for each subproblem with two sketching techniques, TensorSketch and leverage score sampling.
For both techniques, our bounds are at most $\bigO{1/\epsilon}$ times the sketch size upper bounds for the unconstrained linear least squares problem. We also propose an initialization scheme based on randomized range finder to improve the accuracy of leverage score sampling based randomized Tucker decomposition of tensors with high coherence.
Experimental results show that this new ALS algorithm is more accurate than the existing sketching based randomized algorithm for Tucker decomposition. 
This Tucker decomposition algorithm also yields an efficient CP decomposition method, where randomized Tucker compression is performed first, and CP decomposition is applied to the Tucker core tensor afterwards. Experimental results show this algorithm not only converges faster, but also yields more accurate CP decompositions.

We leave high-performance implementation of our sketched ALS algorithm as well as testing its performance on large-scale real sparse datasets for future work.
Additionally, although our theoretical analysis shows a much tighter sketch size upper bound for leverage score sampling compared to TensorSketch, their experimental performance under the same sketch size for multiple tensors are similar. Therefore, it will be  of interest to investigate potential improvements to sketch size bounds for TensorSketch.

%% file: contents/rc_least_squares.tex
In this section, we provide detailed proofs for the sketch size upper bounds of both sketched unconstrained and rank-constrained linear least squares problems. 
In \cref{subsec:general_analysis_unconstrained}, we define the $(\gamma, \delta, \epsilon)$-accurate sketching matrix, and show the error bound for sketched unconstrained linear least squares, under the assumption that the sketching matrix is a $(1/2, \delta, \epsilon)$-accurate sketching matrix.
In \cref{subsec:general_analysis}, we show the error bound for sketched rank-constrained linear least squares.
In \cref{subsec:tensorsketch_background} and \cref{subsec:leverage_background}, we finish the proofs by giving the sketch size bounds that are sufficient for the TensorSketch matrix and leverage score sampling matrix to be the $(1/2, \delta, \epsilon)$-accurate sketching matrix, respectively.

\subsection{Error Bound for Sketched Unconstrained Linear Least Squares}\label{subsec:general_analysis_unconstrained}

We define the $(\gamma, \delta, \epsilon)$-accurate sketching matrix in \cref{def:accurate_skeching}. 
In \cref{thm:structure-ls}, we show the relative error bound for the unconstrained linear least squares problem with a $(1/2, \delta, \epsilon)$-accurate sketching matrix. 
By $\mat{Q}_P$ we denote a matrix whose columns form an orthonormal basis for the column space of $\mat{P}$.

\begin{definition}[$(\gamma, \delta, \epsilon)$-accurate Sketching Matrix]\label{def:accurate_skeching} A random matrix $\mat{S}\in\R^{m \times s}$ is a $(\gamma, \delta, \epsilon)$-accurate sketching matrix for $\mat{P}\in\R^{s\times R}$ if the following two conditions hold simultaneously.
\begin{enumerate}[leftmargin=*]
    \item With probability at least $1 - \delta/2$, each singular value $\sigma$ of $\mat{SQ}_P$ satisfies
    \begin{equation}\label{eq:structure1}
        1-\gamma\leq \sigma^2 \leq 1+\gamma.
    \end{equation}
    \item With probability at least $1 - \delta/2$, for any fixed matrix $\mat{B}$, we have
    \begin{equation}\label{eq:structure2}
        \|\mat{Q}_P^T\mat{S}^T \mat{SB} - \mat{Q}_P^T\mat{B}\|_F^2\leq \epsilon^2 \cdot \|\mat{B}\|_F^2.
    \end{equation}
\end{enumerate} 
\end{definition}

\begin{lemma}[Linear Least Squares with $(1/2, \delta, \epsilon)$-accurate Sketching Matrix~\cite{woodruff2014sketching,larsen2020practical,drineas2011faster}]
\label{thm:structure-ls}
Given a full-rank matrix $\mat{P}\in\R^{s\times R}$ with $s\geq R$, and $\mat{B}\in \mathbb{R}^{s\times n}$. 
Let $\mat{S}\in\R^{m\times s}$ be a $(1/2, \delta, \epsilon)$-accurate sketching matrix.
Let $\mat{B}^{\perp} = \mat{P}\mat{X}_{\text{opt}} - \mat{B}$, with $\mat{X}_{\text{opt}} = \arg\min_{\mat{X}} \left\|\mat{PX}-\mat{B}\right\|_F$, and $\mat{\widetilde{X}}_{\text{opt}} = \arg\min_{\mat{X}} \fnrm{\mat{SPX}-\mat{SB}}$.
Then the following approximation holds with probability at least $1-\delta$,
	\begin{equation}
\left\|\mat{P}\mat{\widetilde{X}}_{\text{opt}}- \mat{P}\mat{X}_{\text{opt}}\right\|_F^2	 
\leq \bigO{\epsilon^2} \left\|\mat{B}^{\perp}\right\|_F^2.
	\end{equation}	
\end{lemma}
\begin{proof}
Define the reduced QR decomposition,
$\mat{P} = \mat{Q}_{P}\mat{R}_P$. 
The unconstrained sketched problem can be rewritten as
\begin{align*}
   \min_{\mat{X}} \left\|\mat{SPX}-\mat{SB}\right\|_F 
   &= 
\min_{\mat{X}} \left\|\mat{SPX}-\mat{S}(\mat{P}\mat{X}_{\text{opt}} + \mat{B}^{\perp})\right\|_F \\
&= 
\min_{\mat{X}} \left\|\mat{S}\mat{Q}_P\mat{R}_P(\mat{X}-\mat{X}_{\text{opt}}) -\mat{S}\mat{B}^{\perp}\right\|_F, 
\end{align*}
thus the optimality condition is
\begin{equation}\label{eq:opt_condition}
(\mat{S}\mat{Q}_P)^T\mat{S}\mat{Q}_P\mat{R}_P(\mat{\widetilde{X}}_{\text{opt}}-\mat{X}_{\text{opt}}) = (\mat{S}\mat{Q}_P)^T\mat{S}\mat{B}^{\perp}.   
\end{equation}
Based on \eqref{eq:structure1},\eqref{eq:structure2}, with probability at least $1 - \delta$,
both of the following hold,
\begin{gather}\label{eq:structure1_theorema1}
    \sigma_{\min}^2 (\mat{S}\mat{Q}_P) \geq 1 - \gamma = 1/2,\\
\label{eq:structure2_theorema1}
    \left\| \mat{Q}_P^T \mat{S}^T \mat{S} \mat{B}^{\perp} \right\|_F^2 
    = 
    \left\|\mat{Q}_P^T\mat{S}^T \mat{S}\mat{B}^{\perp} - \mat{Q}_P^T\mat{B}^{\perp}\right\|_F^2 \leq \epsilon^2 \cdot \left\|\mat{B}^{\perp}\right\|_F^2,
\end{gather}
where $\sigma_{\min} (\mat{S}\mat{Q}_P)$ is the singular value of $\mat{S}\mat{Q}_P$ with the smallest magnitude.
Combining \eqref{eq:opt_condition}, \eqref{eq:structure1_theorema1}, and \eqref{eq:structure2_theorema1}, we obtain
\begin{align*}
\left\|\mat{P}\mat{\widetilde{X}}_{\text{opt}}
-\mat{P}\mat{X}_{\text{opt}}\right\|_F^2
&=
\left\|\mat{R}_P\mat{\widetilde{X}}_{\text{opt}}
-\mat{R}_P\mat{X}_{\text{opt}}\right\|_F^2 \\
&\overset{\eqref{eq:structure1_theorema1}}{\leq}
4\left\|
(\mat{S}\mat{Q}_P)^T\mat{S}\mat{Q}_P\mat{R}_P(\mat{\widetilde{X}}_{\text{opt}}-\mat{X}_{\text{opt}})
\right\|_F^2 \\
&\overset{\eqref{eq:opt_condition}}{=}
4\left\|\mat{Q}_P^T \mat{S}^T \mat{S}\mat{B}^{\perp}\right\|_F^2 \\
&\overset{\eqref{eq:structure2_theorema1}}{\leq} 
4\epsilon^2  \cdot \left\|\mat{B}^{\perp}\right\|_F^2 = \bigO{\epsilon^2} \left\|\mat{B}^{\perp}\right\|_F^2.
\end{align*}
\end{proof}

\subsection{Error Bound for Sketched Rank-constrained Linear Least Squares}\label{subsec:general_analysis}
We show in \cref{thm:structure-rcls} that with at least $1-\delta$ probability, the relative residual norm error for the rank-constrained linear least squares with a $(1/2, \delta, \epsilon)$-accurate sketching matrix is  bounded by $\bigO{\epsilon}$.
We first state Mirsky’s Inequality below, which bounds the perturbation of singular values when the input matrix is perturbed. We direct readers to the reference for its proof. This bound will be used in \cref{thm:structure-rcls}.
\begin{lemma}[Mirsky’s Inequality for Perturbation of Singular Values~\cite{mirsky1960symmetric}]\label{thm:perturb-value-multiple}
Let $\mat{A}$ and $\mat{F}$  be arbitrary matrices (of the same size) where $\sigma_1 \geq \cdots \geq \sigma_n$ are the singular values of $\mat{A}$ and $\sigma_1' \geq \cdots \geq \sigma_n'$ are the singular values of $\mat{A} + \mat{F}$. Then 
\begin{equation}\label{eq:mirsky}
  \sum_{i=1}^n(\sigma_i - \sigma_i')^2 \leq \|\mat{F}\|_F^2.  
\end{equation}
\end{lemma}

\begin{theorem}[Rank-constrained Linear Least Squares with $(1/2, \delta, \epsilon)$-accurate Sketching Matrix]
\label{thm:structure-rcls}
Given $\mat{P}\in\R^{s\times R}$ with orthonormal columns (such that $\mat{P} = \mat{Q}_P$), and $\mat{B}\in \mathbb{R}^{s\times n}$.
Let $\mat{S}\in\R^{m\times s}$ be a $(1/2, \delta, \epsilon)$-accurate sketching matrix.
Let $\mat{\widetilde{X}}_{\text{r}}$ be the best rank-$r$ approximation of the solution of the problem
$\min_{\mat{X}} \fnrm{\mat{SPX}-\mat{SB}}$, and let $\mat{X}_{\text{r}} = \arg\min_{\mat{X},\text{rank}(\mat{X})=r} \left\|\mat{PX}-\mat{B}\right\|_F$.
Then the residual norm error bound,
	\begin{equation}
\left\|\mat{P}\mat{\widetilde{X}}_{\text{r}}- \mat{B}\right\|_F^2	 
\leq \left(1+\bigO{\epsilon}\right) \Big\|\mat{P}\mat{X}_{\text{r}}- \mat{B}\Big\|_F^2	,
	\end{equation}	
holds with probability at least $1-\delta$.
\end{theorem}
\begin{proof}
Let $\mathcal{R} = \left\|\mat{P}\mat{X}_{\text{r}}
- \mat{B}\right\|_F$. In addition, let  $\mat{X}_{\text{opt}} = \arg\min_{\mat{X}} \left\|\mat{PX}-\mat{B}\right\|_F$ be the optimum solution of the unconstrained linear least squares problem.
Since the residual in the true solution for each component of the least-squares problem (column of $\mat{B}^{\perp}$) is orthogonal to the error due to low rank approximation,
\begin{align}
    \mathcal{R}^2 = \left\|\mat{P}\mat{X}_{\text{r}}
- \mat{B}\right\|^2_F &= \left\|\mat{P}\mat{X}_{\text{opt}}
- \mat{B}\right\|^2_F + \left\|\mat{P}\mat{X}_{\text{r}}
- \mat{P}\mat{X}_{\text{opt}}\right\|^2_F \nonumber\\
& = \left\|\mat{B}^{\perp}\right\|_F^2 + \left\|\mat{X}_{\text{r}}
- \mat{X}_{\text{opt}}\right\|^2_F.
\end{align}
The last equality holds since $\mat{P}$ has orthonormal columns. 
Let $\mat{\widetilde{X}}_{\text{opt}} = \arg\min_{\mat{X}} \left\|\mat{SPX}-\mat{SB}\right\|_F$ be the optimum solution of the unconstrained sketched problem. We have
\begin{align}
    \left\|\mat{P}\mat{\widetilde{X}}_{\text{r}}- \mat{B}\right\|_F^2	
    &= \left\|\mat{P}\mat{\widetilde{X}}_{\text{r}}- \mat{P}\mat{\widetilde{X}}_{\text{opt}}\right\|_F^2	+ \left\|\mat{P}\mat{\widetilde{X}}_{\text{opt}}- \mat{B}\right\|_F^2
    + 2 \left\langle \mat{P}\mat{\widetilde{X}}_{\text{r}}- \mat{P}\mat{\widetilde{X}}_{\text{opt}},
    \mat{P}\mat{\widetilde{X}}_{\text{opt}}-\mat{B}
    \right\rangle_F \notag\\
    \label{eq:decompose_error}
&= \left\|\mat{\widetilde{X}}_{\text{r}}-\mat{\widetilde{X}}_{\text{opt}}\right\|_F^2	+ 
\left\|\mat{\widetilde{X}}_{\text{opt}}
-\mat{X}_{\text{opt}}\right\|_F^2
+ \left\|\mat{B}^{\perp}\right\|_F^2
    + 2 \left\langle \mat{\widetilde{X}}_{\text{r}}- \mat{\widetilde{X}}_{\text{opt}},
    \mat{\widetilde{X}}_{\text{opt}}-\mat{X}_{\text{opt}}
    \right\rangle_F.
\end{align}
Next we bound the magnitudes of the first, second and the fourth terms.
According to \cref{thm:structure-ls}, with probability at least $1-\delta$, the second term in \eqref{eq:decompose_error} can be bounded as
\begin{equation}\label{eq:fullrank_error}
  \left\|\mat{\widetilde{X}}_{\text{opt}}
-\mat{X}_{\text{opt}}\right\|_F^2
=
\left\|\mat{P}\mat{\widetilde{X}}_{\text{opt}}
-\mat{P}\mat{X}_{\text{opt}}\right\|_F^2
\leq C\epsilon^2 \left\|\mat{B}^{\perp}\right\|_F^2,
\end{equation}
for some constant $C\geq 1$.
Suppose $ \mat{\widetilde{X}}_{\text{opt}}$ has singular values $\widetilde{\sigma}_i = \sigma_i + \delta \sigma_i$ for $i$ in $\{1, \ldots, \min(R,n)\}$, where $\sigma_i$ are the singular values of
$\mat{X}_{\text{opt}}$.
Since $\mat{\widetilde{X}}_{\text{r}}$ is defined to be the best low rank approximation to $\mat{\widetilde{X}}_\text{opt}$, we have
\begin{equation}\label{eq:lowrank_error}
    \left\|\mat{\widetilde{X}}_{\text{r}}- \mat{\widetilde{X}}_{\text{opt}}\right\|_F^2
    = \sum_{i=r+1}^{\min(R,n)}\widetilde{\sigma}_i^2 =  
    \sum_{i=r+1}^{\min(R,n)}(\sigma_i + \delta \sigma_i)^2
    = \sum_{i=r+1}^{\min(R,n)}\left(\sigma_i^2 + \delta \sigma_i^2
    + 2\sigma_i\delta \sigma_i
    \right).
\end{equation}
Since $\mat{P}$ has orthonormal columns, $\mat{X}_{\text{r}}$ is the best rank-$r$ approximation of $\mat{X}_{\text{opt}}$,
\[
\sum_{i=r+1}^{\min(R,n)}\sigma_i^2 = \left\|\mat{X}_{\text{r}}
- \mat{X}_{\text{opt}}\right\|^2_F .
\]
In addition, based on Mirsky's inequality (\cref{thm:perturb-value-multiple}), 
\begin{equation}\label{eq:delta_sigma}
  \sum_{i=r+1}^{\min(R,n)}\delta \sigma_i^2 
\leq \sum_{i=1}^{\min(R,n)}\delta \sigma_i^2 \overset{\eqref{eq:mirsky}}{\leq}  \left\|\mat{\widetilde{X}}_{\text{opt}}
-\mat{X}_{\text{opt}}\right\|_F^2
\overset{\eqref{eq:fullrank_error}}{\leq}  C\epsilon^2 \left\|\mat{B}^{\perp}\right\|_F^2,  
\end{equation}
and
\begin{align*}
  \sum_{i=r+1}^{\min(R,n)}  \left|2\sigma_i\delta \sigma_i\right|
&= \epsilon\sum_{i=r+1}^{\min(R,n)}  \left|2\sigma_i \frac{\delta \sigma_i}{\epsilon}\right| 
 \leq 
\epsilon \sum_{i=r+1}^{\min(R,n)} \left(
\sigma_i^2 + \frac{\delta \sigma_i^2}{\epsilon^2}
\right) \\
 &\overset{\eqref{eq:delta_sigma}}{\leq}  C\epsilon\left(\left\|\mat{X}_{\text{r}}
- \mat{X}_{\text{opt}}\right\|^2_F + \left\|\mat{B}^{\perp}\right\|_F^2
\right) = C\epsilon \mathcal{R}^2,  
\end{align*}
thus
\eqref{eq:lowrank_error} can be bounded as
\begin{align}\label{eq:lowrank_error_bound}
    \left\|\mat{\widetilde{X}}_{\text{r}}- \mat{\widetilde{X}}_{\text{opt}}\right\|_F^2 &\leq 
\left\|\mat{X}_{\text{r}}
- \mat{X}_{\text{opt}}\right\|^2_F + C\epsilon^2 \left\|\mat{B}^{\perp}\right\|_F^2 + C\epsilon \mathcal{R}^2 \nonumber\\
&= \left\|\mat{X}_{\text{r}}
- \mat{X}_{\text{opt}}\right\|^2_F + \bigO{\epsilon}\mathcal{R}^2.
\end{align}
Next we bound the magnitude of the inner product term in \eqref{eq:decompose_error},
\begin{align}
        \left|\left\langle \mat{\widetilde{X}}_{\text{r}}- \mat{\widetilde{X}}_{\text{opt}},
    \mat{\widetilde{X}}_{\text{opt}}-\mat{X}_{\text{opt}}
    \right\rangle_F\right|
    & \leq 
    \left\|\mat{\widetilde{X}}_{\text{r}}-\mat{\widetilde{X}}_{\text{opt}}\right\|_F	\left\|\mat{\widetilde{X}}_{\text{opt}}
    -\mat{X}_{\text{opt}}\right\|_F \nonumber\\
    & \overset{\eqref{eq:fullrank_error}}{\leq}  \sqrt{C}\epsilon \left\|\mat{\widetilde{X}}_{\text{r}}-\mat{\widetilde{X}}_{\text{opt}}\right\|_F\|\mat{B}^{\perp}\|_F \nonumber\\
    &    \leq 
    \sqrt{C}\frac{\epsilon}{2}
    \left(\left\|\mat{\widetilde{X}}_{\text{r}}-\mat{\widetilde{X}}_{\text{opt}}\right\|_F^2 + 
    \left\|\mat{B}^{\perp}\right\|_F^2
    \right) \nonumber\\
    & \overset{\eqref{eq:lowrank_error_bound}}{\leq}     \sqrt{C}\frac{\epsilon}{2}
    \left(
 \left\|\mat{X}_{\text{r}}
- \mat{X}_{\text{opt}}\right\|^2_F + \bigO{\epsilon}\mathcal{R}^2
    + 
    \left\|\mat{B}^{\perp}\right\|_F^2
    \right) \nonumber \\
    &     = \bigO{\epsilon}\mathcal{R}^2.
    \label{eq:cross_term_bound}
\end{align}
Therefore, based on \eqref{eq:decompose_error},\eqref{eq:fullrank_error},\eqref{eq:lowrank_error_bound},\eqref{eq:cross_term_bound}, with probability at least $1-\delta$,
\[
\left\|\mat{P}\mat{\widetilde{X}}_{\text{r}}- \mat{B}\right\|_F^2	 
\leq \left(1+\bigO{\epsilon}\right)\mathcal{R}^2
= \left(1+\bigO{\epsilon}\right) \Big\|\mat{P}\mat{X}_{\text{r}}- \mat{B}\Big\|_F^2.
\]
\end{proof}

%% file: contents/tensorsketch.tex
\subsection{TensorSketch for Unconstrained \& Rank-constrained Least Squares}
\label{subsec:tensorsketch_background}

In this section, we first give the sketch size bound that is sufficient for the TensorSketch matrix to be the $(1/2, \delta, \epsilon)$-accurate sketching matrix in \cref{lem:structure_tensorsketch}. The proof is based on \cref{lam:matmul_tensorsketch} and \cref{lam:singular_tensorsketch}, which follows from results derived in previous work~\cite{avron2016sharper,diao2018sketching}.
Lemma~\ref{lam:matmul_tensorsketch} bounds the sketch size sufficient to reach certain matrix multiplication accuracy, while Lemma~\ref{lam:singular_tensorsketch} bounds the singular values of the matrix obtained from applying TensorSketch to a matrix with orthonormal columns. 
We direct readers to prior work for a detailed proof of \cref{lam:matmul_tensorsketch}, but provide a simple proof of \cref{lam:singular_tensorsketch} by application of \cref{lam:matmul_tensorsketch}.

\begin{lemma}[Approximate Matrix Multiplication with TensorSketch~\cite{avron2016sharper}]\label{lam:matmul_tensorsketch}
Given matrices $\mat{P}\in\R^{s^{N-1}\times R^{N-1}}$ and $\mat{B}\in \mathbb{R}^{s^{N-1}\times n}$.
Let $\mat{S}\in\R^{m \times s^{N-1}}$ be an order $N-1$ TensorSketch matrix.
For $m \geq (2 + 3^{N-1})/(\epsilon^2\delta) $,
the approximation,
\[
\|\mat{P}^T\mat{S}^T \mat{SB} - \mat{P}^T\mat{B}\|_F^2\leq \epsilon^2 \cdot \|\mat{P}\|_F^2\cdot \|\mat{B}\|_F^2,
\]
holds with probability at least $1-\delta$.
\end{lemma}

\begin{lemma}[Singular Value Bound for TensorSketch~\cite{diao2018sketching}]\label{lam:singular_tensorsketch}
Given a full-rank matrix $\mat{P}\in\R^{s^{N-1}\times R^{N-1}}$ with $s>R$, and $\mat{B}\in \mathbb{R}^{s^{N-1}\times n}$.
Let $\mat{S}\in\R^{m \times s^{N-1}}$ be an order $N-1$ TensorSketch matrix.
For $m \geq R^{2(N-1)}(2 + 3^{N-1})/(\gamma^2\delta) $,
each singular value $\sigma$ of $\mat{SQ}_P$ satisfies
\[
        1-\gamma\leq \sigma^2 \leq 1+\gamma
\]
with probability at least $1-\delta$, where  $\mat{Q}_P$ is an orthonormal basis for the column space of $\mat{P}$.
\end{lemma}
\begin{proof}
Since $\mat{Q}_P$ is an orthonormal basis for $\mat{P}$, $\mat{Q}^T_P\mat{Q}_P=\mat{I}$, and $\|\mat{Q}_P\|_F^2 = R^{N-1}$. Based on \cref{lam:matmul_tensorsketch}, for $m \geq R^{2(N-1)}(2 + 3^{N-1})/(\gamma^2\delta)$, with probability at least $1 - \delta$, we have
\[
\left\|\mat{Q}_P^T\mat{S}^T \mat{S}\mat{Q}_P - \mat{Q}_P^T\mat{Q}_P\right\|_F^2
= 
\left\|\mat{Q}_P^T\mat{S}^T \mat{S}\mat{Q}_P - \mat{I}\right\|_F^2
\leq \frac{\gamma^2}{R^{2(N-1)}} \cdot \|\mat{Q}_P\|_F^4 = \gamma^2.
\]
Therefore, 
\[
\left\|\mat{Q}_P^T\mat{S}^T \mat{S}\mat{Q}_P - \mat{I}\right\|_2 \leq 
\left\|\mat{Q}_P^T\mat{S}^T \mat{S}\mat{Q}_P - \mat{I}\right\|_F \leq \gamma,
\]
which means the singular values of $\mat{SQ}_P$ satisfy $1-\gamma\leq \sigma^2 \leq 1+\gamma$.
\end{proof}

The previous two lemmas can be combined to demonstrate that the TensorSketch matrix provides an accurate sketch within our analytical framework.

\begin{lemma}[$(1/2, \delta, \epsilon)$-accurate TensorSketch Matrix]\label{lem:structure_tensorsketch}
Given the sketch size,
\[m= \bigO{(R^{(N-1)} \cdot 3^{N-1})/\delta \cdot (R^{(N-1)}  + 1/\epsilon^2)},\]
an order $N-1$ TensorSketch matrix $\mat{S}\in\R^{m \times s^{N-1}}$ is a $(1/2, \delta, \epsilon)$-accurate sketching matrix 
for any full rank matrix
$\mat{P}\in\R^{s^{N-1}\times R^{N-1}}$.
\end{lemma}
\begin{proof}
Based on \cref{lam:singular_tensorsketch} with $\gamma=1/2$, for 
\[m \geq R^{2(N-1)}(2 + 3^{N-1})/(1/4 \cdot \delta/2) = \bigO{(R^{2(N-1)} \cdot 3^{N-1})/\delta},\] 
\eqref{eq:structure1} in \cref{def:accurate_skeching} will hold.
Based on \cref{lam:matmul_tensorsketch}, for $m\geq R^{N-1}(2 + 3^{N-1})/(\epsilon^2\delta)$, 
\[
\|\mat{Q}_P^T\mat{S}^T \mat{SB} - \mat{Q}_P^T\mat{B}\|_F^2\leq \frac{\epsilon^2}{R^{N-1}} \cdot \|\mat{Q}_P\|_F^2\cdot \|\mat{B}\|_F^2 = \epsilon^2\|\mat{B}\|_F^2,
\]
thus \eqref{eq:structure2} in \cref{def:accurate_skeching} will hold. Therefore, we need
\begin{align*}
 m &= \bigO{(R^{2(N-1)} \cdot 3^{N-1})/\delta + (R^{(N-1)} \cdot 3^{N-1})/(\epsilon^2\delta)} \\
&= \bigO{(R^{(N-1)} \cdot 3^{N-1})/\delta \cdot (R^{(N-1)}  + 1/\epsilon^2)}.   
\end{align*}
\end{proof}

Using \cref{lem:structure_tensorsketch}, we can then easily derive the upper bounds for both unconstrained and rank-constrained linear least squares with TensorSketch.
\begin{theorem}[TensorSketch for Unconstrained Linear Least Squares]
\label{thm:tensorsketch-ucls}
Given a full-rank matrix $\mat{P}\in\R^{s^{N-1}\times R^{N-1}}$ with $s>R$, and $\mat{B}\in \mathbb{R}^{s^{N-1}\times n}$.
Let $\mat{S}\in\R^{m \times s^{N-1}}$ be an order $N-1$ TensorSketch matrix.
Let $\mat{\widetilde{X}}_{\text{opt}}= \arg\min_{\mat{X}} \fnrm{\mat{SPX}-\mat{SB}}$ and $\mat{X}_{\text{opt}} = \arg\min_{\mat{X}} \left\|\mat{PX}-\mat{B}\right\|_F$. 
With 
\begin{equation}\label{eq:size_tensorsketch_ucls}
  m = \bigO{(R^{(N-1)} \cdot 3^{N-1})/\delta \cdot (R^{(N-1)}  + 1/\epsilon)},  
\end{equation}
the approximation,
$
\left\|\mat{A}\mat{\widetilde{X}}_{\text{opt}}- \mat{B}\right\|_F^2	 \leq \left(1+\bigO{\epsilon}\right) \Big\|\mat{A}\mat{X}_{\text{opt}}- \mat{B}\Big\|_F^2,
$
holds with probability at least $1-\delta$.
\end{theorem}
\begin{proof}
Based on \cref{thm:structure-ls}, to prove this theorem, we derive the sketch size $m$ sufficient to make the sketching matrix $(1/2, \delta, \sqrt{\epsilon})$-accurate. According to \cref{lem:structure_tensorsketch}, the sketch size \eqref{eq:size_tensorsketch_ucls} is sufficient for being $(1/2, \delta, \sqrt{\epsilon})$-accurate.
\end{proof}

\begin{proof}[Proof of \cref{thm:tensorsketch-rcls}]
Based on \cref{thm:structure-rcls}, to prove this theorem, we derive the sketch size $m$ sufficient to make the sketching matrix $(1/2, \delta, \epsilon)$-accurate. 
According to \cref{lem:structure_tensorsketch}, the sketch size
 \[m= \bigO{(R^{(N-1)} \cdot 3^{N-1})/\delta \cdot (R^{(N-1)}  + 1/\epsilon^2)}\] is sufficient for being $(1/2, \delta, \epsilon)$-accurate.
\end{proof}

%% file: contents/leverage.tex
\subsection{Leverage Score Sampling for Unconstrained \& Rank-constrained Least Squares}
\label{subsec:leverage_background}

In this section, we first give the sketch size bound that is sufficient for the leverage score sampling matrix to be ab $(1/2, \delta, \epsilon)$-accurate sketching matrix according to \cref{lem:structure_leverage}. 
Using \cref{lem:structure_leverage}, we can then easily derive the upper bounds for both unconstrained and rank-constrained linear least squares with leverage score sampling.
To establish these results, we leverage two lemmas.
Lemma~\ref{lam:matmul_leverage} bounds the sketch size sufficient to reach certain matrix multiplication accuracy, while Lemma~\ref{lam:singular_leverage} bounds the singular values of the sketched matrix obtained from applying leverage score sampling to a matrix with orthonormal columns. 
These first two lemmas follow from prior work, and we direct readers to references for detailed proofs of both lemmas.

\begin{lemma}[Approximate Matrix Multiplication with Leverage Score Sampling~\cite{larsen2020practical}]\label{lam:matmul_leverage}
Given matrices $\mat{P}\in\R^{s^{N-1}\times R^{N-1}}$ consists of orthonormal columns and $\mat{B}\in \mathbb{R}^{s^{N-1}\times n}$.
Let $\mat{S}\in\R^{m \times s^{N-1}}$ be a leverage score sampling matrix for $\mat{P}$.
For $m \geq 1/(\epsilon^2\delta) $,
the approximation,
\[
\|\mat{P}^T\mat{S}^T \mat{SB} - \mat{P}^T\mat{B}\|_F^2\leq \epsilon^2 \cdot \|\mat{P}\|_F^2\cdot \|\mat{B}\|_F^2,
\]
holds with probability at least $1-\delta$.
\end{lemma}

\begin{lemma}[Singular Value Bound for Leverage Score Sampling~\cite{woodruff2014sketching}]\label{lam:singular_leverage}
Given a full-rank matrix $\mat{P}\in\R^{s^{N-1}\times R^{N-1}}$ with $s>R$, and $\mat{B}\in \mathbb{R}^{s^{N-1}\times n}$.
Let $\mat{S}\in\R^{m \times s^{N-1}}$ be a leverage score sampling matrix for $\mat{P}$.
For $m = \bigO{R^{(N-1)}\log(R^{(N-1)}/\delta)/\gamma^2}= \bigOt{R^{(N-1)}/\gamma^2}$,
each singular value $\sigma$ of $\mat{SQ}_P$ satisfies
\[
        1-\gamma\leq \sigma^2 \leq 1+\gamma
\]
with probability at least $1-\delta$, where  $\mat{Q}_P$ is an orthonormal basis for the column space of $\mat{P}$.
\end{lemma}

\begin{lemma}[$(1/2, \delta, \epsilon)$-accurate Leverage Score Sampling Matrix]\label{lem:structure_leverage}
Let $m=  \bigO{R^{N-1}/(\epsilon^2\delta)}$ denote the sketch size,
then the leverage score sampling matrix $\mat{S}\in\R^{m \times s^{N-1}}$ is a $(1/2, \delta, \epsilon)$-accurate sketching matrix for the full-rank matrix $\mat{P}\in\R^{s^{N-1}\times R^{N-1}}$.
\end{lemma}
\begin{proof}
Based on \cref{lam:singular_leverage} with $\gamma=1/2$, for $m = \bigOt{R^{(N-1)}}$, \eqref{eq:structure1} in \cref{def:accurate_skeching} will hold.
Based on \cref{lam:matmul_leverage}, for $m= \bigO{R^{N-1}/(\epsilon^2\delta)}$, 
\[
\|\mat{Q}_P^T\mat{S}^T \mat{SB} - \mat{Q}_P^T\mat{B}\|_F^2\leq \frac{\epsilon^2}{R^{N-1}} \cdot \|\mat{Q}_P\|_F^2\cdot \|\mat{B}\|_F^2 = \epsilon^2\|\mat{B}\|_F^2,
\]
thus \eqref{eq:structure2} in \cref{def:accurate_skeching} will hold. Thus we need
$m= \bigOt{R^{(N-1)}} + \bigO{R^{N-1}/(\epsilon^2\delta)}$$=\bigO{R^{N-1}/(\epsilon^2\delta)}$.
\end{proof}

\begin{theorem}[Leverage Score Sampling for Unconstrained Linear Least Squares]
\label{thm:leverage-ucls}
Given a full-rank matrix $\mat{P}\in\R^{s^{N-1}\times R^{N-1}}$ with $s>R$, and $\mat{B}\in \mathbb{R}^{s^{N-1}\times n}$.
Let $\mat{S}\in\R^{m \times s^{N-1}}$ be a leverage score sampling matrix.
Let $\mat{\widetilde{X}}_{\text{opt}}= \arg\min_{\mat{X}} \fnrm{\mat{SPX}-\mat{SB}}$ and $\mat{X}_{\text{opt}} = \arg\min_{\mat{X}} \left\|\mat{PX}-\mat{B}\right\|_F$. 
With 
\begin{equation}\label{eq:size_leverage_ucls}
  m = \bigO{R^{N-1}/(\epsilon\delta)},  
\end{equation}
the approximation,
$
\left\|\mat{A}\mat{\widetilde{X}}_{\text{opt}}- \mat{B}\right\|_F^2	 \leq \left(1+\bigO{\epsilon}\right) \Big\|\mat{A}\mat{X}_{\text{opt}}- \mat{B}\Big\|_F^2,
$
holds with probability at least $1-\delta$.
\end{theorem}
\begin{proof}
Based on \cref{thm:structure-ls}, to prove this theorem, we derive the sample size $m$ sufficient to make the sketching matrix $(1/2, \delta, \sqrt{\epsilon})$-accurate. According to \cref{lem:structure_leverage}, the sketch size \eqref{eq:size_leverage_ucls} is sufficient for being $(1/2, \delta, \sqrt{\epsilon})$-accurate.
\end{proof}

\begin{proof}[Proof of \cref{thm:leverage-rcls}]
Based on \cref{thm:structure-rcls}, to prove this theorem, we derive the sketch size $m$ sufficient to make the sketching matrix $(1/2, \delta, \epsilon)$-accurate. 
According to \cref{lem:structure_leverage}, the sketch size 
$\bigO{R^{N-1}/(\epsilon^2\delta)}$ is sufficient for being $(1/2, \delta, \epsilon)$-accurate.
\end{proof}

%% file: contents/general_constrain_tensorsketch.tex
In this section, we provide sketch size upper bound of TensorSketch for general constrained linear least squares problems. 
\begin{theorem}[TensorSketch for General Constrained Linear Least Squares]
\label{thm:ts-rcls-general}
Given a full-rank matrix $\mat{P}\in\R^{s^{N-1}\times R^{N-1}}$ with $s>R$, and $\mat{B}\in \mathbb{R}^{s^{N-1}\times n}$.
Let $\mat{S}\in\R^{m \times s^{N-1}}$ be an order $N-1$ TensorSketch matrix.
Let $\mat{\widetilde{X}}_{\text{opt}}=\arg\min_{\mat{X\in\mathcal{C}}} \fnrm{\mat{SPX}-\mat{SB}}$, and let $\mat{X}_{\text{opt}} = \arg\min_{\mat{X\in\mathcal{C}}} \left\|\mat{PX}-\mat{B}\right\|_F$. 
With 
\[m= \bigO{nR^{2(N-1)}\cdot 3^{N-1}/(\epsilon^2\delta)},
\]
the  approximation,
	\begin{equation}\label{eq:bound_general_ts}
\left\|\mat{P}\mat{\widetilde{X}}_{\text{opt}}- \mat{B}\right\|_F^2	 \leq \left(1+\bigO{\epsilon}\right) \Big\|\mat{P}\mat{X}_{\text{opt}}- \mat{B}\Big\|_F^2,	
	\end{equation}	
holds with probability at least $1-\delta$.
\end{theorem}

\begin{proof}
The proof is similar to the analysis performed in \cite{woodruff2014sketching} for other sketching techniques.
Let the $i$th column of $\mat{B}, \mat{X}$ be denoted $\vcr{b}_i, \vcr{x}_i$, respectively. We can express each column in the residual $\mat{PX} - \mat{B}$ as 
\[
\mat{P}\vcr{x}_i -\vcr{b}_i = 
\begin{bmatrix}
\mat{P} & \vcr{b}_i
\end{bmatrix}
\begin{bmatrix}
\vcr{x}_i \\ -1
\end{bmatrix}
:= 
\widetilde{\mat{P}}^{(i)}\vcr{y}_i.
\]
Based on \cref{lam:singular_tensorsketch}, let $m \geq n(R^{(N-1)}+1)^2(2 + 3^{N-1})/(\epsilon^2\delta)$, we have with probability at least $1 - \delta/n$ that for some $i\in[n]$, each singular value $\sigma$ of $\mat{SQ}_{\widetilde{P}^{(i)}}$ satisfies
\[
        1-\epsilon\leq \sigma^2 \leq 1+\epsilon.
\]
This means for any $\vcr{y}_i \in \R^{R^{N-1} + 1}$, we have 
\begin{equation}\label{eq:bound-ts}
(1-\epsilon)\left\|\widetilde{\mat{P}}^{(i)}\vcr{y}_i\right\|_2^2\leq\left\|\mat{S}\widetilde{\mat{P}}^{(i)}\vcr{y}_i\right\|_2^2 \leq  (1+\epsilon)\left\|\widetilde{\mat{P}}^{(i)}\vcr{y}_i\right\|_2^2.
\end{equation}
Using the union bound, \eqref{eq:bound-ts} implies that with probability at least $1 - \delta$,
\[
(1-\epsilon)\left\|\mat{P}\mat{\widetilde{X}}_{\text{opt}} - \mat{B}\right\|_F \leq 
\left\|\mat{S}\mat{P}\mat{\widetilde{X}}_{\text{opt}} - \mat{S}\mat{B}\right\|_F \quad \text{and} \quad 
\left\|\mat{S}\mat{P}\mat{X}_{\text{opt}} - \mat{S}\mat{B}\right\|_F
\leq (1+\epsilon)\left\|\mat{P}\mat{X}_{\text{opt}} - \mat{B}\right\|_F.
\]
Therefore, we have  
\begin{align*}
\left\|\mat{P}\mat{\widetilde{X}}_{\text{opt}} - \mat{B}\right\|_F &\leq 
\frac{1}{1-\epsilon}\left\|\mat{S}\mat{P}\mat{\widetilde{X}}_{\text{opt}} - \mat{S}\mat{B}\right\|_F \leq \frac{1}{1-\epsilon}
\left\|\mat{S}\mat{P}\mat{X}_{\text{opt}} - \mat{S}\mat{B}\right\|_F \\
&\leq 
\frac{1+\epsilon}{1-\epsilon}\left\|\mat{P}\mat{X}_{\text{opt}} - \mat{B}\right\|_F
=
(1+\bigO{\epsilon})\left\|\mat{P}\mat{X}_{\text{opt}} - \mat{B}\right\|_F.
\end{align*}
Therefore, $m= \bigO{nR^{2(N-1)} \cdot 3^{N-1}/(\epsilon^2\delta)}$ is sufficient for the approximation in \eqref{eq:bound_general_ts}.
\end{proof}

%% file: main.bbl
\begin{thebibliography}{10}

\bibitem{aggour2020adaptive}
K.~S. Aggour, A.~Gittens, and B.~Yener.
\newblock Adaptive sketching for fast and convergent canonical polyadic
  decomposition.
\newblock In {\em International Conference on Machine Learning. PMLR}, 2020.

\bibitem{ahmadi2020randomized}
S.~Ahmadi-Asl, S.~Abukhovich, M.~G. Asante-Mensah, A.~Cichocki, A.~H. Phan,
  T.~Tanaka, and I.~Oseledets.
\newblock Randomized algorithms for computation of {T}ucker decomposition and
  higher order {SVD} ({HOSVD}).
\newblock {\em IEEE Access}, 9:28684--28706, 2021.

\bibitem{anandkumar2014tensor}
A.~Anandkumar, R.~Ge, D.~Hsu, S.~M. Kakade, and M.~Telgarsky.
\newblock Tensor decompositions for learning latent variable models.
\newblock {\em Journal of Machine Learning Research}, 15:2773--2832, 2014.

\bibitem{andersson1998improving}
C.~A. Andersson and R.~Bro.
\newblock Improving the speed of multi-way algorithms: Part {I}. {Tucker3}.
\newblock {\em Chemometrics and intelligent laboratory systems},
  42(1-2):93--103, 1998.

\bibitem{avron2016sharper}
H.~Avron, K.~L. Clarkson, and D.~P. Woodruff.
\newblock Sharper bounds for regularized data fitting.
\newblock {\em arXiv preprint arXiv:1611.03225}, 2016.

\bibitem{battaglino2018practical}
C.~Battaglino, G.~Ballard, and T.~G. Kolda.
\newblock A practical randomized {CP} tensor decomposition.
\newblock {\em SIAM Journal on Matrix Analysis and Applications},
  39(2):876--901, 2018.

\bibitem{Boutsidis2015CommunicationoptimalDP}
C.~Boutsidis and D.~Woodruff.
\newblock Communication-optimal distributed principal component analysis in the
  column-partition model.
\newblock {\em arXiv preprint arXiv:1504.06729}, 2015.

\bibitem{bro1998improving}
R.~Bro and C.~A. Andersson.
\newblock Improving the speed of multiway algorithms: Part {II}: Compression.
\newblock {\em Chemometrics and intelligent laboratory systems},
  42(1-2):105--113, 1998.

\bibitem{candes2009exact}
E.~J. Cand{\`e}s and B.~Recht.
\newblock Exact matrix completion via convex optimization.
\newblock {\em Foundations of Computational mathematics}, 9(6):717, 2009.

\bibitem{carroll1980candelinc}
J.~D. Carroll, S.~Pruzansky, and J.~B. Kruskal.
\newblock {CANDELINC}: A general approach to multidimensional analysis of
  many-way arrays with linear constraints on parameters.
\newblock {\em Psychometrika}, 45(1):3--24, 1980.

\bibitem{che2019randomized}
M.~Che and Y.~Wei.
\newblock Randomized algorithms for the approximations of {T}ucker and the
  tensor train decompositions.
\newblock {\em Advances in Computational Mathematics}, 45(1):395--428, 2019.

\bibitem{che2021randomized}
M.~Che, Y.~Wei, and H.~Yan.
\newblock Randomized algorithms for the low multilinear rank approximations of
  tensors.
\newblock {\em Journal of Computational and Applied Mathematics}, page 113380,
  2021.

\bibitem{cheng2016spals}
D.~Cheng, R.~Peng, Y.~Liu, and I.~Perros.
\newblock {SPALS}: Fast alternating least squares via implicit leverage scores
  sampling.
\newblock {\em Advances in neural information processing systems}, 29:721--729,
  2016.

\bibitem{de2000multilinear}
L.~De~Lathauwer, B.~De~Moor, and J.~Vandewalle.
\newblock A multilinear singular value decomposition.
\newblock {\em SIAM journal on Matrix Analysis and Applications},
  21(4):1253--1278, 2000.

\bibitem{de2000best}
L.~De~Lathauwer, B.~De~Moor, and J.~Vandewalle.
\newblock On the best rank-1 and rank-(r1, r2,..., rn) approximation of
  higher-order tensors.
\newblock {\em SIAM journal on Matrix Analysis and Applications},
  21(4):1324--1342, 2000.

\bibitem{diao2018sketching}
H.~Diao, Z.~Song, W.~Sun, and D.~Woodruff.
\newblock Sketching for {K}ronecker product regression and p-splines.
\newblock In {\em International Conference on Artificial Intelligence and
  Statistics}, pages 1299--1308. PMLR, 2018.

\bibitem{drineas2012fast}
P.~Drineas, M.~Magdon-Ismail, M.~W. Mahoney, and D.~P. Woodruff.
\newblock Fast approximation of matrix coherence and statistical leverage.
\newblock {\em The Journal of Machine Learning Research}, 13(1):3475--3506,
  2012.

\bibitem{drineas2011faster}
P.~Drineas, M.~W. Mahoney, S.~Muthukrishnan, and T.~Sarl{\'o}s.
\newblock Faster least squares approximation.
\newblock {\em Numerische mathematik}, 117(2):219--249, 2011.

\bibitem{erichson2020randomized}
N.~B. Erichson, K.~Manohar, S.~L. Brunton, and J.~N. Kutz.
\newblock Randomized {CP} tensor decomposition.
\newblock {\em Machine Learning: Science and Technology}, 1(2):025012, 2020.

\bibitem{gu1996efficient}
M.~Gu and S.~C. Eisenstat.
\newblock Efficient algorithms for computing a strong rank-revealing {QR}
  factorization.
\newblock {\em SIAM Journal on Scientific Computing}, 17(4):848--869, 1996.

\bibitem{halko2011finding}
N.~Halko, P.-G. Martinsson, and J.~A. Tropp.
\newblock Finding structure with randomness: Probabilistic algorithms for
  constructing approximate matrix decompositions.
\newblock {\em SIAM review}, 53(2):217--288, 2011.

\bibitem{harshman1970foundations}
R.~A. Harshman.
\newblock Foundations of the {PARAFAC} procedure: models and conditions for an
  explanatory multimodal factor analysis.
\newblock 1970.

\bibitem{hitchcock1927expression}
F.~L. Hitchcock.
\newblock The expression of a tensor or a polyadic as a sum of products.
\newblock {\em Studies in Applied Mathematics}, 6(1-4):164--189, 1927.

\bibitem{hohenstein2012tensor}
E.~G. Hohenstein, R.~M. Parrish, and T.~J. Mart{\'\i}nez.
\newblock Tensor hypercontraction density fitting. {I}. {Q}uartic scaling
  second-and third-order {M}{\o}ller-{P}lesset perturbation theory.
\newblock {\em The Journal of chemical physics}, 137(4):044103, 2012.

\bibitem{hummel2017low}
F.~Hummel, T.~Tsatsoulis, and A.~Gr{\"u}neis.
\newblock Low rank factorization of the {C}oulomb integrals for periodic
  coupled cluster theory.
\newblock {\em The Journal of chemical physics}, 146(12):124105, 2017.

\bibitem{jin2019faster}
R.~Jin, T.~G. Kolda, and R.~Ward.
\newblock Faster {J}ohnson-{L}indenstrauss transforms via {K}ronecker products.
\newblock {\em arXiv preprint arXiv:1909.04801}, 2019.

\bibitem{jolliffe1972discarding}
I.~T. Jolliffe.
\newblock Discarding variables in a principal component analysis. {I}:
  Artificial data.
\newblock {\em Journal of the Royal Statistical Society: Series C (Applied
  Statistics)}, 21(2):160--173, 1972.

\bibitem{kaya2016high}
O.~Kaya and B.~U{\c{c}}ar.
\newblock High performance parallel algorithms for the {T}ucker decomposition
  of sparse tensors.
\newblock In {\em 2016 45th International Conference on Parallel Processing
  (ICPP)}, pages 103--112. IEEE, 2016.

\bibitem{kaya2018parallel}
O.~Kaya and B.~U{\c{c}}ar.
\newblock Parallel {CANDECOMP/PARAFAC} decomposition of sparse tensors using
  dimension trees.
\newblock {\em SIAM Journal on Scientific Computing}, 40(1):C99--C130, 2018.

\bibitem{kolda2009tensor}
T.~G. Kolda and B.~W. Bader.
\newblock Tensor decompositions and applications.
\newblock {\em SIAM review}, 51(3):455--500, 2009.

\bibitem{larsen2020practical}
B.~W. Larsen and T.~G. Kolda.
\newblock Practical leverage-based sampling for low-rank tensor decomposition.
\newblock {\em arXiv preprint arXiv:2006.16438}, 2020.

\bibitem{li2020sgd_tucker}
H.~Li, Z.~Li, K.~Li, J.~S. Rellermeyer, L.~Chen, and K.~Li.
\newblock {SGD\_T}ucker: A novel stochastic optimization strategy for parallel
  sparse {T}ucker decomposition.
\newblock {\em arXiv preprint arXiv:2012.03550}, 2020.

\bibitem{li2017model}
J.~Li, J.~Choi, I.~Perros, J.~Sun, and R.~Vuduc.
\newblock Model-driven sparse {CP} decomposition for higher-order tensors.
\newblock In {\em 2017 IEEE international parallel and distributed processing
  symposium (IPDPS)}, pages 1048--1057. IEEE, 2017.

\bibitem{ma2018accelerating}
L.~Ma and E.~Solomonik.
\newblock Accelerating alternating least squares for tensor decomposition by
  pairwise perturbation.
\newblock {\em arXiv preprint arXiv:1811.10573}, 2018.

\bibitem{ma2020efficient}
L.~Ma and E.~Solomonik.
\newblock Efficient parallel {CP} decomposition with pairwise perturbation and
  multi-sweep dimension tree.
\newblock {\em arXiv preprint arXiv:2010.12056}, 2020.

\bibitem{mahoney2011randomized}
M.~W. Mahoney.
\newblock Randomized algorithms for matrices and data.
\newblock {\em arXiv preprint arXiv:1104.5557}, 2011.

\bibitem{malik2018low}
O.~A. Malik and S.~Becker.
\newblock Low-rank {T}ucker decomposition of large tensors using
  {T}ensorsketch.
\newblock {\em Advances in neural information processing systems},
  31:10096--10106, 2018.

\bibitem{minster2020randomized}
R.~Minster, A.~K. Saibaba, and M.~E. Kilmer.
\newblock Randomized algorithms for low-rank tensor decompositions in the
  {T}ucker format.
\newblock {\em SIAM Journal on Mathematics of Data Science}, 2(1):189--215,
  2020.

\bibitem{mirsky1960symmetric}
L.~Mirsky.
\newblock Symmetric gauge functions and unitarily invariant norms.
\newblock {\em The quarterly journal of mathematics}, 11(1):50--59, 1960.

\bibitem{mitchell1994slowly}
B.~C. Mitchell and D.~S. Burdick.
\newblock Slowly converging {PARAFAC} sequences: swamps and two-factor
  degeneracies.
\newblock {\em Journal of Chemometrics}, 8(2):155--168, 1994.

\bibitem{nisa2019load}
I.~Nisa, J.~Li, A.~Sukumaran-Rajam, R.~Vuduc, and P.~Sadayappan.
\newblock Load-balanced sparse {MTTKRP} on {GPUs}.
\newblock In {\em 2019 IEEE International Parallel and Distributed Processing
  Symposium (IPDPS)}, pages 123--133. IEEE, 2019.

\bibitem{oh2018scalable}
S.~Oh, N.~Park, S.~Lee, and U.~Kang.
\newblock Scalable {T}ucker factorization for sparse tensors-algorithms and
  discoveries.
\newblock In {\em 2018 IEEE 34th International Conference on Data Engineering
  (ICDE)}, pages 1120--1131. IEEE, 2018.

\bibitem{pagh2013compressed}
R.~Pagh.
\newblock Compressed matrix multiplication.
\newblock {\em ACM Transactions on Computation Theory (TOCT)}, 5(3):1--17,
  2013.

\bibitem{papailiopoulos2014provable}
D.~Papailiopoulos, A.~Kyrillidis, and C.~Boutsidis.
\newblock Provable deterministic leverage score sampling.
\newblock In {\em Proceedings of the 20th ACM SIGKDD international conference
  on Knowledge discovery and data mining}, pages 997--1006, 2014.

\bibitem{pazner2018approximate}
W.~Pazner and P.-O. Persson.
\newblock Approximate tensor-product preconditioners for very high order
  discontinuous {Galerkin} methods.
\newblock {\em Journal of Computational Physics}, 354:344--369, 2018.

\bibitem{pilanci2016iterative}
M.~Pilanci and M.~J. Wainwright.
\newblock Iterative {H}essian sketch: fast and accurate solution approximation
  for constrained least-squares.
\newblock {\em The Journal of Machine Learning Research}, 17(1):1842--1879,
  2016.

\bibitem{sidiropoulos2017tensor}
N.~D. Sidiropoulos, L.~De~Lathauwer, X.~Fu, K.~Huang, E.~E. Papalexakis, and
  C.~Faloutsos.
\newblock Tensor decomposition for signal processing and machine learning.
\newblock {\em IEEE Transactions on Signal Processing}, 65(13):3551--3582.

\bibitem{smith2016medium}
S.~Smith and G.~Karypis.
\newblock A medium-grained algorithm for sparse tensor factorization.
\newblock In {\em 2016 IEEE International Parallel and Distributed Processing
  Symposium (IPDPS)}, pages 902--911. IEEE, 2016.

\bibitem{smith2015splatt}
S.~Smith, N.~Ravindran, N.~D. Sidiropoulos, and G.~Karypis.
\newblock {SPLATT}: Efficient and parallel sparse tensor-matrix multiplication.
\newblock In {\em 2015 IEEE International Parallel and Distributed Processing
  Symposium}, pages 61--70. IEEE, 2015.

\bibitem{song2019relative}
Z.~Song, D.~P. Woodruff, and P.~Zhong.
\newblock Relative error tensor low rank approximation.
\newblock In {\em Proceedings of the Thirtieth Annual ACM-SIAM Symposium on
  Discrete Algorithms}, pages 2772--2789. SIAM, 2019.

\bibitem{sun2020low}
Y.~Sun, Y.~Guo, C.~Luo, J.~Tropp, and M.~Udell.
\newblock Low-rank {T}ucker approximation of a tensor from streaming data.
\newblock {\em SIAM Journal on Mathematics of Data Science}, 2(4):1123--1150,
  2020.

\bibitem{sun2018tensor}
Y.~Sun, Y.~Guo, J.~A. Tropp, and M.~Udell.
\newblock Tensor random projection for low memory dimension reduction.
\newblock In {\em NeurIPS Workshop on Relational Representation Learning},
  2018.

\bibitem{tucker1966some}
L.~R. Tucker.
\newblock Some mathematical notes on three-mode factor analysis.
\newblock {\em Psychometrika}, 31(3):279--311, 1966.

\bibitem{wang2015practical}
S.~Wang.
\newblock A practical guide to randomized matrix computations with {MATLAB}
  implementations.
\newblock {\em arXiv preprint arXiv:1505.07570}, 2015.

\bibitem{woodruff2014sketching}
D.~P. Woodruff.
\newblock Sketching as a tool for numerical linear algebra.
\newblock {\em arXiv preprint arXiv:1411.4357}, 2014.

\bibitem{zhou2014decomposition}
G.~Zhou, A.~Cichocki, and S.~Xie.
\newblock Decomposition of big tensors with low multilinear rank.
\newblock {\em arXiv preprint arXiv:1412.1885}, 2014.

\end{thebibliography}
